\documentclass[11pt]{article}
\usepackage{amsmath}
\usepackage{amsfonts}
\usepackage{amssymb}
\usepackage{mathrsfs}                  
\usepackage{amsthm}
\usepackage{enumitem}
\usepackage{amscd}
\usepackage{xcolor}
\usepackage[hyperfootnotes=false]{hyperref}
\usepackage{mathtools}
\usepackage{geometry}
\usepackage{bbm}
\usepackage{bm}
\usepackage{bbm}
\usepackage[round]{natbib}

\numberwithin{equation}{section}
\theoremstyle{plain}
\newtheorem{theorem}{Theorem}[section]
\newtheorem{proposition}[theorem]{Proposition}
\newtheorem{lemma}[theorem]{Lemma}
\newtheorem{corollary}[theorem]{Corollary}
\newtheorem{assumption}{Assumption}

\newtheorem{assumptionalt}{Assumption}[assumption]
\newenvironment{assumptionp}[1]{
	\renewcommand\theassumptionalt{#1}
	\assumptionalt
}{\endassumptionalt}

\theoremstyle{definition}

\newtheorem{example}[theorem]{Example}

\theoremstyle{remark}
\newtheorem{remark}[theorem]{Remark}

\newcommand{\E}{\mathbb{E}}
\newcommand{\R}{\mathbb{R}}

\newcommand{\N}{\mathbb{N}}

\renewcommand{\P}{\mathbb{P}}

\newcommand{\Fc}{\mathcal{F}}

\begin{document}
\title{Deep neural network expressivity for optimal stopping problems}

\author{Lukas Gonon\thanks{Department of Mathematics, Imperial College London (l.gonon@imperial.ac.uk)}}

\maketitle
\begin{abstract}
This article studies deep neural network expression rates for optimal stopping problems of discrete-time Markov processes on high-dimensional state spaces. A general framework is established in which the value function and continuation value
of an optimal stopping problem can be approximated with error at most $\varepsilon$ by a deep ReLU neural network of size at most $\kappa d^{\mathfrak{q}} \varepsilon^{-\mathfrak{r}}$. The constants $\kappa,\mathfrak{q},\mathfrak{r} \geq 0$ do not depend on the dimension $d$ of the state space or the approximation accuracy $\varepsilon$. This proves that deep neural networks do not suffer from the curse of dimensionality when employed to solve optimal stopping problems. 
The framework covers, for example,  exponential L\'evy models, discrete diffusion processes and their running minima and maxima. 
These results mathematically justify the use of deep neural networks for numerically solving optimal stopping problems and  pricing American options in high dimensions. 
\end{abstract}
%

\section{Introduction}
\label{sec:Intro}
In the past years, neural network-based methods have been used ubiquitously in all areas of science, technology, economics and finance. In particular, such methods have been applied to various problems in mathematical finance such as pricing,  hedging and calibration. We refer, for instance, to the articles \cite{Buehler2018}, \cite{Becker2020}, \cite{Becker2021a}, \cite{Cuchiero2019} and 
to the survey papers \cite{Ruf2020}, \cite{Germain2021}, \cite{Beck2020} for an overview and further references. The striking computational performance of these methods has also raised questions regarding their theoretical foundations. Towards a complete theoretical understanding, there have been recent results in the literature which prove that deep neural networks are able to approximate option prices in various models 
without the curse of dimensionality. For deep neural network expressivity results for option prices and associated PDEs we refer, for instance, to  \cite{EGJS18_787}, \cite{HornungJentzen2018} for European options in Black-Scholes models, to \cite{HutzenthalerJentzenKruse2019}, \cite{Cioica2022} for certain semilinear PDEs, to \cite{Grohs2021} for certain Hamilton-Jacobi-Bellman equations, to \cite{ReisingerZhang2019}, \cite{JentzenSalimova2021}, \cite{Takahashi2021} for diffusion models and game-type options 
and to \cite{GS21} for certain path-dependent options in jump-diffusion models. A few works are also concerned with generalization (\cite{BernerGrohsJentzen2018}) and learning errors (\cite{Gonon2021}).

The goal of this article is to analyse deep neural network expressivity for American option prices and general optimal stopping problems in discrete-time.
An optimal stopping problem consists in selecting a stopping time $\tau$ such that the expected reward $\E[g_d(\tau,X_\tau^d)]$ is maximized. Here $X^d$ is a given stochastic process taking values in $\R^d$ and $g_d(t,x)$ is the reward obtained if the process is stopped at time $t$ at state $x$. 
Optimal stopping problems arise in a wide range of contexts in statistics, operations research, economics and finance. In mathematical finance, arbitrage-free prices of American and Bermudan options are given as solutions to optimal stopping problems.
The solution to an optimal stopping problem can be described by the so-called Snell envelope or, equivalently, by a backward recursion (discrete-time) or a free-boundary PDE (continuous-time) in the case when $X^d$ is a Markov process.

In recent years, a wide range of computational methods have been developed to numerically solve optimal stopping problems also in high-dimensional situations, i.e., when the dimension $d$ of the state space is large. For regression-based algorithms we refer, e.g., to \cite{Tsitsiklis2001Regression}, \cite{LongSchw01}, for duality-based methods we refer, e.g., to \cite{Rogers2002}, \cite{Andersen2004PrimalDual},  \cite{Haugh2004}, \cite{Belomestny2009},  
for stochastic grid methods we refer, e.g., to \cite{Broadie2004}, \cite{Jain2015} and for methods based on approximating the exercise boundary we refer, e.g., to 
\cite{Garcia2003}; see for instance also the overview in \cite{Bouchard2012}. Recently proposed methods include signature-based methods \cite{Bayer21} and regression trees \cite{EchChafiq21}.
Furthermore, various methods based on deep neural network approximations of the value function, the continuation value or the exercise boundary of the optimal stopping problem have been proposed, see, for instance, \cite{Kohler2010}, \cite{Becker2019}, \cite{Becker2020}, \cite{Becker2021a}, \cite{Herrera21}, 
\cite{LapeyreLelong2021}, \cite{Reppen22} and the methods for continuous-time free boundary problems \cite{Sirignano2018}, 
\cite{Wang2021}.
For many of these methods also theoretical convergence results or even convergence rates (cf., e.g., \cite{Clement2002},
\cite{Belomestny2011}, \cite{Bayer2021a}) for a fixed dimension $d$ have been established. 

In this article we are interested in mathematically analysing the high-dimensional situation, i.e., in explicitly controlling the dependence on the dimension $d$. We analyse deep neural network approximations for the value function of optimal stopping problems.
We provide general conditions on the reward functions $g_d$ and the  stochastic processes $X^d$ which ensure that the value function (and the continuation value) of an optimal stopping problem can be approximated by deep ReLU neural networks without the curse of dimensionality, i.e.,  that an approximation error of size at most $\varepsilon$ can be achieved by a deep ReLU neural network of size $\kappa d^{\mathfrak{q}} \varepsilon^{-\mathfrak{r}}$ for constants $\kappa,\mathfrak{q},\mathfrak{r} \geq 0$ which do not depend on the dimension $d$ or the accuracy $\varepsilon$. 
The framework, in particular, provides deep neural network expressivity results for prices of American and Bermudan options. 
Our conditions cover most practically relevant payoffs and many popular models such as exponential L\'evy models and discrete diffusion processes. 
The constants $\kappa,\mathfrak{q},\mathfrak{r}$ are explicit and thus the obtained results yield bounds for the approximation error component in any algorithm for optimal stopping and American option pricing in high-dimensions which is based on approximating the value function or the continuation value by deep neural networks.

The remainder of the paper is organized as follows. In Section~\ref{sec:Prelim} we formulate the optimal stopping problem, recall its solution by dynamic programming and introduce the notation for deep neural networks. In Section~\ref{sec:DNN} we formulate the assumptions and main results. Specifically, in Section~\ref{subsec:basic} we formulate a basic framework, Assumptions~\ref{ass:Markov} and \ref{ass:Payoff}, and prove that the value function can be approximated by deep neural networks without the curse of dimensionality, see Theorem~\ref{thm:DNNapprox}. In Section~\ref{subsec:refined} we then provide more refined assumptions on the considered Markov processes and extend the approximation result to this refined framework, see Theorem~\ref{thm:refined}, which is the main result of the article. In Sections~\ref{subsec:expLevy}, \ref{subsec:discretediffusions} and \ref{subsec:runningMax} we then apply this result to exponential L\'evy models, discrete diffusion processes and show that also barrier options can be covered via the running maximum or minimum of such processes. In order to make the presentation more streamlined, most proofs, in particular the proofs of Theorems~\ref{thm:DNNapprox} and \ref{thm:refined}, are postponed to Section~\ref{sec:Proofs}.

\subsection{Notation}
Throughout, we fix a time horizon $T \in \N$ and a probability space $(\Omega,\Fc,\P)$ on which all random variables and processes are defined. For $d \in \N$, $x \in \R^d$, $A \in \R^{d \times d}$ we denote  by $\|x\|$ the Euclidean norm of $x$ and by $\|A\|_{F}$ the Frobenius norm of $A$.  For $f \colon \R^{d_0} \times \R^{d_1} \to \R^{d_2}$ we let
\[
\mathrm{Lip}(f) 
= 
\sup_{\substack{(x_1,y_1),(x_2,y_2) \in \R^{d_0} \times \R^{d_1}\\ x_1 \neq x_2, y_1 \neq y_2}} 
\frac{\|f(x_1,y_1)-f(x_2,y_2)\|}{\|x_1-x_2\|+\|y_1-y_2\|}
\;.
\]

\section{Preliminaries}
\label{sec:Prelim}

In this section we first formulate the optimal stopping problem and recall its solution in terms of the value function. Then we introduce the required notation for deep neural networks. 
\subsection{The optimal stopping problem}
\label{subsec:OptStopping}
For each $d \in \N$ consider a discrete-time $\R^d$-valued Markov process $X^{d}=(X_t^{d})_{t \in \{0,\ldots,T\}}$ and a function $g_d \colon \{0,\ldots,T\} \times \R^d \to \R$. Assume for each $t \in \{0,\ldots,T\}$ that $\E[|g_d(t,X_t^d)|]< \infty$ and let $\mathbb{F}=(\mathcal{F}_t)_{t \in \{ {0,\ldots,T}\}}$ be the filtration generated by $X^d$. Denote by $\mathcal{T}$ the set of $\mathbb{F}$-stopping times $\tau \colon \Omega \to \{0,\ldots,T\}$ and by $\mathcal{T}_t$ the set of $\tau \in \mathcal{T}$ with $t\leq \tau$.  
For notational simplicity we omit the dependence on $d$ in $\mathbb{F}$, $\mathcal{T}$ and $\mathcal{T}_t$. 

The optimal stopping problem consists in computing
\begin{equation}\label{eq:optimalStoppingProblem}
\sup_{\tau \in \mathcal{T}} \E[g_d(\tau,X^d_\tau)].
\end{equation}
Consider the value function $V_d$ defined by the backward recursion 
$V_d(T,x) = g_d(T,x)$ and  
\begin{equation}
\label{eq:recursionValueFunction} V_d(t,x) = \max(g_d(t,x),\E[V_d(t+1,X_{t+1}^d)|X_t^d=x])
\end{equation}
for $t=T-1,\ldots,0$ and $\P \circ (X_t^d)^{-1}$-a.e.\ $x \in \R^d$.
Then knowledge of $V_d$ also allows to compute a stopping time $\tau^* \in \mathcal{T}$ which is a maximizer in \eqref{eq:optimalStoppingProblem}:  
\[
\tau^* = \min \{t \in \{0,\ldots,T\} \, \colon \, V_d(t,X^d_t) = g_d(t,X^d_t)\}
\] 
satisfies $\E[g_d(\tau^*,X_{\tau^*}^d)] = \sup_{\tau \in \mathcal{T}} \E[g_d(\tau,X^d_\tau)]$.
Indeed, by backward induction and the Markov property we obtain that $V_d(t,X^d_t)=U^d_t$, $\P$-a.s., where $U^d$ is the Snell envelope 
defined by the backward recursion $U_T^d = g_d(T,X_T^d)$ and $U_t^d = \max(g_d(t,X_t^d),\E[U_{t+1}^d|\Fc_t])$ for $t=T-1,\ldots,0$. Then, for instance by \cite[Theorem~6.18]{follmerschied}, for all $t \in \{0,\ldots,T\}$, $\P$-a.s.
\begin{equation}\label{eq:auxEq3}
V_d(t,X^d_t) = U_t^d = \mathrm{ess}\sup_{\tau \in \mathcal{T}_t} \E[g_d(\tau,X_\tau^d) | \Fc_t] = \E[g_d(\tau_{\min}^{(t)},X_{\tau_{\min}^{(t)}}^d) | \Fc_t],
\end{equation}
where $\tau_{\min}^{(t)} = \inf\{s \geq t \, \colon \, U_s^d = g_d(s,X_s^d) \}$. In particular, $\tau_{\min}^{(0)}=\tau^*$ is a maximizer of \eqref{eq:optimalStoppingProblem} and, in the case when $X^d_0$ is constant, $V_d(0,X_0^d)$ is the optimal value in \eqref{eq:optimalStoppingProblem}.

The idea of our approach is as follows: in many situations the function $g_d$ is in fact a neural network or can be approximated well by a deep neural network. Then the recursion \eqref{eq:recursionValueFunction} also yields a recursion for a neural network approximation. This argument will be made precise in the proof of Theorem~\ref{thm:DNNapprox} below.

\begin{remark}
Alternatively, we could also define 
\begin{equation}
\label{eq:valueFunction}
V_d(t,x) = \sup_{\tau \in \mathcal{T}_t} \E[g_d(\tau,X^d_\tau) |X^d_t = x].
\end{equation}
Then under the strong Markov property, for each $\tau \in \mathcal{T}_t$ it holds that $\E[g_d(\tau,X_\tau^d) | \Fc_t] = h_\tau^d(X_t^d)$, $\P$-a.s., where $h_\tau^d(x)=\E[g_d(\tau,X_\tau^d) |X_t^d=x]$. 
The definition of the essential supremum then implies that for each $\tau \in \mathcal{T}_t$ it holds $\P$-a.s.\ that
$h_{\tau_{\min}^{(t)}}^d(X_t^d) \geq h_\tau^d(X_t^d)$. But this implies that for $\P \circ (X_t^d)^{-1}$-a.e. $x \in \R^d$  and all $\tau \in \mathcal{T}_t$ it holds that $h_{\tau_{\min}^{(t)}}^d(x) \geq h_\tau^d(x)$, hence $h_{\tau_{\min}^{(t)}}^d(x) \geq \sup_{\tau \in \mathcal{T}_t} h_\tau^d(x)$ for each such $x$. Combining this and \cite[Theorem~6.18]{follmerschied} yields that $\P$-a.s.,
\begin{equation}\label{eq:auxEq4}
U_t^d =  h_{\tau_{\min}^{(t)}}^d(X_t^d) = \sup_{\tau \in \mathcal{T}_t} h_\tau^d(X_t^d) = V_d(t,X_t^d).
\end{equation}
By definition of the Snell envelope, this then yields the  recursion \eqref{eq:recursionValueFunction} for the value function.
\end{remark}

\begin{remark}
The conditional expectation in \eqref{eq:valueFunction} is defined in terms of the transition kernels $\mu^d_{s,t}$, $0\leq s < t \leq T$ of the Markov process $X^d$
(see \cite[p.143]{Kallenberg2002}). In fact, formally we start with transition kernels $\mu^d$ on $\R^d$ which then allow us to construct a family of probability measures $\P_x$ on the canonical path space $((\R^d)^{T+1},\mathcal{B}((\R^d)^{T+1}))$ such that, under $\P_x$, the coordinate process is a Markov process starting at $x$ and with transition kernels $\mu^d$. We refer to \cite[Theorem~8.4]{Kallenberg2002} or \cite[Chapter~II.4.1]{peskir2006optimal}; see also \cite[Chapter~1]{revuz1984markov}.
\end{remark}

\subsection{Deep neural networks}
In this article we will consider neural networks with the ReLU activation function $\varrho \colon \R \to \R$ given by $\varrho(x)=\max(x,0)$. 
For each $d \in \N$ we also denote by $\varrho \colon \R^d \to \R^d $ the compontentwise application of the ReLU activation function. 
Let $L,d \in \N$, $N_0:=d$, $N_1,\ldots,N_L \in \N$ and $A^{\ell} \in \R^{N_\ell \times N_{\ell -1}}$, $b^\ell \in \R^{N_\ell}$ for $\ell = 1,\ldots,L$. A deep neural network with $L$ layers, $d$-dimensional input, activation function $\varrho$ and parameters $((A^1,b^1),\ldots,(A^L,b^L))$ is the function $\phi \colon \R^d \to \R^{N_L}$ given by 
\begin{equation}\label{eq:DNNdef}
\phi(x) = \mathcal{W}_L \circ (\varrho \circ \mathcal{W}_{L-1}) \circ \cdots \circ (\varrho \circ \mathcal{W}_1)(x), \quad x \in \R^d
\end{equation} 
where $\mathcal{W}_\ell \colon \R^{N_{\ell-1}} \to \R^{N_\ell}$ denotes the (affine) function $\mathcal{W}_\ell(z) = A^\ell z + b^\ell$ for $z \in \R^{N_{\ell-1}}$ and $\ell = 1,\ldots,L$. 
We let 
\[
\mathrm{size}(\phi) = \sum_{\ell=1}^L \sum_{i=1}^{N_\ell} \left( \mathbbm{1}_{\{b^\ell_i \neq 0\}} + \sum_{j=1}^{N_{\ell-1}} \mathbbm{1}_{\{A^\ell_{i,j} \neq 0\}} \right)
\]
denote the total number of non-zero entries in the parameter matrices and vectors of the neural network.  
In most cases the number of layers, the activation function and the parameters of the network are not mentioned explicitly and we simply speak of a deep neural network $\phi \colon \R^d \to \R^{N_L}$. We say that a function $f \colon \R^d \to \R^{m}$ can be realized by a deep neural network if there exists a deep neural network $\phi \colon \R^d \to \R^m$ such that $f(x) = \phi(x)$ for all $x \in \R^d$. 
In the literature a deep neural network is often defined as the collection of parameters $\Phi= ((A^1,b^1),\ldots,(A^L,b^L))$ and $\phi$ in \eqref{eq:DNNdef} is called the realization of $\Phi$, see for instance \cite{PETERSEN2018296}, \cite{Opschoor2020}, \cite{GS20_925}. In order to simplify the notation we do not distinguish between the neural network realization and its parameters here, since the parameters are always (at least implicitly) part of the definition. 
Note that in general a function $f$ may admit several different realizations by deep neural networks, i.e., several different choices of parameters may result in the same realization. However, in the present article this is not an issue, because pathological cases are excluded by bounds on the size of the networks.

\section{DNN Approximations for Optimal Stopping Problems} \label{sec:DNN}

This section contains the main results of the article, the deep neural network approximation results for optimal stopping problems. We start by formulating in Assumption~\ref{ass:Markov} a general Markovian framework. In Assumption~\ref{ass:Payoff} we introduce the hypotheses on the reward functions. We then formulate in Theorem~\ref{thm:DNNapprox} the approximation result for this basic framework. Subsequently, we provide a more refined framework, see Assumption~\ref{ass:Markov2}  below, and prove the main result of the article, Theorem~\ref{thm:refined}. This theorem proves that 
 the value function can be approximated by deep neural networks without the curse of dimensionality. Corollary~\ref{cor:continuation} shows that an analogous approximation result also holds for the continuation value. Subsequently, in Sections~\ref{subsec:expLevy}, \ref{subsec:discretediffusions} and \ref{subsec:runningMax} we specialize the result to the case of exponential L\'evy models, discrete diffusion processes and show that also barrier options can be covered by including the running maximum or minimum.

\subsection{Basic framework}
\label{subsec:basic}
Let $p \geq 0$ be a fixed rate of growth. For instance, in financial applications typically $p=1$. 
We start by formulating in Assumption~\ref{ass:Markov} a collection of hypotheses on the Markov processes $X^d$. These hypotheses will be weakened later on in Assumption~\ref{ass:Markov2}. 

\begin{assumption}\label{ass:Markov}[Assumptions on $X^d$]
	\begin{itemize}
		\item[(i)] 
	For each $d \in \N$, $t \in \{0,\ldots,T-1\}$ there exists a measurable function $f^d_t \colon \R^d \times \R^d \to \R^d$ and a random vector $Y^d_t$ such that 
\begin{equation}\label{eq:MarkovUpdate} X_{t+1}^d = f^d_t(X_t^d,Y^d_t).
\end{equation}
\item[(ii)] 
For each $d \in \N$ the random vectors $X_0^d,Y^d_0,\ldots,Y^d_{T-1}$ are independent and $\E[\|X_0^d\|]<\infty$.
\end{itemize}
Furthermore, there exists constants $c>0$, $q\geq 0, \alpha\geq0$ such that 
\begin{itemize}
	\item[(iii)] for all $\varepsilon \in (0,1]$, $d \in \N$, $t \in \{0,\ldots,T-1\}$
	there exists a neural network $\eta_{\varepsilon,d,t} \colon \R^d \times \R^d \to \R^d$ with 
\begin{align}
\label{eq:fapprox}
\| f^d_{t}(x,y) - \eta_{\varepsilon,d,t}(x,y)\| & \leq \varepsilon c d^{q} (1+\|x\|^p + \|y\|^p), \quad \text{ for all } x,y \in \R^d, 
\\ \label{eq:etasparse} 
\mathrm{size}(\eta_{\varepsilon,d,t}) &\leq c d^{q} \varepsilon^{-\alpha}, 
\\ \label{eq:etaLipschitz}
\mathrm{Lip}(\eta_{\varepsilon,d,t}) & \leq c d^{q},
\end{align}
\item[(iv)] for all $d \in \N$, $t \in \{0,\ldots,T-1\}$ it holds that $\|f_t^d(0,0)\|\leq c d^q$ and $\E[\|Y_{t}^d\|^{2\max{(2,p)}}] \leq c d^q$.
	\end{itemize}
\end{assumption}
Assumption~\ref{ass:Markov}(i) requires a recursive updating of the Markov processes $X^d$ according to update functions $f^d_t$ and noise processes $Y^d$. Assumption~\ref{ass:Markov}(ii) asks that the noise random variables and the initial condition are independent. Assumption~\ref{ass:Markov}(iii) imposes that the updating functions $f^d_t$ can be approximated well by deep neural networks. Finally, Assumption~\ref{ass:Markov}(iv) requires that certain moments of the noise random variables and the ``constant parts'' of the update functions exhibit at most polynomial growth.

\begin{remark}
In Assumption~\ref{ass:Markov}(iii)-(iv) we could also put different constants $c$ and $q$ in each of the hypotheses. But then  Assumption~\ref{ass:Markov}(iii)-(iv) still holds with $c$ and $q$ chosen as the respective maximum and so for notational simplicity we choose to directly work with the same constants for all these hypotheses.  
\end{remark}

\begin{remark} \label{rmk:condExpDefined}
Let $s\geq t$ and consider $\bar{g}_{d,s} \colon \R^d \to \R$ defined on all of $\R^d$ with $\E[|\bar{g}_{d,s}(X_s^d)|]<\infty$. Then, under Assumption~\ref{ass:Markov}(i)-(ii), 
\begin{equation}
\label{eq:condExp}
\E[\bar{g}_{d,s}(X_s^d)|X_t^d=x] = \E[[\bar{g}_{d,s} \circ f_{s-1}^d(\cdot,Y_{s-1}^d)\circ\cdots \circ f_t^d(\cdot,Y_t^d)](x)]
\end{equation}
for $\P \circ (X_t^d)^{-1}$-a.e.\ $x \in \R^d$. But the right hand side of \eqref{eq:condExp} is defined for any $x \in \R^d$ for which the expectation is finite, and so in what follows we will also consider the conditional expectation  $\E[\bar{g}_{d,s}(X_s^d)|X_t^d=x]$ to be defined for all such $x \in \R^d$ (by \eqref{eq:condExp}). Note that also $\E[|\max(g_{d}(t,X_t^d),\E[\bar{g}_{d,s}(X_s^d)|X_t^d])|]\leq \E[|g_{d}(t,X_t^d)|+|\bar{g}_{d,s}(X_s^d)|] <\infty$ and so by backward induction this reasoning allows to define in our framework the value function $V_d(t,\cdot)$ on all of $\R^d$, for each $t$. 
\end{remark}

Next, we formulate a collection of hypotheses on the reward (or payoff) functions $g_d$.

\begin{assumption}\label{ass:Payoff}[Assumptions on $g_d$]
	There exists constants $c>0$, $q\geq 0, \alpha\geq0$ such that 
	\begin{itemize}
		\item[(i)] for all $\varepsilon \in (0,1]$, $d \in \N$, $t \in \{0,\ldots,T\}$
		there exists a neural network $\phi_{\varepsilon,d,t} \colon \R^d \to \R$ 
		with  
		\begin{align} \label{eq:NNclose}
		|g_d(t,x)-\phi_{\varepsilon,d,t}(x)| 
		& \leq \varepsilon c d^{q} (1+\|x\|^p), \quad \text{ for all } x \in \R^d,
		\\ \label{eq:NNsparse} 
		\mathrm{size}(\phi_{\varepsilon,d,t}) &\leq c d^{q} \varepsilon^{-\alpha}, 
		\\ \label{eq:NNlipschitz} 
		\mathrm{Lip}(\phi_{\varepsilon,d,t}) & \leq c d^{q},
		\end{align}
		\item[(ii)] for all $d \in \N$, $t \in \{0,\ldots,T\}$ it holds that $|g_d(t,0)|\leq c d^q$.
	\end{itemize}
\end{assumption}
Assumption~\ref{ass:Payoff}(i) means that $g_d(t,\cdot) \colon \R^d \to \R$ can be approximated well be neural networks for any $d \in \N$, $t \in \{0,\ldots,T\}$. Assumption~\ref{ass:Payoff}(ii) imposes that the ``constant part'' of the payoff grows at most polynomially in $d$. 
Lemma~\ref{lem:expFinite} below shows that the framework indeed ensures that 
$\E[|g_d(t,X_t^d)|]< \infty$, as required in Section~\ref{subsec:OptStopping}.

\begin{example}
\label{ex:payoffs}
Assumption~\ref{ass:Payoff} is satisfied in typical applications from mathematical finance. For instance, fix a strike price $K>0$, an interest rate $r\geq 0$ and consider the payoff of a max-call option $g_d(t,x) = e^{-r t}(\max_{i=1,\ldots,d} x_i - K)^+$. Then (see, e.g., \cite[Lemma~4.12]{HornungJentzen2018}) for each $t$ the payoff $g_d(t,\cdot)$ can be realized exactly by a neural network $\phi_{d,t}$ with $\mathrm{size}(\phi_{d,t}) \leq 6d^3$. In addition, $\mathrm{Lip}(g_d(t,\cdot))=1$ and 
therefore, setting $\phi_{\varepsilon,d,t} = \phi_{d,t}$ for all $\varepsilon \in (0,1]$ we get that Assumption~\ref{ass:Payoff} is satisfied with $c=6$, $\alpha=0$, $q=3$.
Further examples include basket call options, basket put options, call on min options and, by similar techniques, also put on min options, put on max options and many related payoffs.  
\end{example}

We now state the main deep neural network approximation result under the assumptions introduced above. 

\begin{theorem} \label{thm:DNNapprox}
Suppose Assumptions~\ref{ass:Markov} and \ref{ass:Payoff} are satisfied. Let $c >0$, $q \geq 0$ and assume for all $d \in \N$ that $\rho^d$ is a probability measure on $\R^d$ with $\int_{\R^d} \|x\|^{2\max(p,2)} \rho^d(dx) \leq c d^q$.

Then there exist constants $\kappa,\mathfrak{q},\mathfrak{r} \in [0,\infty)$ and 
neural networks $\psi_{\varepsilon,d,t}$, $\varepsilon \in (0,1]$, $d \in \N$, $t \in \{0,\ldots,T\}$,  
such that for any $\varepsilon \in (0,1]$,  $d \in \N$, $t \in \{0,\ldots,T\}$   
the number of neural network weights grows at most polynomially 
and the approximation error between the neural network $\psi_{\varepsilon,d,t}$ 
and the value function is at most $\varepsilon$, that is, $\mathrm{size}(\psi_{\varepsilon,d,t}) \leq \kappa d^{\mathfrak{q}}\varepsilon^{-\mathfrak{r}}$  and 
\begin{equation}
\label{eq:approxError}
\left(\int_{\R^d} |V_d(t,x) - \psi_{\varepsilon,d,t}(x)|^2 \rho^d(d x) \right)^{1/2} \leq \varepsilon.
\end{equation}
\end{theorem}

The proof of Theorem~\ref{thm:DNNapprox} will be given in Section~\ref{subsec:DNNapproxProof} below. 

Theorem~\ref{thm:DNNapprox} shows that under Assumptions~\ref{ass:Markov} and \ref{ass:Payoff} the value function $V_d$ can be approximated by deep neural networks without the curse of dimensionality: an approximation error at most $\varepsilon$ can be achieved by a deep neural network whose size is at most polynomial in $\varepsilon^{-1}$ and $d$. The approximation error in Theorem~\ref{thm:DNNapprox} is measured in the $L^2(\rho^d)$-norm.

Theorem~\ref{thm:DNNapprox} can also be used to deduce further properties of $V_d$.
In the basic framework we obtain for instance the following corollary, which shows that under Assumptions~\ref{ass:Markov} and ~\ref{ass:Payoff} for each $t$ the value function satisfies a certain average Lipschitz property with constant growing at most polynomially in $d$. 
\begin{corollary}
	\label{cor:Lipschitz}
Suppose Assumptions~\ref{ass:Markov} and \ref{ass:Payoff} are satisfied. Let $\nu_0^d$ be the standard Gaussian measure on $\R^d$. 
Then for any $R>0$ there exist constants $\kappa,\mathfrak{q} \in [0,\infty)$ 
such that for any $d \in \N$, $t \in \{0,\ldots,T\}$, $h \in [-R,R]^d$
the value function satisfies
\begin{equation}
\left(\int_{\R^d} |V_d(t,x) - V_d(t,x+h)|^2 \nu_0^d(d x) \right)^{1/2}
\leq \|h\| \kappa d^{\mathfrak{q}}.
\end{equation}
\end{corollary}

The proof of Corollary~\ref{cor:Lipschitz} will be given at the end of Section~\ref{subsec:DNNapproxProof}.

\subsection{Refined framework}
\label{subsec:refined}
We now introduce a refined framework, in which the approximation hypothesis \eqref{eq:fapprox} and the Lipschitz condition \eqref{eq:etaLipschitz} in Assumption~\ref{ass:Markov}(iii) are weakened, see \eqref{eq:fapproxBdd} and \eqref{eq:etaLipschitzBdd} below. 
Due to these weaker hypotheses we need to introduce potentially stronger moment assumptions on the noise variables $Y^d_t$. Note that the additional growth conditions \eqref{eq:fGrowthbdd} and \eqref{eq:etaGrowthbdd} are satisfied automatically under Assumption~\ref{ass:Markov} (see Lemma~\ref{lem:growth} and Remark~\ref{remark:growth} below).

\begin{assumptionp}{1'}\label{ass:Markov2}[Weaker assumptions on $X^d$]
Assume that (i), (ii) and (iv) in Assumption~\ref{ass:Markov} are satisfied. Furthermore, assume that 
there exist constants $c>0$, $h >0$, $q,\bar{q} \geq 0$, $\alpha \geq 0$, $\beta>0$, $m \in \N$, $\theta \geq 0$ and $\zeta \geq 0$  such that 
 for all $\varepsilon \in (0,1]$, $d \in \N$, $t \in \{0,\ldots,T-1\}$
		there exists a neural network $\eta_{\varepsilon,d,t} \colon \R^d \times \R^d \to \R^d$ with 
		\begin{align}
		\label{eq:fapproxBdd}
		\| f^d_{t}(x,y) - \eta_{\varepsilon,d,t}(x,y)\| & \leq \varepsilon c d^{q} (1+\|x\|^p + \|y\|^p), \quad \text{ for all } x,y \in [-(\varepsilon^{-\beta}),\varepsilon^{-\beta}]^d, 
		\\ \label{eq:etasparseBdd} 
		\mathrm{size}(\eta_{\varepsilon,d,t}) &\leq c d^{q} \varepsilon^{-\alpha}, 
		\\ \label{eq:etaLipschitzBdd}
		\mathrm{Lip}(\eta_{\varepsilon,d,t}) & \leq c d^{q} \varepsilon^{-\zeta}
		\end{align}
	and for all $x,y \in \R^d$
	\begin{align}
	\label{eq:fGrowthbdd}
	\E[\|f_t^d(x,Y_t^d)\|^{2m\max(p,2)}] & \leq  h d^{\bar{q}}  (1+\|x\|^{2m\max(p,2)}),
	\\ \label{eq:etaGrowthbdd}
	\|\eta_{\varepsilon,d,t}(x,y)\| & \leq   c d^{q} \varepsilon^{-\theta} (2+\|x\|+\|y\|)
	\end{align}
and $\E[\|Y_{t}^d\|^{2m\max{(2,p)}}] \leq c d^{q}$.
\end{assumptionp}

\begin{remark}\label{remark:growth}
A sufficient condition for \eqref{eq:fGrowthbdd} is that there exist $\tilde{c}>0$ and $\tilde{q}\geq 0$ such that for all $d\in \N$, $x,y \in \R^d$ we have $\E[\|Y_{t}^d\|^{2m\max{(2,p)}}] \leq \tilde{c} d^{\tilde{q}}$ and 
$\|f_t^d(x,y)\| \leq  \tilde{c} d^{\tilde{q}}  (1+\|x\|+\|y\|)$. Then
\[\begin{aligned}
\E[\|f_t^d(x,Y_t^d)\|^{2m\max(p,2)}] & \leq (\tilde{c} d^{\tilde{q}})^{2m\max(p,2)} \E[(1+\|x\|+\|Y_t^d\|)^{2m\max(p,2)}]   
\\ & \leq  (3\tilde{c} d^{\tilde{q}})^{2m\max(p,2)}  (1+\|x\|^{2m\max(p,2)}+\tilde{c} d^{\tilde{q}}).
\end{aligned}
\]
\end{remark}

\begin{remark}
While in many relevant applications the number of time steps $T$ is only moderate (e.g.\ around $10$ in \cite[Sections~4.1--4.2]{Becker2019}), it is also important to analyse the situation when $T$ is large. To this end, in Assumption~\ref{ass:Markov2} we have introduced the constants $h$ and $\bar{q}$ instead of using the common upper bounds $c$, $q$. This makes it possible to get first insights about the situation in which $T$ is large from the proofs in Section~\ref{sec:Proofs}. Indeed, if $h=1+\tilde{h}$ and $\tilde{h}$ is sufficiently small (as it is the case for instance in certain discretized diffusion models), then the constants in Lemma~\ref{lem:growthMoments} and Lemma~\ref{lem:Vgrowth} are small also for large $T$.
\end{remark}

Examples of processes that satisfy Assumption~\ref{ass:Markov2} are provided further below. These include, in particular, the Black-Scholes model, more general exponential L\'evy processes and discrete diffusions.

We now state the main theorem of the article. 

\begin{theorem}\label{thm:refined}
Suppose Assumptions~\ref{ass:Markov2} and \ref{ass:Payoff} are satisfied. Let $c >0$, $q \geq 0$ and assume for all $d \in \N$ that $\rho^d$ is a probability measure on $\R^d$ with $\int_{\R^d} \|x\|^{2m\max(p,2)} \rho^d(dx) \leq c d^q$. Furthermore, assume that $\zeta < \frac{\min(1,\beta m-\theta)}{T-1}$, where $m,\beta,\zeta,\theta$ are the constants appearing in Assumption~\ref{ass:Markov2}.

Then there exist constants $\kappa,\mathfrak{q},\mathfrak{r} \in [0,\infty)$ and 
neural networks $\psi_{\varepsilon,d,t}$, $\varepsilon \in (0,1]$, $d \in \N$, $t \in \{0,\ldots,T\}$, 
such that for any $\varepsilon \in (0,1]$,  $d \in \N$, $t \in \{0,\ldots,T\}$   
the number of neural network weights grows at most polynomially 
and the approximation error between the neural network $\psi_{\varepsilon,d,t}$ 
and the value function is at most $\varepsilon$, that is, $\mathrm{size}(\psi_{\varepsilon,d,t}) \leq \kappa d^{\mathfrak{q}}\varepsilon^{-\mathfrak{r}}$  and
\begin{equation}
\label{eq:approxError2}
\left(\int_{\R^d} |V_d(t,x) - \psi_{\varepsilon,d,t}(x)|^2 \rho^d(d x) \right)^{1/2} \leq \varepsilon.
\end{equation}
\end{theorem}

The proof of Theorem~\ref{thm:refined} will be given in Section~\ref{subsec:DNNapproxProofrefined} below. 

Theorem~\ref{thm:refined} shows that for Markov processes satisfying Assumption~\ref{ass:Markov2} and for reward functions satisfying Assumption~\ref{ass:Payoff} the value function of the associated optimal stopping problem can be approximated by deep neural networks without the curse of dimensionality. In other words, an approximation error at most $\varepsilon$ can be achieved by a deep neural network whose size is at most polynomial in $\varepsilon^{-1}$ and $d$.

The condition $\zeta < \frac{\min(1,\beta m-\theta)}{T-1}$ in Theorem~\ref{thm:refined} can be viewed as a condition on $m$, which needs to be sufficiently large. This means that sufficiently high moments of $Y^d_t$ need to exist and grow only polynomially in $d$. 

A key step in the proof consists in constructing a deep neural network approximating the continuation value. Therefore, we immediately obtain the following corollary. 

\begin{corollary}\label{cor:continuation}
Consider the setting of Theorem~\ref{thm:refined}. Then for each $\varepsilon \in (0,1]$, $d \in \N$, $t \in \{0,\ldots,T\}$   there exists a neural network $\gamma_{\varepsilon,d,t}$
such that 	$\mathrm{size}(\gamma_{\varepsilon,d,t}) \leq \kappa d^{\mathfrak{q}}\varepsilon^{-\mathfrak{r}}$ 
and 	
\begin{equation}
	\label{eq:approxError2Cor}
	\left(\int_{\R^d} |\E[V_d(t+1,X_{t+1}^d)|X_t^d=x]  - \gamma_{\varepsilon,d,t}(x)|^2 \rho^d(d x) \right)^{1/2} \leq \varepsilon.
	\end{equation}
\end{corollary}

\subsection{Exponential L\'evy models}
\label{subsec:expLevy}
In this subsection we apply Theorem~\ref{thm:refined} to exponential L\'evy models. 
Recall that an $\R^d$-valued stochastic process $L^d = (L^d_t)_{t \geq 0}$ is called a ($d$-dimensional) L\'evy process if it is stochastically continuous, its sample paths are almost surely right continuous with left limits, it has stationary and independent increments and $\P(L^d_0=0)=1$. A L\'evy process $L^d$ is fully characterized by its  L\'evy  triplet 
$(A^d,\gamma^d, \nu^d)$ where $A^d \in \R^{d\times d}$ is a symmetric, nonnegative definite matrix, $\gamma^d \in \R^d$
and $\nu^d$ is a measure on $\R^d$ describing the jump structure of $L^d$.  We refer, e.g., to \cite{Sato1999}, \cite{Applebaum2009} for more detailed statements of these definitions, proofs of this characterization and further details on L\'evy processes.

A stochastic process $X^d$ is said to follow an exponential L\'evy model, if 
\begin{equation} \label{eq:expLevy}
X^d_t
= 
\big(x^d_1\exp(L_{t,1}^d),\ldots, x^d_d\exp(L_{t,d}^d)\big), 
\quad t \in \{0,\ldots,T\}
\end{equation}
for a $d$-dimensional L\'evy process $L^d = (L^d_t)_{t \geq 0}$ and $x^d \in \R^d$.  
We refer to \cite{Cont2004}, \cite{EberKall19} for more details on financial modelling using exponential L\'evy models.

From Theorem~\ref{thm:refined} we now obtain the following deep neural network approximation result. This result includes the case of a Black-Scholes model ($\nu^d=0$) as well as pure jump models ($A^d_{i,j} = 0$) with sufficiently integrable tails. In particular, Corollary~\ref{cor:expLevy} applies to prices of American / Bermudan basket put options, put on min or put on max options in such models (cf. Example~\ref{ex:payoffs} for payoffs that satisfy Assumption~\ref{ass:Payoff}). 

\begin{corollary}
\label{cor:expLevy}
Let $X^d$ follow an exponential L\'evy model with L\'evy triplet $(A^d,\gamma^d,\nu ^d)$ and assume the triplets are bounded in the dimension, that is, 
there exists $B > 0$ such that 
for any $d \in \N$, $i,j=1,\ldots,d$
\begin{equation}
\label{eq:LevyTripletBounded2}
\max
\left(
A^d_{i,j},
\gamma^d_i ,
\int_{\{\|z\|> 1\}} e^{2(T+1)\max(p,2) z_i} \nu^d(d z),
\int_{\{\|z\|\leq 1\}} z_i^2 \nu^d(d z)
\right) \leq B.
\end{equation} 
Suppose the payoff functions $g_d$ satisfy Assumption~\ref{ass:Payoff}. Let $c >0$, $q \geq 0$ and assume for all $d \in \N$ that $\rho^d$ is a probability measure on $\R^d$ with $\int_{\R^d} \|x\|^{2(T+1)\max(p,2)} \rho^d(dx) \leq c d^q$.  
		
	Then there exist constants $\kappa,\mathfrak{q},\mathfrak{r} \in [0,\infty)$ and 
	neural networks $\psi_{\varepsilon,d,t}$, $\varepsilon \in (0,1]$, $d \in \N$, $t \in \{0,\ldots,T\}$,  
	such that for any $\varepsilon \in (0,1]$,  $d \in \N$, $t \in \{0,\ldots,T\}$ 
	\begin{equation}
	\label{eq:approxErrorCor}
	\mathrm{size}(\psi_{\varepsilon,d,t}) \leq \kappa d^{\mathfrak{q}}\varepsilon^{-\mathfrak{r}} \quad \text{ and } \quad 
	\left(\int_{\R^d} |V_d(t,x) - \psi_{\varepsilon,d,t}(x)|^2 \rho^d(d x) \right)^{1/2} \leq \varepsilon.
	\end{equation}
\end{corollary}
\begin{proof}
This follows directly from Theorem~\ref{thm:refined} and Lemma~\ref{lem:ExpLevy} with the choice $\zeta = \theta = \beta = \frac{1}{T}$, $m=T+1$, which ensures that  $\zeta < \frac{1}{T-1} = \frac{\min(1,\beta m-\theta)}{T-1}$.
\end{proof}

\subsection{Discrete diffusion models}
\label{subsec:discretediffusions}
Let $\bar{T}>0$ and let $X^d$ follow a discrete diffusion model with coefficients $\mu^d \colon [0,\bar{T}] \times \R^d \to \R^d$, $\sigma^d \colon [0,\bar{T}] \times \R^d \to \R^{d \times d}$, that is, $X^d$ satisfies $X_0^d=x^d$ and 
\begin{equation}
\label{eq:discreteDiffusion}
X^d_{k+1} = X_{k}^d + \mu^d(\mathfrak{t}_k,X_{k}^d) (\mathfrak{t}_{k+1}-\mathfrak{t}_k) + \sigma^d(\mathfrak{t}_k,X_k^d) (W^d_{\mathfrak{t}_{k+1}}-W^d_{\mathfrak{t}_k}), \quad k \in \{0,\ldots,T-1\}
\end{equation}
for some $0 \leq \mathfrak{t}_0<\mathfrak{t}_1<\ldots<\mathfrak{t}_T \leq \bar{T}$, $x^d \in \R^d$ and $W^d$ a $d$-dimensional Brownian motion.
Consider the following assumption on the drift and diffusion coefficients: 
\begin{assumption} \label{ass:driftdiffusion}
Assume that there exist constants $C>0$, $q,\tilde{\alpha},\tilde{\zeta} \geq 0$ and, for each $d \in \mathbb{N}$, $t \in \{0,\ldots,T-1\}$, $\varepsilon \in (0,1]$, 
there exist neural networks $\mu_{\varepsilon,d,t} \colon \R^d \to \R^d$ and
$\sigma_{\varepsilon,d,t,i} \colon \R^d \to \R^d$, $i=1,\ldots,d$, 
such that for all $d \in \mathbb{N}$, $\varepsilon \in (0,1]$, $t \in \{0,\ldots,T-1\}$, $x \in \R^d$ it holds that
\[ 
\begin{aligned}
\|\mu^d(\mathfrak{t}_t,x)-\mu_{\varepsilon,d,t}(x)\|+\|\sigma^d(\mathfrak{t}_t,x)-\sigma_{\varepsilon,d,t}(x)\|_F 
& \leq \varepsilon C d^q (1+\|x\|),
\\
\|\mu^{d}(\mathfrak{t}_t,x)\|+\|\sigma^{d}(\mathfrak{t}_t,x)\|_F
& \leq  C d^q (1 + \|x\|),		
\\ 
\mathrm{size}(\mu_{\varepsilon,d,t}) + \sum_{i=1}^d \mathrm{size}(\sigma_{\varepsilon,d,t,i}) 
& \leq C d^q \varepsilon^{-\tilde{\alpha}},
\\
\max\left(\mathrm{Lip}(\mu_{\varepsilon,d,t}),\mathrm{Lip}(\sigma_{\varepsilon,d,t,1}),\ldots,\mathrm{Lip}(\sigma_{\varepsilon,d,t,d})\right) & \leq C d^{q} \varepsilon^{-\tilde{\zeta}}.
\end{aligned} 
\]
\end{assumption}
Here we denote by $\sigma_{\varepsilon,d,t}(x) \in  \R^{d \times d}$ the matrix with $i$-th row $\sigma_{\varepsilon,d,t,i}(x)$.

\begin{corollary}
	\label{cor:DNNApproxDiscreteDiffusion} Let $X^d$ follow a discrete diffusion model with coefficients satisfying Assumption~\ref{ass:driftdiffusion} with $\tilde{\zeta}<\frac{1}{T-1}$.  
	Suppose $p \geq 2$ and  the payoff functions $g_d$ satisfy Assumption~\ref{ass:Payoff}. Let $c >0$, $q \geq 0$ and assume for all $d \in \N$ that $\rho^d$ is a probability measure on $\R^d$ with $\int_{\R^d} \|x\|^{2m\max(p,2)} \rho^d(dx) \leq c d^q$ for $m=\lceil \frac{2(1+\tilde{\zeta})}{\frac{1}{T-1}-\tilde{\zeta}} +1 \rceil $. 

Then there exist constants $\kappa,\mathfrak{q},\mathfrak{r} \in [0,\infty)$ and 
neural networks $\psi_{\varepsilon,d,t}$, $\varepsilon \in (0,1]$, $d \in \N$, $t \in \{0,\ldots,T\}$,  
such that for any $\varepsilon \in (0,1]$,  $d \in \N$, $t \in \{0,\ldots,T\}$ 
\begin{equation}
\label{eq:approxErrorCor2}
\mathrm{size}(\psi_{\varepsilon,d,t}) \leq \kappa d^{\mathfrak{q}}\varepsilon^{-\mathfrak{r}} \quad \text{ and } \quad 
\left(\int_{\R^d} |V_d(t,x) - \psi_{\varepsilon,d,t}(x)|^2 \rho^d(d x) \right)^{1/2} \leq \varepsilon.
\end{equation}	
\end{corollary}

\begin{proof}
By Lemma~\ref{lem:DiscreteDiffusion} it follows that Assumption~\ref{ass:Markov2} is satisfied. In addition, the constant $\beta >0 $ in Assumption~\ref{ass:Markov2} may be chosen arbitrarily and $\zeta = \theta = \beta + \tilde{\zeta}$. Thus, we may select $\beta = \frac{1}{T-1}-\tilde{\zeta}-\delta$ for some $\delta >0$ and then $\beta>0$ and $\zeta = \theta = \frac{1}{T-1}-\delta$. Choosing $\delta = \frac{1}{2}(\frac{1}{T-1}-\tilde{\zeta})$,  $m=\lceil \frac{1+\tilde{\zeta}}{\beta} +1 \rceil $ then ensures that  $\zeta < \frac{1}{T-1} = \frac{\min(1,\beta m-\theta)}{T-1}$. Theorem~\ref{thm:refined} hence implies \eqref{eq:approxErrorCor2}. 
\end{proof}

\subsection{Running minimum and maximum}
\label{subsec:runningMax}
In this subsection we show that our framework can also cover barrier options. This follows from the next proposition, which proves that for processes satisfying Assumption~\ref{ass:Markov2} also the processes augmented by their running maximum or minimum satisfy Assumption~\ref{ass:Markov2}. 
\begin{proposition} \label{prop:max}
Suppose Assumption~\ref{ass:Markov2} holds. Let $\bar{X}^1 = X^1$ and for $d \in \N$, $d \geq 2$, $t \in \{0,\ldots,T\}$ consider the $\R^{d}$-valued process 
$\bar{X}_t^d = (X^{d-1}_t,M^d_t)$, where either $M^d_t = \min_{i=1,\ldots,d-1} \min_{s=0,\ldots,t} X^{d-1}_{s,i}$ or $M^d_t = \max_{i=1,\ldots,d-1} \max_{s=0,\ldots,t} X^{d-1}_{s,i}$.  
Then $\bar{X}^d$, $d \in \N$, satisfy Assumption~\ref{ass:Markov2}.   
\end{proposition}

The proof is given at the end of Section~\ref{subsec:sufficient} below.

\section{Proofs}\label{sec:Proofs}

This section contains the remaining proofs of the results in Section~\ref{sec:DNN}. The section is split in several subsections. In Section~\ref{subsec:product} we provide a refined result on deep neural network approximations of the product function $\R \times \R \to \R, (x,y) \mapsto xy$. Section~\ref{subsec:sufficient} then contains two lemmas in which this approximation result is applied to verify that suitable exponential L\'evy and discrete diffusion models satisfy Assumption~\ref{ass:Markov2}. Subsequently, Section~\ref{subsec:auxiliary} contains auxiliary results needed for the proofs of Theorem~\ref{thm:DNNapprox} and Theorem~\ref{thm:refined}. The proofs of these two results are then contained in Sections~\ref{subsec:DNNapproxProof} and \ref{subsec:DNNapproxProofrefined}. 

\subsection{Deep neural network approximation of the product}
\label{subsec:product}
Based on \cite[Proposition~3]{Yarotsky2017} we provide here a refined result regarding deep neural network approximations of the product function $\R \times \R \to \R, (x,y) \mapsto xy$.
\begin{lemma}
	\label{lem:productrefined}
	There exists $c>0$ such that for any $\varepsilon \in (0,1]$, $M\geq 1$ there exists a neural network ${\mathfrak{n}}_{\varepsilon,M} \colon \R \times \R \to \R$
	with
	\begin{equation}
	\label{eq:productApprox}
	\sup_{x,y \in [-M,M]} |{\mathfrak{n}}_{\varepsilon,M}(x,y) - x y| < \varepsilon,
	\end{equation}
	$\mathrm{size}({\mathfrak{n}}_{\varepsilon,M}) \leq c(\log(\varepsilon^{-1})+\log(M)+1)$ and for all $x,x',y,y' \in \R$
	\[
	|{\mathfrak{n}}_{\varepsilon,M}(x,y)-{\mathfrak{n}}_{\varepsilon,M}(x',y')| \leq M c(|x-x'|+|y-y'|).
	\]
\end{lemma}
\begin{proof}
	By \cite[Lemma~4.2]{Grohs2021} or \cite[Proposition~4.1]{Opschoor2020} (based on \cite[Proposition~3]{Yarotsky2017}), there exists $c>0$ such that for any ${\bar{\varepsilon}} \in (0,\frac{1}{2})$ there exists a neural network $\mathfrak{n}_{{\bar{\varepsilon}}} \colon \R \times \R \to \R$  with the property that $\sup_{x,y \in [-1,1]} |\mathfrak{n}_{\bar{\varepsilon}}(x,y) - x y| < {\bar{\varepsilon}}$,  $\mathrm{size}(\mathfrak{n}_{\bar{\varepsilon}}) \leq c(\log({\bar{\varepsilon}}^{-1})+1)$ and \[\sup_{x,x',y,y' \in [-1,1]}|\mathfrak{n}_{\bar{\varepsilon}}(x,y)-\mathfrak{n}_{\bar{\varepsilon}}(x',y')| \leq c(|x-x'|+|y-y'|).
	\]

	Consider now the capped neural network $\bar{\mathfrak{n}}_{\bar{\varepsilon}}(x,y) = \mathfrak{n}_{\bar{\varepsilon}}(\pi_1(x),\pi_1(y))$, where we set $\pi_1(z)=\max(-1,\min(z,1))$. Then $\bar{\mathfrak{n}}_{\bar{\varepsilon}}(x,y) = \mathfrak{n}_{\bar{\varepsilon}} \circ \mathrm{cap}(x,y)$ and it can be verified that $\mathrm{cap}$ is again a neural network 
	and for $x,y \in [-1,1]$ we have  $\bar{\mathfrak{n}}_{\bar{\varepsilon}}(x,y) = \mathfrak{n}_{\bar{\varepsilon}}(x,y) $.	
	The fact that the composition of two ReLU neural networks can again be realized by a ReLU neural network with size bounded by twice the sum of the respective sizes (see, e.g. \citet[Proposition~2.2]{Opschoor2020}) hence proves that there exists $\tilde{c}\geq c$ such that for all ${\bar{\varepsilon}} \in (0,\frac{1}{2})$ we have $\mathrm{size}(\bar{\mathfrak{n}}_{\bar{\varepsilon}}) \leq \tilde{c}(\log({\bar{\varepsilon}}^{-1})+1)$. Furthermore,   $\sup_{x,y \in [-1,1]} |\bar{\mathfrak{n}}_{\bar{\varepsilon}}(x,y) - x y| < {\bar{\varepsilon}}$ and for all $x,x',y,y' \in \R$
	\[
	|\bar{\mathfrak{n}}_{\bar{\varepsilon}}(x,y)-\bar{\mathfrak{n}}_{\bar{\varepsilon}}(x',y')| \leq c(|\pi_1(x)-\pi_1(x')|+|\pi_1(y)-\pi_1(y')|) \leq \tilde{c}(|x-x'|+|y-y'|).
	\]
	Now let $\varepsilon \in (0,1]$, $M \geq 1$ be given, choose $\bar{\varepsilon} = 3^{-1} M^{-2} \varepsilon$ and define the rescaled network $\mathfrak{n}_{{{\varepsilon}},M}(x,y) = M^2 \bar{\mathfrak{n}}_{\bar{\varepsilon}}(\frac{x}{M},\frac{y}{M})$. Then
	\[
	\sup_{x,y \in [-M,M]} |{\mathfrak{n}}_{{{\varepsilon}},M}(x,y) - x y| = M^2 \sup_{x,y \in [-M,M]} |\bar{\mathfrak{n}}_{{\bar{\varepsilon}}}(\frac{x}{M},\frac{y}{M}) - \frac{x}{M} \frac{y}{M}| < M^2 {\bar{\varepsilon}},
	\]
	$\mathrm{size}(\mathfrak{n}_{{{\varepsilon}},M}) \leq \tilde{c}(\log({\bar{\varepsilon}}^{-1})+1)$ and for all $x,x',y,y' \in \R$
	\[
	|\mathfrak{n}_{{{\varepsilon}},M}(x,y)-{\mathfrak{n}}_{{{\varepsilon}},M}(x',y')| = M^2 |\bar{\mathfrak{n}}_{{\bar{\varepsilon}}}(\frac{x}{M},\frac{y}{M})-\bar{\mathfrak{n}}_{{\bar{\varepsilon}}}(\frac{x'}{M},\frac{y'}{M})| \leq M \tilde{c}(|x-x'|+|y-y'|).
	\]
\end{proof}

\subsection{Sufficient conditions}
\label{subsec:sufficient}
In this subsection we prove Lemma~\ref{lem:ExpLevy} and Lemma~\ref{lem:DiscreteDiffusion}, which show that the exponential L\'evy and discrete diffusion models considered above satisfy Assumption~\ref{ass:Markov2}. We also provide a proof of Proposition~\ref{prop:max}. 
\begin{lemma}
\label{lem:ExpLevy}
Let  $X^d$ follow an exponential L\'evy model (cf. \eqref{eq:expLevy}) for each $d \in \N$ and assume that the L\'evy triplets $(A^d,\gamma^d,\nu ^d)$ are bounded in the dimension, that is, 
there exists $B > 0$ such that 
for any $d \in \N$, $i,j=1,\ldots,d$
\begin{equation}
\label{eq:LevyTripletBounded}
\max
\left(
A^d_{i,j},
\gamma^d_i ,
\int_{\{\|z\|> 1\}} e^{\bar{p} z_i} \nu^d(d z),
\int_{\{\|z\|\leq 1\}} z_i^2 \nu^d(d z)
\right) \leq B,
\end{equation}
where $\bar{p} = 2m\max{(2,p)}$. 
Then Assumption~\ref{ass:Markov2} is satisfied with constant $\beta >0 $ in Assumption~\ref{ass:Markov2} chosen arbitrarily and with $\zeta = \theta = \beta$. 
\end{lemma}

\begin{proof}
Firstly, \eqref{eq:expLevy} shows for each $d \in \N$, $t \in \{0,\ldots,T-1\}$ that $X_{t+1,i}^d = X_{t,i}^d \exp(L_{t+1,i}^d-L_{t,i}^d)$ for all $i=1,\ldots, d$. Therefore, $X_{t+1}^d = f^d_t(X_t^d,Y^d_t)$ with $Y_{t,i}^d = \exp(L_{t+1,i}^d-L_{t,i}^d)$ and $f^d_t \colon \R^d \times \R^d \to \R^d$ given by $f^d_t(x,y) = (x_1 y_1,\ldots,x_d y_d)$ for $x,y \in \R^d$, i.e., \eqref{eq:MarkovUpdate} is satisfied. 
Since $L^d$ has independent increments, it follows that Assumption~\ref{ass:Markov}(ii) is satisfied.
Next, we can employ an argument from the proof of \cite[Theorem~5.1]{GS20_925} (which uses \cite[Theorem~25.17]{Sato1999}  and 
\eqref{eq:LevyTripletBounded}) to obtain for any $d \in \N$, $i=1,\ldots,d$ that
\begin{align} 
\nonumber
\E[e^{\bar{p} L_{1,i}^d}] & = \exp\left(\frac{\bar{p}^2}{2}A^d_{i,i} + \int_{\R^d} (e^{\bar{p} y_i}-1-\bar{p} y_i \mathbbm{1}_{\{\|y\|\leq 1\}}) \nu^d(d y) + \bar{p} \gamma^d_i \right)
\\ & \leq \exp\left(\frac{5 \bar{p}^2}{2}B + \bar{p}^2 e^{\bar{p}} B \right).
\label{eq:expMoment} 
\end{align}
Combined with Minkowski's inequality and the stationarity increments property of $L^d$ this yields 
\begin{equation}
\begin{aligned}
\label{eq:auxEq57} 
\E[\|Y_{t}^d\|^{2m\max{(2,p)}}] & = \left(\E\left[\left(\sum_{i=1}^d |Y_{t,i}^d|^2\right)^{m\max{(2,p)}}\right]^{\frac{1}{m\max{(2,p)}}}\right)^{m\max{(2,p)}}
\\ & \leq  \left(\sum_{i=1}^d \E\left[ |Y_{t,i}^d|^{2m\max{(2,p)}}\right]^{\frac{1}{m\max{(2,p)}}}\right)^{m\max{(2,p)}}
\\ & =  \left(\sum_{i=1}^d \E\left[ e^{2m\max{(2,p)}L_{1,i}^d}\right]^{\frac{1}{m\max{(2,p)}}}\right)^{m\max{(2,p)}}
\\ & \leq d^{m\max{(2,p)}} \exp\left(\frac{5 \bar{p}^2}{2}B + \bar{p}^2 e^{\bar{p}} B \right).
\end{aligned} 
\end{equation}
Furthermore, $f_t^d(0,0) = 0$ and thus Assumption~\ref{ass:Markov}(iv) is satisfied. Next, for $\varepsilon \in (0,1]$, $d \in \N$, $t \in \{0,\ldots,T-1\}$ let $M=\varepsilon^{-\beta}$ and let $\eta_{\varepsilon,d,t} \colon \R^d \times \R^d \to \R^d$ be the $d$-fold parallelization of ${\mathfrak{n}}_{\varepsilon,M}$ from Lemma~\ref{lem:productrefined}. Then for all $x,y \in [-(\varepsilon^{-\beta}),\varepsilon^{-\beta}]^d$ we obtain  
\[
\| f^d_{t}(x,y) - \eta_{\varepsilon,d,t}(x,y)\| = \left(\sum_{i=1}^d |x_i y_i - {\mathfrak{n}}_{\varepsilon,M}(x_i,y_i)|^2 \right)^{1/2} \leq \varepsilon d^{\frac{1}{2}},
\]
$\mathrm{size}(\eta_{\varepsilon,d,t}) \leq d \mathrm{size}({\mathfrak{n}}_{\varepsilon,M}) \leq  cd((\beta+1)\log(\varepsilon^{-1})+1) \leq c_1 d \varepsilon^{-1}$ with $c_1 = c(\beta+2)$ and for all $x,x',y,y' \in \R^d$
\[
\begin{aligned}
\| \eta_{\varepsilon,d,t}(x,y) - \eta_{\varepsilon,d,t}(x',y')\| & = \left(\sum_{i=1}^d | {\mathfrak{n}}_{\varepsilon,M}(x_i,y_i)-{\mathfrak{n}}_{\varepsilon,M}(x_i',y_i')|^2 \right)^{1/2} 
\\ & \leq 
\left(\sum_{i=1}^d | Mc(|x_i-x_i'| + |y_i-y_i'|)|^2 \right)^{1/2} 
\\ & \leq  \sqrt{2}
\varepsilon^{-\beta} c (\|x-x'\| + \|y-y'\|).
\end{aligned}
\]
Finally, for all $x,y \in \R^d$
\[
\begin{aligned}
\|\eta_{\varepsilon,d,t}(x,y)\| &  \leq  \| \eta_{\varepsilon,d,t}(x,y)-\eta_{\varepsilon,d,t}(0,0)\|+\|\eta_{\varepsilon,d,t}(0,0)-f_t^d(0,0)\| 
\\ & \leq  \varepsilon^{-\beta} c (\|x\| + \|y\|) + \varepsilon d^{\frac{1}{2}}
\\ & \leq \varepsilon^{-\beta} \max(c,1)  d^{\frac{1}{2}} (1+\|x\| + \|y\|) 
\end{aligned}
\]
and Minkowski's integral inequality and \eqref{eq:auxEq57} imply
\begin{align*}
\E[\|f_t^d(x,Y_t^d)\|^{2m\max(p,2)}] &  = \E\left[\left(\sum_{i=1}^d x_i^2 (Y_{t,i}^d)^2 \right)^{m\max(p,2)}\right]
\\ &  \leq \left( \sum_{i=1}^d \left(\E[x_i^{2m\max(p,2)} (Y_{t,i}^d)^{2m\max(p,2)}]\right)^{1/(m\max(p,2))} \right)^{m\max(p,2)}
\\ &  \leq \left( \sum_{i=1}^d x_i^{2}  \right)^{m\max(p,2)} \E[ \|Y_{t}^d\|^{2m\max(p,2)}]
\\ & \leq  d^{m\max{(2,p)}} \exp\left(\frac{5 \bar{p}^2}{2}B + \bar{p}^2 e^{\bar{p}} B \right)  (1+\|x\|^{2m\max(p,2)}).
\end{align*}
\end{proof}

\begin{lemma}
\label{lem:DiscreteDiffusion}
Assume $p \geq 2$, let $X^d$ follow a discrete diffusion model with coefficients $\mu^d \colon [0,\bar{T}] \times \R^d \to \R^d$, $\sigma^d \colon [0,\bar{T}] \times \R^d \to \R^{d \times d}$ and suppose Assumption~\ref{ass:driftdiffusion} holds. 
Then Assumption~\ref{ass:Markov2} is satisfied with constants $m \in \N$, $\beta >0 $ in Assumption~\ref{ass:Markov2} chosen arbitrarily and with $\zeta = \theta = \beta + \tilde{\zeta}$. 
\end{lemma}
\begin{proof} 
Firstly, \eqref{eq:MarkovUpdate} holds with $f^d_{t}(x,y) = x + \mu^d(\mathfrak{t}_t,x) (\mathfrak{t}_{t+1}-\mathfrak{t}_t) + \sigma^d(\mathfrak{t}_t,x)y$ and $Y_t^d= W^d_{\mathfrak{t}_{t+1}}-W^d_{\mathfrak{t}_t}$.
Assumption~\ref{ass:Markov}(ii) holds by the independent increment property of Brownian motion.  

Next, for all $d \in \N$, $t \in \{0,\ldots,T-1\}$ it holds that $\|f_t^d(0,0)\| = \|\mu^d(\mathfrak{t}_t,0) (\mathfrak{t}_{t+1}-\mathfrak{t}_t) \| \leq C \bar{T} d^q$ and, with $\bar{p}= m\max{(2,p)}$, 
$\E[\|Y_{t}^d\|^{2\bar{p}}]  \leq \bar{T}^{\bar{p}} \E[\|Z\|^{2\bar{p}}]$ for $Z$ standard normally distributed in $\R^d$. The fact that $\|Z\|^2 \sim \chi^2(d)$ and the upper and lower bounds for the gamma function (see, e.g., \cite[Lemma~2.4]{Gonon2019}) thus yield
\begin{equation}
\label{eq:auxEq59}
\begin{aligned}
\E[\|Y_{t}^d\|^{2\bar{p}}]  & \leq \bar{T}^{\bar{p}}  \frac{2^{\bar{p}}\Gamma(\bar{p}+\frac{d}{2})}{\Gamma(\frac{d}{2})} \leq \bar{T}^{\bar{p}} 2^{\bar{p}} \left(\frac{e}{\frac{d}{2}} \right)^{\frac{d}{2}} \left(\frac{\bar{p}+\frac{d}{2}}{e}\right)^{\bar{p}+\frac{d}{2}} e \\ & \leq  [(4\bar{T}\bar{p})^{\bar{p}} c_{\bar{p}}  e^{1-\bar{p}}] d^{\bar{p}}
\end{aligned}
\end{equation}
with $c_{\bar{p}} = \max_{n \in \N} \left(1+\frac{2\bar{p}}{n} \right)^{\frac{n}{2}} < \infty$. 

Next, for $\varepsilon \in (0,1]$, $d \in \N$, $t \in \{0,\ldots,T-1\}$ let $M=4 \max(C,1) d^{q+\frac{1}{2}} \varepsilon^{-\beta}$ and consider  $\eta_{\varepsilon,d,t} \colon \R^d \times \R^d \to \R^d$ given by
\begin{equation}\label{eq:auxEq54}
\eta_{{\varepsilon},d,t,i}(x,y) = x_i +\mu_{\varepsilon,d,t,i}(x) (\mathfrak{t}_{t+1}-\mathfrak{t}_t) + \sum_{j=1}^d {\mathfrak{n}}_{\varepsilon,M}(\sigma_{\varepsilon,d,t,i,j}(x),y_j)
\end{equation}
for $i=1,\ldots,d$ with ${\mathfrak{n}}_{\varepsilon,M}$ from Lemma~\ref{lem:productrefined}. By using the operations of parallelization and concatenation it follows that we can realize $(x,y)\mapsto {\mathfrak{n}}_{\varepsilon,M}(\sigma_{\varepsilon,d,t,i,j}(x),y_j)$ by a neural network of size $\mathfrak{s}_{i,j} := 2(\mathrm{size}(\mathfrak{n}_{\varepsilon,M})+2+\mathrm{size}(\sigma_{\varepsilon,d,t,i,j}))$, see, e.g., \citet[Propositions~2.2 and 2.3]{Opschoor2020}.  Recall that the identity on $\R$ can be realized by a ReLU deep neural network of arbitrary depth $\ell$ and size $2 \ell$ (see \cite[Remark~2.4]{PETERSEN2018296}, \cite[Proposition~2.4]{Opschoor2020}). Thus, we may insert identity networks in \eqref{eq:auxEq54} to ensure that all summands can be realized by networks of the same depth, which is at most $\max(\mathrm{size}(\mu_{\varepsilon,d,t,i}),\mathfrak{s}_{i,1},\ldots,\mathfrak{s}_{i,d})$ due to the fact that the depth of a network is bounded by its size. By applying the summing operation for neural networks of equal depth (see, e.g., \citet[Lemma~3.2]{GS20_925}) it follows that 
$\eta_{{\varepsilon},d,t,i}$ can be realized by a deep neural network with 
\[
\begin{aligned}
\mathrm{size}(\eta_{\varepsilon,d,t,i}) & \leq 2\left(2 + \mathrm{size}(\mu_{\varepsilon,d,t,i})+\sum_{j=1}^d\mathfrak{s}_{i,j}\right)+4 \max(\mathrm{size}(\mu_{\varepsilon,d,t,i}),\mathfrak{s}_{i,1},\ldots,\mathfrak{s}_{i,d})
\\ 
& 
\leq 12\left(1 + C d^q \varepsilon^{-\tilde{\alpha}} +d(c(\log(\varepsilon^{-1})+\log(M)+1)+2)\right). 
\end{aligned}
\]
Next, we use Assumption~\ref{ass:driftdiffusion} to estimate 
\[ 
\|\sigma_{\varepsilon,d,t}(x)\|_F  \leq \|\sigma_{\varepsilon,d,t}(x)-\sigma^d(\mathfrak{t}_t,x)\|_F + \|\sigma^d(\mathfrak{t}_t,x)\|_F \leq  2 C d^q (1 + \|x\|)	
\]
and thus for $x \in [-(\varepsilon^{-\beta}),\varepsilon^{-\beta}]^d$ it follows that 
$\|\sigma_{\varepsilon,d,t}(x)\|_F \leq 2 C d^q(1+\sqrt{d}{\varepsilon^{-\beta}}) \leq M$. 
Hence, Assumption~\ref{ass:driftdiffusion} and \eqref{eq:productApprox} imply for $x,y \in [-(\varepsilon^{-\beta}),\varepsilon^{-\beta}]^d$ that 
\[
\begin{aligned}
& \| f^d_{t}(x,y) - \eta_{\varepsilon,d,t}(x,y)\| 
\\ & \quad \leq \| \mu^d(\mathfrak{t}_t,x)-\mu_{\varepsilon,d,t}(x)\| (\mathfrak{t}_{t+1}-\mathfrak{t}_t) + \left(\sum_{i=1}^d \left|(\sigma^d(\mathfrak{t}_t,x)y)_i - \sum_{j=1}^d {\mathfrak{n}}_{\varepsilon,M}(\sigma_{\varepsilon,d,t,i,j}(x),y_j)\right|^2\right)^{1/2}
\\ & \quad \leq \varepsilon C \bar{T} d^q (1+\|x\|)  + \|\sigma^d(\mathfrak{t}_t,x)y-\sigma_{\varepsilon,d,t}(x)y\| 
\\ & \quad \quad + \left(\sum_{i=1}^d \left|\sum_{j=1}^d \sigma_{\varepsilon,d,t,i,j}(x)y_j - {\mathfrak{n}}_{\varepsilon,M}(\sigma_{\varepsilon,d,t,i,j}(x),y_j)\right|^2\right)^{1/2}
\\ & \quad \leq \varepsilon C \bar{T} d^q (1+\|x\|)  + \|\sigma^d(\mathfrak{t}_t,x)-\sigma_{\varepsilon,d,t}(x)\|_F \|y\| + d^{\frac{3}{2}} \varepsilon
\\ & \quad \leq \varepsilon (C \bar{T} d^q + d^{\frac{3}{2}}) (1+\|x\|)  + \varepsilon C  d^q (1+\|x\|) \|y\| 
\\ & \quad \leq 2 \varepsilon (C (\bar{T}+1) d^q + d^{\frac{3}{2}}) (1+\|x\|^2+\|y\|^2).  
\end{aligned}
\]
Furthermore, Assumption~\ref{ass:driftdiffusion} and the Lipschitz property of ${\mathfrak{n}}_{\varepsilon,M}$ yield for all $x,x',y,y' \in \R^d$
\[
\begin{aligned}
\| \eta_{\varepsilon,d,t}(x,y) - \eta_{\varepsilon,d,t}(x',y')\| & \leq \|x-x'\| +\|\mu_{\varepsilon,d,t}(x)-\mu_{\varepsilon,d,t}(x')\| (\mathfrak{t}_{t+1}-\mathfrak{t}_t) 
\\ & \quad + \left(\sum_{i=1}^d \left|\sum_{j=1}^d {\mathfrak{n}}_{\varepsilon,M}(\sigma_{\varepsilon,d,t,i,j}(x'),y_j') - {\mathfrak{n}}_{\varepsilon,M}(\sigma_{\varepsilon,d,t,i,j}(x),y_j)\right|^2\right)^{\frac{1}{2}}
\\
& \leq (1+C d^q \varepsilon^{-\tilde{\zeta}} \bar{T}) \|x-x'\|
\\ & \quad + \left(\sum_{i=1}^d \left|\sum_{j=1}^d Mc(|\sigma_{\varepsilon,d,t,i,j}(x')-\sigma_{\varepsilon,d,t,i,j}(x)| +|y_j-y_j'|) \right|^2\right)^{\frac{1}{2}}
\\ & \leq 8 \max(C,1)^2 d^{2q+2} \varepsilon^{-\beta-\tilde{\zeta}} c(1+C \bar{T})( \|x-x'\| +\|y-y'\|).
\end{aligned}
\]
Finally, for all $x,y \in \R^d$
\[
\begin{aligned}
\E[\|f_t^d(x,Y_t^d)\|^{2\bar{p}}]^{\frac{1}{2\bar{p}}} & \leq \|x\| + \|\mu^d(\mathfrak{t}_t,x)\|\bar{T} + \E[\|\sigma^d(\mathfrak{t}_t,x)Y_t^d\|^{2\bar{p}}]^{\frac{1}{2\bar{p}}}
\\ & \leq (1+C\bar{T}d^q)(1+\|x\|) + Cd^q(1+\|x\|) \E[\|Y_t^d\|^{2\bar{p}}]^{\frac{1}{2\bar{p}}}
\end{aligned}
\]
so that \eqref{eq:auxEq59} implies a polynomial growth bound \eqref{eq:fGrowthbdd} and the estimate
\[
\begin{aligned}
\|\eta_{\varepsilon,d,t}(x,y)\| &  = \| \eta_{\varepsilon,d,t}(x,y)-\eta_{\varepsilon,d,t}(0,0)\|+\|\eta_{\varepsilon,d,t}(0,0)-f_t^d(0,0)\| + \|f_t^d(0,0)\| 
\end{aligned}
\]
combined with the Lipschitz, growth and approximation properties that we already established implies a polynomial bound \eqref{eq:etaGrowthbdd}
with $\theta = \beta + \tilde{\zeta}$. 

Altogether, this proves that Assumption~\ref{ass:Markov2} is satisfied with the claimed choices of $\zeta$ and $\theta$.
\end{proof}

\begin{proof}[Proof of Proposition~\ref{prop:max}]
 
Consider first the case of the running minimum. For $z \in \R^d$ partition $z = (z_{1:d-1},z_d)$ into the first $d-1$ and the last component. Define the transition map for the augmented process, $\bar{f}^d_t \colon \R^{d} \times \R^{d} \to \R^{d}$, by 
	\[\bar{f}^d_t(x,y)=(f^{d-1}_t(x_{1:d-1},y_{1:d-1}),\min(\min_{j=1,\ldots,d-1} f^{d-1}_{t,j}(x_{1:d-1},y_{1:d-1}),x_d))
	\] and $\bar{Y}^d_t = (Y^{d-1}_{t},0)$.
	Then $\bar{X}^{d}_0 =(X^{d-1}_0,\min_{i=1,\ldots,d-1}X^{d-1}_{0,i}) $ and so the independence and moment conditions on $\bar{Y}^d$ are satisfied  and $\|\bar{f}_t^d(0,0)\| \leq  2\|f^{d-1}_t(0,0)\|$.
	Thus, (i), (ii) and (iv) in Assumption~\ref{ass:Markov} are satisfied.
	
	Furthermore, by the identity $x=x^+ - (-x)^+$ and  \cite[Lemma~4.12]{HornungJentzen2018} the function $\mathfrak{min}_k \colon \R^{k} \to \R$, $z \mapsto \min_{j=1,\ldots,k} z_j$ can be realized by a deep neural network with size at most $12k^3$. We now set 
	\[ 
	\bar{\eta}_{\varepsilon,d,t}(x,y) = (\eta_{{\varepsilon},d-1,t}(x_{1:d-1},y_{1:d-1}),\mathfrak{min}_{2}(\mathfrak{min}_{d-1}(\eta_{{\varepsilon},d-1,t}(x_{1:d-1},y_{1:d-1})),x_d)).
	\]
	Then the $1$-Lipschitz property of $\mathfrak{min}_k$, which follows from the fact that the pointwise minimum of $1$-Lipschitz functions is again $1$-Lipschitz, implies that $\mathrm{Lip}(\bar{\eta}_{\varepsilon,d,t}) \leq \sqrt{2} \mathrm{Lip}({\eta}_{\varepsilon,d-1,t})$ and 
	$\| \bar{f}^d_{t}(x,y) - \bar{\eta}_{\varepsilon,d,t}(x,y)\| \leq \sqrt{2}\| f^{d-1}_{t}(x_{1:d-1},y_{1:d-1}) - \eta_{\varepsilon,d-1,t}(x_{1:d-1},y_{1:d-1})\|$. The bound on $\mathrm{size}(\bar{\eta}_{\varepsilon,d,t})$ follows from the bound on $\mathrm{size}(\eta_{\varepsilon,d-1,t})$ and bounds for the operations composition, parallelization and the realization of the identity (which yields a bound for the size of the neural network realizing $x \mapsto x_{1:d}$). Finally, $\|\bar{f}_t^d(x,y)\| \leq \sqrt{2} \|f_t^{d-1}(x_{1:d-1},y_{1:d-1})\|$ and  $\|\bar{\eta}_{\varepsilon,d,t}(x,y)\| \leq \sqrt{2} \|\eta_{\varepsilon,d-1,t}(x_{1:d-1},y_{1:d-1})\|$ so that all the required bounds follow from the corresponding properties of $X^{d-1}$. 
	
	In the case of the running maximum one proceeds analogously, except that the growth bounds are now a bit different, namely, $\|\bar{f}_t^d(x,y)\| \leq d \|f_t^{d-1}(x_{1:d-1},y_{1:d-1})\|+\|x\|$ and  $\|\bar{\eta}_{\varepsilon,d,t}(x,y)\| \leq d \|\eta_{\varepsilon,d-1,t}(x_{1:d-1},y_{1:d-1})\|+\|x\|$, which still allows us to deduce the claimed statement. 
\end{proof}	

\subsection{Auxiliary results}
\label{subsec:auxiliary}
This section contains auxiliary results that are needed for the proof of Theorems~\ref{thm:DNNapprox} and \ref{thm:refined}. We start with Lemma~\ref{lem:ggrowth}, which establishes growth properties of the payoff function and its neural network approximation. 

\begin{lemma}\label{lem:ggrowth}
Suppose Assumption~\ref{ass:Payoff} is satisfied. Then for all $\varepsilon \in (0,1]$, $d \in \N$, $t \in \{0,\ldots,T\}$, $x \in \R^d$ it holds that
\begin{align}
\label{eq:gGrowth}
|g_d(t,x)| & \leq  c d^{q} (1+\|x\|),
\\ \label{eq:phiGrowth}
|\phi_{\varepsilon,d,t}(x)| & \leq  c d^{q} (2+\|x\|).
\end{align}
\end{lemma}
\begin{proof}
First note that from 	\eqref{eq:NNclose}, \eqref{eq:NNlipschitz} and the growth assumption on $g_d$ we obtain for every $\bar{\varepsilon} \in (0,1]$ that
\begin{equation}
\label{eq:auxEq2}
\begin{aligned}
|g_d(t,x)| & \leq |g_d(t,x)-\phi_{\bar{\varepsilon},d,t}(x)| + |\phi_{\bar{\varepsilon},d,t}(x)-\phi_{\bar{\varepsilon},d,t}(0)|+|\phi_{\bar{\varepsilon},d,t}(0)-g_d(t,0)|+|g_d(t,0)|
\\
& \leq  \bar{\varepsilon} c d^{q} (2+\|x\|^p) + c d^{q} (1+\|x\|).
\end{aligned}
\end{equation}
Letting $\bar{\varepsilon}$ tend to $0$ we thus obtain \eqref{eq:gGrowth}. 
Next, note that 
the same properties of $g_d$ and $\phi_{{\varepsilon},d,t}$ imply 
		\begin{equation}
		\label{eq:phiGrowth1}
		\begin{aligned}
		|\phi_{\varepsilon,d,t}(x)| & \leq  |\phi_{\varepsilon,d,t}(x)-\phi_{\varepsilon,d,t}(0)|+|\phi_{\varepsilon,d,t}(0)-g_d(t,0)|+|g_d(t,0)| \leq   c d^{q} (2+\|x\|).
		\end{aligned}
		\end{equation}
\end{proof}

The next result, Lemma~\ref{lem:growth}, establishes growth properties of the Markov update function and its neural network approximation.  

\begin{lemma}\label{lem:growth}
Suppose Assumption~\ref{ass:Markov} is satisfied. Then for all $\varepsilon \in (0,1]$, $d \in \N$, $t \in \{0,\ldots,T-1\}$, $x,y \in \R^d$ it holds that
\begin{align}
\label{eq:fGrowth}
\|f_t^d(x,y)\| & \leq  c d^{q} (1+\|x\|+\|y\|),
\\ \label{eq:etaGrowth}
\|\eta_{\varepsilon,d,t}(x,y)\| & \leq   c d^{q} (2+\|x\|+\|y\|).
\end{align}
\end{lemma}
\begin{proof}
The proof is a straightforward consequence of
\eqref{eq:fapprox}, Assumption~\ref{ass:Markov}(iv) and \eqref{eq:etaLipschitz}. Indeed, these hypotheses imply for every $\bar{\varepsilon} \in (0,1]$ that
\begin{equation}
\label{eq:auxEq6}
\begin{aligned}
\|f_t^d(x,y)\| & \leq \|f_t^d(x,y)-\eta_{\bar{\varepsilon},d,t}(x,y)\| + \|\eta_{\bar{\varepsilon},d,t}(x,y)-\eta_{\bar{\varepsilon},d,t}(0,0)\|
\\ & \quad \quad +\|\eta_{\bar{\varepsilon},d,t}(0,0)-f_t^d(0,0)\|+\|f_t^d(0,0)\|
\\
& \leq  \bar{\varepsilon} c d^{q} (2+\|x\|^p + \|y\|^p) + c d^{q} (1+ \|x\| + \|y\|).
\end{aligned}
\end{equation}
Letting $\bar{\varepsilon}$ tend to $0$ we thus obtain \eqref{eq:fGrowth}. 
In addition, the same hypotheses yield
\begin{equation}
\label{eq:etaGrowth1}
\begin{aligned}
\|\eta_{\varepsilon,d,t}(x,y)\| & \leq  \|\eta_{{\varepsilon},d,t}(x,y)-\eta_{{\varepsilon},d,t}(0,0)\|+\|\eta_{{\varepsilon},d,t}(0,0)-f_t^d(0,0)\|+\|f_t^d(0,0)\| \\ & \leq   c d^{q} (2+\|x\|+\|y\|).
\end{aligned}
\end{equation}
\end{proof}

In the next lemma we establish a bound on the conditional moments of $X^d$. The proof and $\E[\|X_0^d\|]<\infty$ also yield $\E[\|X_t^d\|]<\infty$ for all $t$ and so  we may consider the conditional expectation in \eqref{eq:condMomentGrowth2} to be defined for all $x \in \R^d$, cf. Remark~\ref{rmk:condExpDefined} above.
\begin{lemma}
\label{lem:growthMoments}
Suppose Assumption~\ref{ass:Markov} or \ref{ass:Markov2} is satisfied. Then for all $d \in \N$, $x \in \R^d$, $s,t \in \{0,\ldots,T\}$ with $s \geq t$ it holds that
\begin{equation}
\begin{aligned}
\label{eq:condMomentGrowth2}
\E[\|X^d_s\| |X^d_t=x]  \leq  \tilde{c}_1 d^{\tilde{q}_1} (1+\|x\|) 
\end{aligned}
\end{equation}
with $\tilde{c}_1 = 2\max(c,1)^{T+1} T$, $\tilde{q}_1 = q(T+1)$ in the case of Assumption~\ref{ass:Markov} and with $\tilde{c}_1 = T\max(h,1)^{\frac{T}{2m\max(p,2)}}$, $\tilde{q}_1 = \frac{\bar{q}T}{2m\max(p,2)}$ in the case of Assumption~\ref{ass:Markov2}.  
\end{lemma}
\begin{proof}
Assume w.l.o.g.\ that $c \geq 1$. Consider first the case when Assumption~\ref{ass:Markov} holds. Then \eqref{eq:fGrowth} can be used to prove inductively that for all $s \geq t$, 
\begin{equation}
\label{eq:condMomentGrowth}
\E[\|X^d_s\| |X^d_t=x] \leq (c d^q)^{s-t} \|x\| + \sum_{k=1}^{s-t} (c d^{q})^{k} (1+\E[\|Y_{s-k}^d\|]).
\end{equation}
Indeed, for $s=t$ this directly follows from the definition. Assume now $s>t$ and \eqref{eq:condMomentGrowth} holds for $s-1,s-2,\ldots,t$, then \eqref{eq:fGrowth}, the induction hypothesis and independence yield
\begin{equation}
\begin{aligned}
\label{eq:auxEq24}
\E[\|X^d_s\| |X^d_t=x] &  = \E[\|f_{s-1}^d(X^d_{s-1},Y_{s-1}^d)\| |X^d_t=x] 
\\
& \leq c d^{q} (1+\E[\|X^d_{s-1}\| |X^d_t=x] +\E[\|Y_{s-1}^d\||X^d_t=x])
\\
& \leq c d^{q} \left(1+(c d^q)^{s-1-t} \|x\| + \sum_{k=1}^{s-1-t} (c d^{q})^{k} (1+\E[\|Y_{s-1-k}^d\|]) +\E[\|Y_{s-1}^d\|]\right)
\\
& =  c d^{q} (1+\E[\|Y_{s-1}^d\|]) + (c d^q)^{s-t} \|x\| +  \sum_{k=2}^{s-t} (c d^{q})^{k} (1+\E[\|Y_{s-k}^d\|]),
\end{aligned}
\end{equation}
as claimed. This shows that \eqref{eq:condMomentGrowth} holds for all $s\geq t$.   From \eqref{eq:condMomentGrowth} and Assumption~\ref{ass:Markov}(iv) we obtain
\begin{equation}
\begin{aligned}
\label{eq:condMomentGrowth22}
\E[\|X^d_s\| |X^d_t=x] & \leq c^T d^{Tq} \|x\| + \sum_{k=1}^{s-t} (c d^{q})^{k} (1+cd^q)
\\ & \leq  \tilde{c}_1 d^{\tilde{q}_1} (1+\|x\|).
\end{aligned}
\end{equation}
In the case when Assumption~\ref{ass:Markov2} holds we first note that independence, Jensen's inequality and \eqref{eq:fGrowthbdd} yield
\begin{equation}
\begin{aligned}
\label{eq:auxEq58}
\E[\|f_{s-1}^d(X^d_{s-1},Y_{s-1}^d)\| |X^d_t=x]  & = \int_{\R^d} \E[\|f_{s-1}^d(z,Y_{s-1}^d)\|] \mu_{t,s-1}^d(x,dz)
\\
 & \leq \int_{\R^d} (h d^{\bar{q}} (1+\|z\|^{2m\max(p,2)}))^{\frac{1}{2m\max(p,2)}} \mu_{t,s-1}^d(x,dz)
 \\
 & \leq (h d^{\bar{q}})^{\frac{1}{2m\max(p,2)}} \int_{\R^d} (1+\|z\|) \mu_{t,s-1}^d(x,dz)
 \\
& = (h d^{\bar{q}})^{\frac{1}{2m\max(p,2)}} (1+\E[\|X^d_{s-1}\| |X^d_t=x]).
\end{aligned}
\end{equation}
We can now apply this estimate instead of \eqref{eq:fGrowth} to get from the first to the second line in \eqref{eq:auxEq24} and arrive at 
\begin{equation}
\label{eq:condMomentGrowthAlternative}
\E[\|X^d_s\| |X^d_t=x] \leq ((h d^{\bar{q}})^{\frac{1}{2m\max(p,2)}})^{s-t} \|x\| + \sum_{k=1}^{s-t} ((h d^{\bar{q}})^{\frac{1}{2m\max(p,2)}})^{k},
\end{equation}
hence the conclusion follows. 
\end{proof}
The next result ensures that the optimal value \eqref{eq:optimalStoppingProblem} is finite in our setting. 
\begin{lemma}
\label{lem:expFinite}
Suppose Assumption~\ref{ass:Payoff} holds and Assumption~\ref{ass:Markov} or \ref{ass:Markov2} is satisfied. Then $\E[|g_d(t,X_t^d)|]< \infty$ for all $d \in \N$, $t \in \{0,\ldots,T\}$. 
\end{lemma}
\begin{proof}
Let  $d \in \N$, $t \in \{0,\ldots,T\}$. Then Lemma~\ref{lem:ggrowth},  Lemma~\ref{lem:growthMoments} and Assumption~\ref{ass:Markov} or \ref{ass:Markov2} ensure that
\[
\begin{aligned}
\label{eq:gGrowthExp}
\E[|g_d(t,X_t^d)|] & \leq  c d^{q} (1+\E[\|X_t^d\|]) =  c d^{q} (1+\E[\E[\|X_t^d\||X_0^d]])
\\ & \leq c d^{q} (1+\tilde{c}_1 d^{\tilde{q}_1} (1+\E[\|X_0^d\|]) ) < \infty.
\end{aligned}
\]
\end{proof}

The next lemma proves that the value function grows at most linearly. Recall from Remark~\ref{rmk:condExpDefined} that Lemma~\ref{lem:expFinite} allows us to recursively define the value function for all $x \in \R^d$ as the right hand side of \eqref{eq:recursionValueFunction}. 
\begin{lemma}
\label{lem:Vgrowth}
Suppose Assumption~\ref{ass:Payoff} holds and Assumption~\ref{ass:Markov} or \ref{ass:Markov2} is satisfied.
Then for all $d \in \N$, $t \in \{0,\ldots,T\}$, $x \in \R^d$ it holds that
\begin{align}
\label{eq:VGrowthLemma}
|V_d(t,x)| & \leq  \hat{c}_t d^{\hat{q}_t} (1+\|x\|),
\end{align}
where $\hat{c}_t = \max(c,1) (3\max(c,1)^2)^{T-t}$, $\hat{q}_t = q + 2q(T-t)$ in the case of Assumption~\ref{ass:Markov} and $\hat{c}_t = c (T+1) \max(h,1)^{\frac{T-t}{2m\max(p,2)}}$, $\hat{q}_t = q + \frac{\bar{q}(T-t)}{2m\max(p,2)}$ in the case of Assumption~\ref{ass:Markov2}.
\end{lemma}
\begin{proof} Consider first the case of Assumption~\ref{ass:Markov}.
The proof proceeds by backward induction. For $t=T$ the statement directly follows from \eqref{eq:gGrowth}. Assume now the statement holds for $t+1$. Then \eqref{eq:recursionValueFunction}, \eqref{eq:gGrowth}, the induction hypothesis and \eqref{eq:condMomentGrowth} yield
\begin{equation}
\label{eq:recursionValueFunctionLemma}
\begin{aligned} |V_d(t,x)| & \leq \max(|g_d(t,x)|,\E[|V_d(t+1,X_{t+1}^d)||X_t^d=x])
\\ &  \leq \max(c d^{q} (1+\|x\|),\hat{c}_{t+1} d^{\hat{q}_{t+1}} (1+\E[\|X_{t+1}^d\||X_t^d=x]))
\\ & \leq  \max(c d^{q} (1+\|x\|),\hat{c}_{t+1} d^{\hat{q}_{t+1}} (1+cd^q(1+\|x\|+cd^q)))
\\ & \leq  \hat{c}_{t+1} d^{\hat{q}_{t+1}} 3\max(c,1)^2d^{2q}(1+\|x\|).
\end{aligned}
\end{equation}
Hence, \eqref{eq:VGrowthLemma} also holds at $t$ and so by induction the statement follows. 

In the case of Assumption~\ref{ass:Markov2} we aim to provide a tighter estimate and instead inductively prove that $|V_d(t,x)| \leq \hat{a}_{t} + \hat{b}_t\|x\|$ with $\hat{a}_t = \hat{a}_{t+1} + \hat{b}_{t+1} (h d^{\bar{q}})^{\frac{1}{2m\max(p,2)}}$, $\hat{a}_T = c d^q$, $\hat{b}_t = \hat{b}_{t+1} (\max(h,1) d^{\bar{q}})^{\frac{1}{2m\max(p,2)}}$,  $\hat{b}_T = c d^q$.
Indeed,  using \eqref{eq:condMomentGrowthAlternative} instead of \eqref{eq:condMomentGrowth}  we analogously obtain 
\begin{equation} 
\label{eq:recursionValueFunctionLemma2}
\begin{aligned} |V_d(t,x)|  & \leq  \max(c d^{q} + c d^{q}\|x\| ,\hat{a}_{t+1} + \hat{b}_{t+1} (h d^{\bar{q}})^{\frac{1}{2m\max(p,2)}}(1+\|x\|)) \leq \hat{a}_{t} + \hat{b}_t\|x\|
\end{aligned}
\end{equation}
from which the statement follows by noting that $\hat{b}_t = cd^q (\max(h,1) d^{\bar{q}})^{\frac{T-t}{2m\max(p,2)}}$,
\[
\hat{a}_t = cd^q + \sum_{s=t+1}^T \hat{b}_{s} (h d^{\bar{q}})^{\frac{1}{2m\max(p,2)}} \leq  cd^q (T+1) (\max(h,1) d^{\bar{q}})^{\frac{T-t}{2m\max(p,2)}}. 
\]
\end{proof}

Lemma~\ref{lem:nnfixed} mathematically proves the intuitively obvious fact that a neural network in which some input arguments are held
at fixed values is still a neural network with at most as many non-zero parameters as the original neural network. 
\begin{lemma}
	\label{lem:nnfixed}	
	Let $d_0,d_1,m \in \N$ and let $\phi \colon \R^{d_0+d_1} \to \R^m$ be a neural network. Let $y \in \R^{d_1}$. Then $\Phi_y \colon \R^{d_0} \to \R^m$, $x \mapsto \phi((x,y))$ can again be realized by a neural network $\phi_y$ with $\mathrm{size}(\phi_y) \leq  \mathrm{size}(\phi)$.
\end{lemma}
\begin{proof}
	Denote by $((A^1,b^1),\ldots,(A^L,b^L))$ the parameters of $\phi$ with $L \in \N$, $N_0= d_0+d_1$, $N_1,\ldots,N_{L-1} \in \N$, $N_L=m$ and $A^{\ell} \in \R^{N_\ell \times N_{\ell -1}}$, $b^\ell \in \R^{N_\ell}$, $\ell=1,\ldots,L$. Denote by $A^{1,0} \in \R^{N_1 \times d_0}$ and $A^{1,1} \in \R^{N_1 \times d_1}$ the first $d_0$ and the remaining $d_1$ columns of $A^1$, respectively. Consider the neural network $\phi_y$ with parameters  $((A^{1,0},A^{1,1}y+b^1),(A^2,b^2),\ldots,(A^{L},b^L))$. Then 
	\begin{equation} 
	\label{eq:auxEq41}
	\begin{aligned}
	\phi_y(x) & = (\mathcal{W}_L \circ (\varrho \circ \mathcal{W}_{L-1}) \circ \cdots \circ (\varrho \circ \mathcal{W}_2) \circ \varrho)(A^{1,0}x + A^{1,1}y+b^1)
	\\ & = (\mathcal{W}_L \circ (\varrho \circ \mathcal{W}_{L-1}) \circ \cdots \circ (\varrho \circ \mathcal{W}_1))((x,y))
	\\ & = \Phi_y(x)
	\end{aligned}
	\end{equation}
	for all $x \in \R^{d_0}$ and $\mathrm{size}(\phi_y) \leq  \mathrm{size}(\phi)$, as claimed. 
\end{proof}

The next lemma will allow us to construct a realization of a random neural network and at the same time obtain a bound on the neural network weights. 

\begin{lemma}\label{lem:jointProb}
	Let $U$ be a nonnegative random variable, let $d, N \in \N$, let $M_1,M_2 >0$ and let $Y_1,\ldots,Y_N$ be i.i.d.\ $\R^d$-valued random variables. Suppose $\E[U] \leq M_1$ and $\E[\|Y_1\|]\leq M_2$. Then 
	\[
	\P\left(U\leq 3 M_1, \max_{i=1,\ldots,N} \|Y_i\| \leq 3 N M_2\right) > 0.
	\]
\end{lemma}
\begin{proof}
	Firstly, by the i.i.d.\ assumption it follows that
	\begin{equation}
	\label{eq:auxEq34}
	\begin{aligned}
	\P\left(\max_{i=1,\ldots,N} \|Y_i\| > 3 N M_2\right) & = 1- \left(\P\left(\|Y_1\| \leq 3 N M_2\right)\right)^N
	\\ & = 1- (1-\P\left(\|Y_1\| > 3 N M_2\right))^N.
	\end{aligned}
	\end{equation}
	Next, note that Bernoulli's inequality implies $(\frac{2}{3})^{1/N} \leq 1 - \frac{1}{3N}$ and therefore, by Markov's inequality, 
	\[
	\P\left(\|Y_1\| > 3 N M_2\right) \leq \frac{\E[\|Y_1\|]}{3 N M_2} \leq \frac{1}{3 N} \leq 1- \left(\frac{2}{3}\right)^{\frac{1}{N}}. 
	\]
	Thus, we obtain $(1-\P(\|Y_1\| > 3 N M_2))^N \geq \frac{2}{3}$ and inserting this in \eqref{eq:auxEq34} yields
	\begin{equation}
	\label{eq:auxEq35}
	\begin{aligned}
	\P\left(\max_{i=1,\ldots,N} \|Y_i\| > 3 N M_2\right) & \leq \frac{1}{3}.
	\end{aligned}
	\end{equation}
	Furthermore, Markov's inequality implies that
	\begin{equation}
	\label{eq:auxEq36}
	\begin{aligned}
	\P\left(U > 3 M_1\right) & \leq \frac{\E[U]}{3M_1} \leq \frac{1}{3}.
	\end{aligned}
	\end{equation}
	Combining \eqref{eq:auxEq35} and \eqref{eq:auxEq36} with the elementary fact that $\P(A\cap B) = \P(A) + \P(B) - \P(A \cup B) \geq \P(A) + \P(B) -1$ for $A, B \in \Fc$ then shows that
	\[
	\P\left(U\leq 3 M_1, \max_{i=1,\ldots,N} \|Y_i\| \leq 3 N M_2\right) \geq \frac{2}{3} + \frac{2}{3} -1 > 0,
	\]
	as claimed.
\end{proof}

\subsection{Proof of Theorem~\ref{thm:DNNapprox} and Corollary~\ref{cor:Lipschitz}}
\label{subsec:DNNapproxProof}
With these preparations we are now ready to prove Theorem~\ref{thm:DNNapprox}.
The proof is divided into several steps, which are highlighted in bold in order to facilitate reading. Let us first provide a brief sketch of the proof. The proof proceeds by backward induction. This entails some subtleties regarding the probability measure $\rho^d$, which we will not discuss here. We refer to the proof below for details. Here we rather provide an easy-to-follow overview. 

The starting point is the backward recursion \eqref{eq:recursionValueFunction}. Our goal is to provide a neural network approximation of the right hand side in \eqref{eq:recursionValueFunction}. At time $t$ we first aim to bound the $L^2(\rho^d)$-approximation error $E^d_t$ between the 
continuation value $\E[V_d(t+1,X_{t+1}^d)|X_t^d=x]$ and the random function $\Gamma_{\varepsilon,d,t}(x) = \frac{1}{N}\sum_{i=1}^N \hat{v}_{\varepsilon,d,t+1}(\eta_{\varepsilon,d,t}(x,Y^{d,i}_t))$, where  $\hat{v}_{\varepsilon,d,t+1}$ is a neural network approximating the value function at time $t+1$ and $Y^{d,1}_t, \ldots, Y^{d,N}_t$ are i.i.d.\ copies of $Y^{d}_t$. The existence and suitable properties of $\hat{v}_{\varepsilon,d,t+1}$ follow from the induction hypothesis. We derive a bound on $\E[E^d_t]$, which we can then use to obtain existence of a realization $\gamma_{\varepsilon,d,t}$ of $\Gamma_{\varepsilon,d,t}$ satisfying a slightly worse bound and such that the realization of $\max_{i=1,\ldots,N} \|Y^{d,i}_t\|$ can also be bounded suitably. This last point is necessary to control the growth of $\gamma_{\varepsilon,d,t}(x)$.
Then $\gamma_{\varepsilon,d,t}(x)$ is an approximation of the continuation value and so we naturally define the approximate value function at time $t$ by 
\begin{equation}
\begin{aligned}\label{eq:approxvalue}
\hat{v}_{\varepsilon,d,t}(x) = \max\left(\phi_{\varepsilon,d,t}(x) - \delta ,\gamma_{\varepsilon,d,t}(x) \right)
\end{aligned}
\end{equation}
for a suitably chosen $\delta$ (depending on $\varepsilon$). 
We then consider the continuation region  
\begin{equation}\label{eq:continuation} C_t = \{x \in \R^d \, \colon \, g_d(t,x) < \E[V_d(t+1,X_{t+1}^d)|X_t^d=x] \}
\end{equation}
and the approximate continuation region 
\begin{equation}\label{eq:approxcontinuation} \hat{C}_t = \{x \in \R^d \, \colon \, \phi_{{\varepsilon},d,t}(x) - \delta < \gamma_{{\varepsilon},d,t}(x) \}.
\end{equation}
Then we may decompose
\begin{equation}
\begin{aligned}
\label{eq:decomposition}
& |V_d(t,x)  -\hat{v}_{{\varepsilon},d,t}(x)| \\ &  = | g_d(t,x) -\phi_{{\varepsilon},d,t}(x) + \delta |\mathbbm{1}_{C_t^c \cap \hat{C}_t^c}(x)  + |\E[V_d(t+1,X_{t+1}^d)|X_t^d=x] - \phi_{{\varepsilon},d,t}(x) + \delta| \mathbbm{1}_{C_t \cap \hat{C}_t^c}(x) 
\\ & \quad +  | g_d(t,x) -\gamma_{{\varepsilon},d,t}(x) |\mathbbm{1}_{C_t^c \cap \hat{C}_t}(x) +  | \E[V_d(t+1,X_{t+1}^d)|X_t^d=x] -\gamma_{{\varepsilon},d,t}(x)  |\mathbbm{1}_{C_t \cap \hat{C}_t}(x).
\end{aligned}
\end{equation} 
The $L^2(\rho^d)$-error of the last term has already been analysed, and it remains to analyse the remaining terms. The first term is small due to Assumption~\ref{ass:Payoff}. The second and third term may not necessarily be small, but we will be able to show that $\rho^d(C_t \cap \hat{C}_t^c)$ and $\rho^d(C_t^c \cap \hat{C}_t)$ are small. Hence, the overall $L^2(\rho^d)$-error can be controlled. The proof is then completed by showing that the neural network \eqref{eq:approxvalue} satisfies the growth, size and Lipschitz properties required to carry out the induction argument.

\begin{proof}[Proof of Theorem~\ref{thm:DNNapprox}] 
$ $ \newline
\textbf{1. Preliminaries:}
Without loss of generality we may assume that the constants $c>0$, $q,\alpha \geq 0 $ in the statement of the theorem and in Assumptions~\ref{ass:Markov} and \ref{ass:Payoff} coincide; otherwise we replace each of them by the respective maximum and all the assumptions are still satisfied. We may also assume that $c \geq 1$. 

Furthermore,
if for each fixed $t \in \{0,\ldots,T\}$ there exist constants $\kappa_t,\mathfrak{q}_t,\mathfrak{r}_t \in [0,\infty)$ and a neural network $\psi_{\varepsilon,d,t}$ such that $\mathrm{size}(\psi_{\varepsilon,d,t}) \leq \kappa_t d^{\mathfrak{q}_t}\varepsilon^{-\mathfrak{r}_t}$  and \eqref{eq:approxError} holds for all $\varepsilon \in (0,1]$,  $d \in \N$, then also the statement of the theorem follows by choosing $\kappa$, $\mathfrak{q}$, $\mathfrak{r}$ as the respective maximum over $t \in \{0,\ldots,T\}$. 
	
Next, let $c_0,\ldots,c_T$ satisfy $c_0 = c$, $c_{t+1} =  \max(3c,1)^{2\max(p,2)}(1+c_t+c)$ and set $q_t = (2\max(p,2)+1)qt+q$. Then 
\[
\begin{aligned}
c_t & = c ( \max(3c,1)^{2\max(p,2)})^{t} + \sum_{k=0}^{t-1}( \max(3c,1)^{2\max(p,2)})^{k+1} (1+c) 
\end{aligned}
\]
for all $t \in \{0,\ldots,T\}$ and $c_t$ does not depend on $d$. 

\textbf{2. Stronger statement:} We will now proceed to prove the following stronger statement, which shows that the constants $\kappa_t,\mathfrak{q}_t,\mathfrak{r}_t$ can be chosen essentially independently of the probability measure $\rho^d$ and, in addition, $\rho^d$ may be allowed to depend on $t$. Specifically,  we will prove that for any $t \in \{0,\ldots,T\}$ there exist constants $\kappa_t,\mathfrak{q}_t,\mathfrak{r}_t \in [0,\infty)$ such 
 that
 for any family of probability measures $\rho_t^d$ on $\R^d$, $d \in \N$, satisfying 
\begin{equation}
\label{eq:CondMomentIntegrable}
\int_{\R^d} \|x\|^{2\max(p,2)} \rho^d_t(dx) \leq c_t d^{q_t}
\end{equation}
and for all $d \in \N$, $\varepsilon \in (0,1]$ there exists a neural network $\psi_{\varepsilon,d,t}$ such that
 \begin{equation}
 \label{eq:approxErrort}
 \left(\int_{\R^d} |V_d(t,x) - \psi_{\varepsilon,d,t}(x)|^2 \rho^d_t(d x) \right)^{1/2} \leq \varepsilon
 \end{equation}
 and  
\begin{align} \label{eq:NNgrowth}
|\psi_{\varepsilon,d,t}(x)| 
& \leq \kappa_t d^{\mathfrak{q}_t} \varepsilon^{-\mathfrak{r}_t} (1+\|x\|), \quad \text{ for all } x \in \R^d,
\\ \label{eq:NNsparse1} 
\mathrm{size}(\psi_{\varepsilon,d,t}) &  \leq \kappa_t d^{\mathfrak{q}_t}\varepsilon^{-\mathfrak{r}_t}, 
\\ \label{eq:NNlipschitz1} 
\mathrm{Lip}(\psi_{\varepsilon,d,t}) & \leq \kappa_t d^{\mathfrak{q}_t}.
\end{align}
Choosing $\rho_t^d = \rho^d$ for all $t$ and noting that \eqref{eq:CondMomentIntegrable} is satisfied due to $q \leq q_t$, $c \leq c_t$ the statement of Theorem~\ref{thm:DNNapprox} then follows.

In order to prove the stronger statement for each fixed $t$, we now proceed by backward induction. 

\textbf{3. Base case of backward induction:} In the case $t=T$ we have $V_d(T,x) = g_d(T,x)$ and therefore we may choose $\psi_{\varepsilon,d,T} = \phi_{\tilde{\varepsilon},d,T}$ with $\tilde{\varepsilon} = \varepsilon [c d^{q} (1+ (c_Td^{q_T} )^{\frac{p}{2\max(p,2)}})]^{-1}$. Then \eqref{eq:NNclose}, Jensen's inequality and \eqref{eq:CondMomentIntegrable} imply
\[\begin{aligned}
\left(\int_{\R^d} |g_d(T,x) - \psi_{\varepsilon,d,T}(x)|^2 \rho^d_T(d x) \right)^{1/2} & \leq \tilde{\varepsilon} c d^{q}\left(1+ \left(\int_{\R^d} \|x\|^{2p} \rho^d_T(d x) \right)^{1/2}\right)
\\ & \leq \tilde{\varepsilon} c d^{q} \left(1+ \left(\int_{\R^d} \|x\|^{2\max(p,2)} \rho^d_T(d x) \right)^{\frac{p}{2\max(p,2)}}\right)
\\ & \leq \tilde{\varepsilon} c d^{q} \left(1+ (c_Td^{q_T} )^{\frac{p}{2\max(p,2)}}\right) = \varepsilon.
\end{aligned}
\]
Furthermore, 
\eqref{eq:NNsparse} implies $\mathrm{size}(\psi_{\varepsilon,d,T}) \leq c d^{q} \varepsilon^{-\alpha} [c d^{q} (1+ (c_Td^{q_T} )^{\frac{p}{2\max(p,2)}})]^{\alpha}$ and so, recalling \eqref{eq:NNlipschitz} and \eqref{eq:phiGrowth}, the statement follows in the case $t=T$. 	

\textbf{4. Start of the induction step:}
The remainder of the proof will now be dedicated to the induction step. To improve readability we will again divide it into several steps. 

For the induction step we now assume that the stronger statement formulated in Step~2 above holds for time $t+1$ and aim to prove it for time $t$. To this end, let $\rho_t^d$ be a probability measure satisfying \eqref{eq:CondMomentIntegrable} and denote by $\nu^d_t$ the distribution of $Y_t^d$. 

\textbf{5. Induction hypothesis:}
Let $\kappa_{t+1},\mathfrak{q}_{t+1},\mathfrak{r}_{t+1} \in [0,\infty)$ denote the constants with which the stronger statement formulated in Step~2 above holds for time $t+1$.

Consider the probability measure $\rho_{t+1}^d = (\rho_t^d \otimes \nu_t^d) \circ (f_t^d)^{-1}$ given as the pushforward measure of $\rho_t^d \otimes \nu_t^d$ under $f_t^d$. Then, using the change-of-variables formula, \eqref{eq:fGrowth},  \eqref{eq:CondMomentIntegrable} and Assumption~\ref{ass:Markov}(iv) and writing $\bar{p}=2\max(p,2)$, this measure satisfies 
\begin{equation}
\label{eq:auxEq30}
\begin{aligned}
\int_{\R^d} \|z\|^{2\max(p,2)} \rho_{t+1}^d(dz) & = \int_{\R^d} \int_{\R^d} \|f_t^d(x,y)\|^{\bar{p}} \nu_{t}^d(dy) \rho^d_t(dx)
\\ & \leq \int_{\R^d} \int_{\R^d} (cd^q(1+\|x\|+\|y\|))^{\bar{p}} \nu_{t}^d(dy) \rho^d_t(dx)
\\ & \leq  (3cd^q)^{\bar{p}} \left( 1 + \int_{\R^d}\|x\|^{\bar{p}} \rho^d_t(dx) + \int_{\R^d}    \|y\|^{\bar{p}} \nu_{t}^d(dy) \right)
\\ & \leq (3cd^q)^{\bar{p}} \left( 1 + c_t d^{q_t} + c d^q \right)
\\ & \leq \max(3c,1)^{\bar{p}} d^{\bar{p}q+q_t+q} \left( 1 + c_t + c \right)
\\ & = c_{t+1} d^{q_{t+1}}.
\end{aligned}
\end{equation}
Hence, by induction hypothesis, for any $\varepsilon \in (0,1]$,  $d \in \N$ there exists a 
 neural network $\psi_{\varepsilon,d,t+1}$ such that
\begin{equation}
\label{eq:approxErrort+1}
\left(\int_{\R^d} |V_d(t+1,x) - \psi_{\varepsilon,d,t+1}(x)|^2 \rho^d_{t+1}(d x) \right)^{1/2} \leq \varepsilon
\end{equation}
and
\begin{align}
 \label{eq:NNgrowtht+1}
|\psi_{\varepsilon,d,t+1}(x)| 
& \leq \kappa_{t+1} d^{\mathfrak{q}_{t+1}} \varepsilon^{-\mathfrak{r}_{t+1}}(1+\|x\|), \quad \text{ for all } x \in \R^d,
\\ \label{eq:NNsparse1t+1} 
\mathrm{size}(\psi_{\varepsilon,d,t+1}) &  \leq \kappa_{t+1} d^{\mathfrak{q}_{t+1}}\varepsilon^{-\mathfrak{r}_{t+1}}, 
\\ \label{eq:NNlipschitz1t+1} 
\mathrm{Lip}(\psi_{\varepsilon,d,{t+1}}) & \leq \kappa_{t+1} d^{\mathfrak{q}_{t+1}}.
\end{align}

Now let $\varepsilon \in (0,1]$, $d \in \N$ be given. The remainder of the proof consists in selecting $\kappa_t,\mathfrak{q}_t,\mathfrak{r}_t$ (only depending on $c,\alpha,p,q,t,T,\kappa_{t+1},\mathfrak{q}_{t+1},\mathfrak{r}_{t+1}$) and constructing a neural network $\psi_{\varepsilon,d,t}$ such that \eqref{eq:approxErrort}--\eqref{eq:NNlipschitz1} are satisfied. This will complete the proof.  

In what follows we fix $\bar{\varepsilon} \in (0,1)$ and choose
\begin{equation}
\label{eq:NandDelaChoice} N
 =\lceil \bar{\varepsilon}^{-2 \mathfrak{r}_{t+1}-2} \rceil \qquad \text{ and } \qquad \delta  = \bar{\varepsilon}^{\frac{1}{2}}.
\end{equation}
The value of $\bar{\varepsilon}$ will be chosen later (depending on $\varepsilon$ and $d$).    

\textbf{6. Approximation of the continuation value:}
Let $Y^{d,i}_t$, $i \in \N$,  be i.i.d.\ copies of $Y^{d}_t$ and set 
$\hat{v}_{\bar{\varepsilon},d,t+1} = \psi_{\bar{\varepsilon},d,t+1} $. 
Define the (random) function \[\Gamma_{\bar{\varepsilon},d,t}(x) = \frac{1}{N}\sum_{i=1}^N \hat{v}_{\bar{\varepsilon},d,t+1}(\eta_{\bar{\varepsilon},d,t}(x,Y^{d,i}_t)).\]
Note that $\Gamma_{\bar{\varepsilon},d,t}$ is a random function, since $Y^{d,i}_t$ is random.


We now estimate the expected $L^2(\rho^d_t)$-error that arises when $\Gamma_{\bar{\varepsilon},d,t}$ is used to approximate the continuation value. 
Denote 
$Z^{\bar{\varepsilon},d,i}(x)= \hat{v}_{\bar{\varepsilon},d,t+1}(\eta_{\bar{\varepsilon},d,t}(x,Y^{d,i}_t))$ and recall that Assumption~\ref{ass:Markov}(i)--(ii) implies that $Y^{d}_t$ is independent of $X^d_t$.
 Then  $\Gamma_{\bar{\varepsilon},d,t}(x) = \frac{1}{N} \sum_{i=1}^N Z^{\bar{\varepsilon},d,i}(x)$ and thus the bias-variance decomposition and independence show
\begin{equation}
\label{eq:auxEq14}
\begin{aligned}
& \int_{\R^d}  \E[| \E[V_d(t+1,X_{t+1}^d)|X_t^d=x] -\Gamma_{\bar{\varepsilon},d,t}(x) |^2] \rho^d_t(dx)  
\\ &  = \int_{\R^d} | \E[V_d(t+1,X_{t+1}^d)|X_t^d=x] -\E[ Z^{\bar{\varepsilon},d,1}(x)]|^2 + \E[| \E[Z^{\bar{\varepsilon},d,1}(x)] -\Gamma_{\bar{\varepsilon},d,t}(x) |^2] \rho^d_t(dx)
\\ &  = \int_{\R^d} | \E[V_d(t+1,X_{t+1}^d)|X_t^d=x] -\E[ Z^{\bar{\varepsilon},d,1}(x)]|^2 + \frac{1}{N}\E[| \E[Z^{\bar{\varepsilon},d,1}(x)] -Z^{\bar{\varepsilon},d,1}(x) |^2] \rho^d_t(dx).
\end{aligned}
\end{equation}

The term corresponding to the first integral in the last line of \eqref{eq:auxEq14} can be estimated as 
\begin{equation}
\label{eq:auxEq25}
\begin{aligned}
\int_{\R^d} & | \E[V_d(t+1,X_{t+1}^d)|X_t^d=x] -\E[ \hat{v}_{\bar{\varepsilon},d,t+1}(\eta_{\bar{\varepsilon},d,t}(x,Y^{d,i}_t))]|^2 \rho^d_t(dx)
\\ & \leq 2 \int_{\R^d} | \E[V_d(t+1,X_{t+1}^d)|X_t^d=x] -\E[ \hat{v}_{\bar{\varepsilon},d,t+1}(X_{t+1}^d)|X_t^d=x]|^2 \rho^d_t(dx)
\\ & \quad + 2 \int_{\R^d} | \E[ \hat{v}_{\bar{\varepsilon},d,t+1}(X_{t+1}^d)|X_t^d=x]-\E[ \hat{v}_{\bar{\varepsilon},d,t+1}(\eta_{\bar{\varepsilon},d,t}(x,Y^{d}_t))]|^2 \rho^d_t(dx).
\end{aligned}
\end{equation}

\textbf{6.a) Applying the error estimate from $t+1$:}
Now consider the first term in the right hand side of \eqref{eq:auxEq25} and recall that $\rho_{t+1}^d = (\rho_t^d \otimes \nu_t^d) \circ (f_t^d)^{-1}$ is the pushforward measure of $\rho_t^d \otimes \nu_t^d$ under $f_t^d$. Then Jensen's inequality, \eqref{eq:MarkovUpdate}, Assumption~\ref{ass:Markov}(ii) and \eqref{eq:approxErrort+1} yield
\begin{equation}
\label{eq:auxEq26}
\begin{aligned}
\int_{\R^d} & | \E[V_d(t+1,X_{t+1}^d)|X_t^d=x] -\E[ \hat{v}_{\bar{\varepsilon},d,t+1}(X_{t+1}^d)|X_t^d=x]|^2 \rho^d_t(dx)
\\ 
& \leq \int_{\R^d}   \E[|V_d(t+1,f_t^d(x,Y_t^d)) - \hat{v}_{\bar{\varepsilon},d,t+1}(f_t^d(x,Y_t^d))|^2] \rho^d_t(dx)
\\ & = \int_{\R^d} \int_{\R^d} |V_d(t+1,f_t^d(x,y)) - \hat{v}_{\bar{\varepsilon},d,t+1}(f_t^d(x,y))|^2 \nu_t^d(dy) \rho^d_t(dx)
\\ & = \int_{\R^d} |V_d(t+1,z) - \hat{v}_{\bar{\varepsilon},d,t+1}(z)|^2 \rho_{t+1}^d(dz)
\\ & \leq \bar{\varepsilon}^2.
\end{aligned}
\end{equation}

\textbf{6.b) Applying the Lipschitz property of the network at $t+1$:}
For the second term in the right hand side of \eqref{eq:auxEq25}, note that by induction hypothesis \eqref{eq:NNlipschitz1t+1} we have $\mathrm{Lip}(\hat{v}_{\bar{\varepsilon},d,t+1})  \leq \kappa_{t+1} d^{\mathfrak{q}_{t+1}}$ and hence \eqref{eq:fapprox}, the assumption $\E[\|Y^{d}_t\|^p] \leq c d^q$ and \eqref{eq:CondMomentIntegrable} imply
\begin{equation}
\label{eq:auxEq27}
\begin{aligned}
\int_{\R^d} & | \E[ \hat{v}_{\bar{\varepsilon},d,t+1}(f^d_t(x,Y_t^d))]-\E[ \hat{v}_{\bar{\varepsilon},d,t+1}(\eta_{\bar{\varepsilon},d,t}(x,Y^{d}_t))]|^2 \rho^d_t(dx)
\\ & \leq (\kappa_{t+1} d^{\mathfrak{q}_{t+1}})^2 \int_{\R^d} | \E[ \| f^d_t(x,Y_t^d)- \eta_{\bar{\varepsilon},d,t}(x,Y^{d}_t)\|]|^2 \rho^d_t(dx)
\\ & \leq (\bar{\varepsilon} c \kappa_{t+1} d^{q+\mathfrak{q}_{t+1}})^2 \int_{\R^d} |1+\|x\|^p + \E[\|Y^{d}_t\|^p]|^2 \rho^d_t(dx)
\\ & \leq 3 (\bar{\varepsilon} c \kappa_{t+1} d^{q+\mathfrak{q}_{t+1}})^2  \left(1+ c^2 d^{2q} + \left(\int_{\R^d}\|x\|^{2\max(p,2)}  \rho^d_t(dx)\right)^{\frac{p}{\max(p,2)}} \right)
\\ & \leq 3 (\bar{\varepsilon} c \kappa_{t+1})^2 d^{2(q+\mathfrak{q}_{t+1})+\max(2q,\frac{pq_t}{\max(p,2)})}  \left(1+ c^2 + c_t^{\frac{p}{\max(p,2)}} \right).
\end{aligned}
\end{equation}
\textbf{6.c) Applying the growth property of the network at $t+1$:}
For the last term in \eqref{eq:auxEq14}, note that 
the induction hypothesis $|\hat{v}_{\bar{\varepsilon},d,t+1}(x)| 
\leq \kappa_{t+1} d^{\mathfrak{q}_{t+1}} \bar{\varepsilon}^{-\mathfrak{r}_{t+1}} (1+\|x\|)$, \eqref{eq:etaGrowth}, $\E[\|Y_{t}^d\|^{2}] \leq c d^q$, Hölder's inequality and \eqref{eq:CondMomentIntegrable} yield
\begin{equation}
\label{eq:auxEq28}
\begin{aligned}
\int_{\R^d} & \E[| \E[Z^{\bar{\varepsilon},d,1}(x)] -Z^{\bar{\varepsilon},d,1}(x) |^2] \rho^d_t(dx) \\ & \leq  \int_{\R^d}  \E[| Z^{\bar{\varepsilon},d,1}(x) |^2] \rho^d_t(dx) \\ & = \int_{\R^d}  \E[|\hat{v}_{\bar{\varepsilon},d,t+1}(\eta_{\bar{\varepsilon},d,t}(x,Y^{d}_t)) |^2] \rho^d_t(dx) 
\\ & \leq (\kappa_{t+1} d^{\mathfrak{q}_{t+1}}\bar{\varepsilon}^{-\mathfrak{r}_{t+1}})^2 \int_{\R^d}  \E[ (1+\|\eta_{\bar{\varepsilon},d,t}(x,Y^{d}_t)\|)^2] \rho^d_t(dx) 
\\ & \leq 2(\kappa_{t+1} d^{\mathfrak{q}_{t+1}} \bar{\varepsilon}^{-\mathfrak{r}_{t+1}})^2 \left(1 + \int_{\R^d}  \E[(c d^{q} (2+\|x\|+\|Y^{d}_t\|))^2] \rho^d_t(dx) \right)
\\ & \leq 2(\kappa_{t+1} d^{\mathfrak{q}_{t+1}+q}\bar{\varepsilon}^{-\mathfrak{r}_{t+1}})^2 \left(1 + 3 c^2 \left[ 4+cd^q+\left(\int_{\R^d} \|x\|^{2\max(p,2)} \rho^d_t(dx)\right)^{\frac{2}{2\max(p,2)}}\right] \right)
\\ & \leq 2 \kappa_{t+1}^2 d^{2\mathfrak{q}_{t+1}+2q+\max(q,\frac{q_t}{\max(p,2)})} \bar{\varepsilon}^{-2\mathfrak{r}_{t+1}} \left(1 + 3 c^2 \left[ 4+c+c_t^{\frac{1}{\max(p,2)}}\right] \right).
\end{aligned}
\end{equation}
\textbf{6.d) Bounding the overall error and constructing a realization:}
We can now insert the estimates from \eqref{eq:auxEq26} and \eqref{eq:auxEq27} into \eqref{eq:auxEq25} and subsequently insert the resulting bound and \eqref{eq:auxEq28} into \eqref{eq:auxEq14}. We obtain 
\begin{equation}
\label{eq:auxEq29}
\begin{aligned}
& \int_{\R^d}  \E[| \E[V_d(t+1,X_{t+1}^d)|X_t^d=x] -\Gamma_{\bar{\varepsilon},d,t}(x) |^2] \rho^d_t(dx)  
\\ & \leq 2 \int_{\R^d} | \E[V_d(t+1,X_{t+1}^d)|X_t^d=x] -\E[ \hat{v}_{\bar{\varepsilon},d,t+1}(X_{t+1}^d)|X_t^d=x]|^2 \rho^d_t(dx)
\\ & \quad + 2 \int_{\R^d} | \E[ \hat{v}_{\bar{\varepsilon},d,t+1}(X_{t+1}^d)|X_t^d=x]-\E[ \hat{v}_{\bar{\varepsilon},d,t+1}(\eta_{\bar{\varepsilon},d,t}(x,Y^{d}_t))]|^2 \rho^d_t(dx)
\\ & \quad + \frac{1}{N}\int_{\R^d} \E[| \E[Z^{\bar{\varepsilon},d,1}(x)] -Z^{\bar{\varepsilon},d,1}(x) |^2] \rho^d_t(dx)
\\ & \leq 2 \bar{\varepsilon}^2 + 6 (\bar{\varepsilon} c \kappa_{t+1})^2 d^{2(q+\mathfrak{q}_{t+1})+\max(2q,\frac{pq_t}{\max(p,2)})}  \left(1+ c^2 + c_t^{\frac{p}{\max(p,2)}} \right) 
\\ & \quad + \frac{2}{N} \kappa_{t+1}^2 d^{2\mathfrak{q}_{t+1}+2q+\max(q,\frac{q_t}{\max(p,2)})} \bar{\varepsilon}^{-2\mathfrak{r}_{t+1}} \left(1 + 3 c^2 \left[ 4+c+c_t^{\frac{1}{\max(p,2)}}\right] \right)
\\ & < \tilde{c}_2 d^{\tilde{q}_2} [\bar{\varepsilon}^2 + N^{-1}\bar{\varepsilon}^{-2\mathfrak{r}_{t+1}}]
\end{aligned}
\end{equation}
with $\tilde{c}_2 = 2+8\max(c,1)^2\kappa_{t+1}^2(4+\max(c,1)^2 + \max(c_t,1))$ and $\tilde{q}_2 = 2(q+\mathfrak{q}_{t+1})+\max(2q,q_t)$. 
But \eqref{eq:auxEq29} implies 
\begin{equation}
\label{eq:auxEq31}
\begin{aligned}
\E\left[\int_{\R^d} | \E[V_d(t+1,X_{t+1}^d)|X_t^d=x] -\Gamma_{\bar{\varepsilon},d,t}(x) |^2 \rho^d_t(dx) \right]  
< \tilde{c}_2 d^{\tilde{q}_2} [\bar{\varepsilon}^2 + N^{-1}\bar{\varepsilon}^{-2\mathfrak{r}_{t+1}}].
\end{aligned}
\end{equation}
Hence, Assumption~\ref{ass:Markov}(iv), the fact that $Y^{d,1}_t,\ldots,Y^{d,N}_t$ are  i.i.d.\ copies of $Y^{d}_t$ and Lemma~\ref{lem:jointProb} show that there exists 
$\omega \in \Omega$ such that  $\gamma_{\bar{\varepsilon},d,t}(x) = \frac{1}{N}\sum_{i=1}^N \hat{v}_{\bar{\varepsilon},d,t+1}(\eta_{\bar{\varepsilon},d,t}(x,Y^{d,i}_t(\omega)))$ (i.e.\ the realization of $\Gamma_{\bar{\varepsilon},d,t}$ at $\omega$) satisfies
\begin{equation}
\label{eq:auxEq32}
\begin{aligned}
\int_{\R^d} | \E[V_d(t+1,X_{t+1}^d)|X_t^d=x] - \gamma_{\bar{\varepsilon},d,t}(x) |^2 \rho^d_t(dx)
\leq 3 \tilde{c}_2 d^{\tilde{q}_2} [\bar{\varepsilon}^2 + N^{-1}\bar{\varepsilon}^{-2\mathfrak{r}_{t+1}}]
\end{aligned}
\end{equation}
and 
\begin{equation}
\label{eq:Ybound}
\max_{i=1,\ldots,N}\|Y_t^{d,i}(\omega)\| \leq 3 N cd^q.
\end{equation}
We now define
\begin{equation}
\begin{aligned}\label{eq:auxEq5}
\hat{v}_{\bar{\varepsilon},d,t}(x) = \max\left(\phi_{\bar{\varepsilon},d,t}(x) - \delta ,\gamma_{\bar{\varepsilon},d,t}(x) \right)
\end{aligned}
\end{equation}
and claim that $\psi_{\varepsilon,d,t}=\hat{v}_{\bar{\varepsilon},d,t}$ satisfies all the properties required in \eqref{eq:approxErrort}--\eqref{eq:NNlipschitz1}.

\textbf{7. Growth bound on the constructed network:}
Let us first verify \eqref{eq:NNgrowth}. Indeed, the growth bound on $\phi_{\bar{\varepsilon},d,t}$ in \eqref{eq:phiGrowth}, the induction hypothesis \eqref{eq:NNgrowtht+1}, the growth bound \eqref{eq:etaGrowth} on $\eta_{\bar{\varepsilon},d,t}$, the bound \eqref{eq:Ybound} on $\|Y_t^{d,i}(\omega)\|$ and the choice of $N$ in \eqref{eq:NandDelaChoice} imply  for all $x \in \R^d$ that
\begin{equation} 
\begin{aligned}
\label{eq:auxEq44} 
|\psi_{\varepsilon,d,t}(x)| 
& \leq  |\phi_{\bar{\varepsilon},d,t}(x)| + \delta +| \gamma_{\bar{\varepsilon},d,t}(x)|
\\ & \leq c d^q (2 + \|x\|) + \delta + \frac{1}{N}\sum_{i=1}^N |\hat{v}_{\bar{\varepsilon},d,t+1}(\eta_{\bar{\varepsilon},d,t}(x,Y^{d,i}_t(\omega)))|
\\ & \leq c d^q (2 + \|x\|) + \delta + \frac{1}{N}\sum_{i=1}^N \kappa_{t+1} d^{\mathfrak{q}_{t+1}} \bar{\varepsilon}^{-\mathfrak{r}_{t+1}} (1+\|\eta_{\bar{\varepsilon},d,t}(x,Y^{d,i}_t(\omega))\|)
\\ & \leq c d^q (2 + \|x\|) + \delta + \frac{1}{N}\sum_{i=1}^N \kappa_{t+1} d^{\mathfrak{q}_{t+1}} \bar{\varepsilon}^{-\mathfrak{r}_{t+1}} (1+cd^q(2+\|x\|+\|Y^{d,i}_t(\omega)\|))
\\ & \leq c d^q (2 + \|x\|) + \delta + \kappa_{t+1} d^{\mathfrak{q}_{t+1}} \bar{\varepsilon}^{-\mathfrak{r}_{t+1}} (1+cd^q(2+\|x\|+3 N c d^q ))
\\ & \leq \tilde{c}_3 d^{\tilde{q}_3} \bar{\varepsilon}^{-\tilde{r}_3} (1+\|x\|)  
\end{aligned}
\end{equation}
with $\tilde{c}_3 = 18 \max(c,1,\kappa_{t+1})\max(c,1)^2$, $\tilde{q}_3 = \mathfrak{q}_{t+1}+2q$ and $\tilde{r}_3 = 3\mathfrak{r}_{t+1}+2$. 

\textbf{8. Bounding the size of the constructed network:}
Next, we verify \eqref{eq:NNsparse1}. To achieve this, first note that for each $i$, Lemma~\ref{lem:nnfixed} shows that the map $x \mapsto \eta_{\bar{\varepsilon},d,t}(x,Y^{d,i}_t(\omega))$ can be realized as a neural network with size at most $\mathrm{size}(\eta_{\bar{\varepsilon},d,t})$. Next, the composition of two ReLU neural networks $\phi_1$, $\phi_2$ can again be realized by a ReLU neural network with size at most $2 (\mathrm{size}(\phi_1) + \mathrm{size}(\phi_2))$ (see, e.g. \citet[Proposition~2.2]{Opschoor2020}). Finally, \citet[Lemma~3.2]{GS20_925} shows that the weighted sum of deep neural networks $\phi_1,\ldots,\phi_N$ with the same number of layers, the same input dimension and the same output dimension can be realized by another deep neural network with size at most $\sum_{i=1}^N \mathrm{size}(\phi_i)$.   Therefore, $\gamma_{\bar{\varepsilon},d,t}$ can be realized as a deep neural network with  
\begin{equation}
\label{eq:sizegamma}
\begin{aligned}
\mathrm{size}(\gamma_{\bar{\varepsilon},d,t}) & \leq \sum_{i=1}^N 2 (\mathrm{size}(\hat{v}_{\bar{\varepsilon},d,t+1}) + \mathrm{size}(\eta_{\bar{\varepsilon},d,t}))
\\ & \leq 2N (\kappa_{t+1} d^{\mathfrak{q}_{t+1}}\bar{\varepsilon}^{-\mathfrak{r}_{t+1}}+ c d^q \bar{\varepsilon}^{-\alpha}),
\end{aligned}
\end{equation}
where the last step follows from the induction hypothesis \eqref{eq:NNsparse1t+1} and the bound  \eqref{eq:etasparse} on the size of $\eta_{\bar{\varepsilon},d,t}$. Next, subtracting a constant corresponds to a change of the ``bias'' $b^L$  in the last layer and so $\phi_{\bar{\varepsilon},d,t} - \delta$ is a neural network with $\mathrm{size}(\phi_{\bar{\varepsilon},d,t} - \delta) = \mathrm{size}(\phi_{\bar{\varepsilon},d,t})$. Define the neural network $\mathfrak{m}\colon \R^2 \to \R$, $\mathfrak{m}(x,y) = A^2 \varrho(A^1(x,y)^\top)$ with 
\begin{equation*}
A^1= \left(
\begin{array}{cr}
1&-1\\
0&1\\
0&-1
\end{array}
\right) \quad \mbox{and} \quad
A^2= \left(
\begin{array}{ccc}
1&1&-1\\
\end{array}
\right),
\end{equation*}
then $\mathfrak{m}(x,y)= \max(x-y,0)+\max(y,0)-\max(-y,0) = \max(x,y)$.
Thus, $\hat{v}_{\bar{\varepsilon},d,t}$ in \eqref{eq:auxEq5} can be realized as a neural network by the composition of $\mathfrak{m}$ with the parallelization of $\phi_{\bar{\varepsilon},d,t} - \delta$ and $\gamma_{\bar{\varepsilon},d,t}$ (see, e.g., \citet[Proposition~2.3]{Opschoor2020}) and the size of the parallelization is bounded by $\mathrm{size}(\phi_{\bar{\varepsilon},d,t}) + \mathrm{size}(\gamma_{\bar{\varepsilon},d,t})$. This and the bound \eqref{eq:NNsparse} on the size of $\phi_{\bar{\varepsilon},d,t}$, the choice of $N$ in \eqref{eq:NandDelaChoice} and the bound \eqref{eq:sizegamma} on the size of $\gamma_{\bar{\varepsilon},d,t}$ imply
\begin{equation}
\label{eq:sizeVhat}
\begin{aligned}
\mathrm{size}(\psi_{\varepsilon,d,t}) & \leq  2(\mathrm{size}(\mathfrak{m})+\mathrm{size}(\phi_{\bar{\varepsilon},d,t})+ \mathrm{size}(\gamma_{\bar{\varepsilon},d,t}))
\\ & \leq 2(7+cd^q \bar{\varepsilon}^{-\alpha} + 2N (\kappa_{t+1} d^{\mathfrak{q}_{t+1}}\bar{\varepsilon}^{-\mathfrak{r}_{t+1}}+ c d^q \bar{\varepsilon}^{-\alpha}))
\\ & \leq 
\tilde{c}_4 d^{\tilde{q}_4} \bar{\varepsilon}^{-\tilde{r}_4}  
\end{aligned}
\end{equation}
with $\tilde{c}_4 = 2(7+5c + 4\kappa_{t+1})$, $\tilde{q}_4 = \max(q,\mathfrak{q}_{t+1})$ and $\tilde{r}_4 =2 \mathfrak{r}_{t+1}+2+\max(\alpha,\mathfrak{r}_{t+1})$. 

\textbf{9. Lipschitz constant of the constructed network:}
Next, we verify \eqref{eq:NNlipschitz1}. To do this, we note that the induction hypothesis \eqref{eq:NNlipschitz1t+1} and the Lipschitz property \eqref{eq:etaLipschitz} of $\eta_{\bar{\varepsilon},d,t}$ imply for all $x,y \in \R^d$ that
\begin{equation} 
\begin{aligned}
\label{eq:auxEq37}
|\gamma_{\bar{\varepsilon},d,t}(x)-\gamma_{\bar{\varepsilon},d,t}(y)| &  = \left| \frac{1}{N}\sum_{i=1}^N (\hat{v}_{\bar{\varepsilon},d,t+1}(\eta_{\bar{\varepsilon},d,t}(x,Y^{d,i}_t(\omega))) -  \hat{v}_{\bar{\varepsilon},d,t+1}(\eta_{\bar{\varepsilon},d,t}(y,Y^{d,i}_t(\omega)))) \right| 
\\ & \leq \frac{1}{N}\sum_{i=1}^N  \kappa_{t+1} d^{\mathfrak{q}_{t+1}} \|\eta_{\bar{\varepsilon},d,t}(x,Y^{d,i}_t(\omega)) -  \eta_{\bar{\varepsilon},d,t}(y,Y^{d,i}_t(\omega))\| 
\\ & \leq  \kappa_{t+1} d^{\mathfrak{q}_{t+1}} c d^q \|x -  y\|. 
\end{aligned}
\end{equation}
In addition, \eqref{eq:NNlipschitz} implies $\mathrm{Lip}(\phi_{\bar{\varepsilon},d,t} - \delta) = \mathrm{Lip}(\phi_{\bar{\varepsilon},d,t}) \leq c d^q$ and the pointwise maximum of two Lipschitz continuous functions is again Lipschitz continuous with Lipschitz constant given by the maximum of the two Lipschitz constants. Combining this with \eqref{eq:auxEq37} yields
\begin{equation}
\label{eq:psiLipschitzConstant}
\mathrm{Lip}(\psi_{\varepsilon,d,t})  \leq \max(c d^q,\kappa_{t+1} d^{\mathfrak{q}_{t+1}+q} c ) \leq c \max(\kappa_{t+1},1) d^{\mathfrak{q}_{t+1}+q}.
\end{equation}

\textbf{10. Bounding the overall approximation error:}
We now work towards verifying the approximation error bound \eqref{eq:approxErrort}. To achieve this, let $C_t = \{x \in \R^d \, \colon \, g_d(t,x) < \E[V_d(t+1,X_{t+1}^d)|X_t^d=x] \}$ be the continuation region and let $\hat{C}_t = \{x \in \R^d \, \colon \, \phi_{\bar{\varepsilon},d,t}(x) - \delta < \gamma_{\bar{\varepsilon},d,t}(x) \}$ be the approximate continuation region. Then 
\begin{equation}
\begin{aligned}
\label{eq:auxEq7}
 & |V_d(t,x)  -\hat{v}_{\bar{\varepsilon},d,t}(x)| \\ &  = | g_d(t,x) -\phi_{\bar{\varepsilon},d,t}(x) + \delta |\mathbbm{1}_{C_t^c \cap \hat{C}_t^c}(x)  + |\E[V_d(t+1,X_{t+1}^d)|X_t^d=x] - \phi_{\bar{\varepsilon},d,t}(x) + \delta| \mathbbm{1}_{C_t \cap \hat{C}_t^c}(x) 
 \\ & \quad +  | g_d(t,x) -\gamma_{\bar{\varepsilon},d,t}(x) |\mathbbm{1}_{C_t^c \cap \hat{C}_t}(x) +  | \E[V_d(t+1,X_{t+1}^d)|X_t^d=x] -\gamma_{\bar{\varepsilon},d,t}(x)  |\mathbbm{1}_{C_t \cap \hat{C}_t}(x).
\end{aligned}
\end{equation} 
We now estimate (the integral of) each of these four terms separately. 
For the first term, from \eqref{eq:NNclose} we directly get 
\begin{equation}
\label{eq:auxEq8}
| g_d(t,x) -\phi_{\bar{\varepsilon},d,t}(x) + \delta |\mathbbm{1}_{C_t^c \cap \hat{C}_t^c}(x) \leq \delta + \bar{\varepsilon} c d^{q} (1+\|x\|^p)
\end{equation}
and so we proceed with analysing the second term. 

\textbf{10.a) Bounding the approximation error on $C_t \cap \hat{C}_t^c$:}
From Lemma~\ref{lem:Vgrowth} we have the growth bound $|V_d(t+1,x)|  \leq  \hat{c}_{t+1} d^{\hat{q}_{t+1}} (1+\|x\|)$ 
and so, using \eqref{eq:phiGrowth}, the second term in \eqref{eq:auxEq7} can be estimated as 
\begin{equation}
\label{eq:auxEq9}
\begin{aligned}
& |\E[V_d(t+1,X_{t+1}^d)|X_t^d=x]  - \phi_{\bar{\varepsilon},d,t}(x) + \delta| \mathbbm{1}_{C_t \cap \hat{C}_t^c}(x) 
\\ & \quad \leq \left[\hat{c}_{t+1} d^{\hat{q}_{t+1}} (1+ \E[\|X^d_{t+1}\| |X^d_t=x]) +  c d^{q} (2+\|x\|) + \delta \right] \mathbbm{1}_{C_t \cap \hat{C}_t^c}(x) 
\\ & \quad \leq 4 \max(\hat{c}_{t+1},c) d^{\hat{q}_{t+1}} \left[1+\sum_{s=t}^{t+1} \E[\|X^d_s\| |X^d_t=x] \right] \mathbbm{1}_{C_t \cap \hat{C}_t^c}(x).
\end{aligned}
\end{equation}
Combining this with \eqref{eq:CondMomentIntegrable}, \eqref{eq:condMomentGrowth2} in Lemma~\ref{lem:growthMoments}
and Hölder's inequality we estimate 
\begin{equation}
\label{eq:auxEq10}
\begin{aligned}
	& \left(\int_{\R^d} |\E[V_d(t+1,X_{t+1}^d)|X_t^d=x]  - \phi_{\bar{\varepsilon},d,t}(x) + \delta|^2 \mathbbm{1}_{C_t \cap \hat{C}_t^c}(x)  \rho^d_t(dx)\right)^{1/2}
	\\ & \quad \leq 4 \max(\hat{c}_{t+1},c) d^{\hat{q}_{t+1}} \left(\int_{\R^d} \left[1+\sum_{s=t}^{t+1} \E[\|X^d_s\| |X^d_t=x]  \right]^2 \mathbbm{1}_{C_t \cap \hat{C}_t^c}(x) \rho^d_t(dx)\right)^{1/2}
	\\ & \quad \leq 4 \max(\hat{c}_{t+1},c) d^{\hat{q}_{t+1}} \left(\int_{\R^d} \left[1+\sum_{s=t}^{t+1} \E[\|X^d_s\| |X^d_t=x]  \right]^4 \rho^d_t(dx) \right)^{\frac{1}{4}}  \left(\rho^d_t({C_t \cap \hat{C}_t^c})\right)^{\frac{1}{4}}
	\\ & \quad \leq 4 \max(\hat{c}_{t+1},c) d^{\hat{q}_{t+1}}\left[1+\sum_{s=t}^{t+1} \left(\int_{\R^d}  \E[\|X^d_s\| |X^d_t=x]^4 \rho^d_t(dx) \right)^{\frac{1}{4}}\right]  \left(\rho^d_t({C_t \cap \hat{C}_t^c})\right)^{\frac{1}{4}}
	\\ & \quad \leq 4 \max(\hat{c}_{t+1},c) d^{\hat{q}_{t+1}}\left[1+ \tilde{c}_1 d^{\tilde{q}_1} 2 (1+ c_t^{\frac{1}{4}}d^{\frac{q_t}{4}})\right]  \left(\rho^d_t({C_t \cap \hat{C}_t^c})\right)^{\frac{1}{4}}.	 
\end{aligned}
\end{equation}

\textbf{10.b) Estimating $\rho^d_t({C_t \cap \hat{C}_t^c})$:}
Next, we aim to estimate $\rho^d_t({C_t \cap \hat{C}_t^c})$. To do this, set $A = \big\{x \in \R^d \, \colon   \, |g_d(t,x) - \phi_{\bar{\varepsilon},d,t}(x)|> \frac{\delta}{2} \big\}$,  $B = \big\{x \in \R^d \, \colon   \, |\E[V_d(t+1,X_{t+1}^d)|X_t^d=x]- \gamma_{\bar{\varepsilon},d,t}(x)|> \frac{\delta}{2} \big\}$ and note that
\begin{equation}
\label{eq:auxEq11}
\begin{aligned}
 C_t \cap \hat{C}_t^c & = \big\{x \in \R^d \, \colon   \, g_d(t,x) < \E[V_d(t+1,X_{t+1}^d)|X_t^d=x], \phi_{\bar{\varepsilon},d,t}(x) - \delta \geq \gamma_{\bar{\varepsilon},d,t}(x)  \big\}
 \\ & \subset
  A \cup B,
\end{aligned}
\end{equation}
since $|g_d(t,x) - \phi_{\bar{\varepsilon},d,t}(x)|\leq \frac{\delta}{2}$, $|\E[V_d(t+1,X_{t+1}^d)|X_t^d=x]- \gamma_{\bar{\varepsilon},d,t}(x)|\leq \frac{\delta}{2}  $ and $g_d(t,x) < \E[V_d(t+1,X_{t+1}^d)|X_t^d=x]$ implies 
\[
\phi_{\bar{\varepsilon},d,t}(x) - \delta \leq g_d(t,x) - \frac{\delta}{2}< \E[V_d(t+1,X_{t+1}^d)|X_t^d=x]- \frac{\delta}{2} \leq  \gamma_{\bar{\varepsilon},d,t}(x).
\]

Furthermore, Markov's inequality, \eqref{eq:CondMomentIntegrable} and \eqref{eq:NNclose} imply that
\begin{equation}
\label{eq:auxEq13}
\begin{aligned}
\rho^d_t (A) = \rho^d_t & \left(\big\{x \in \R^d \, \colon   \, |g_d(t,x) - \phi_{\bar{\varepsilon},d,t}(x)|> \frac{\delta}{2} \big\}\right) \\ & \leq \frac{4}{\delta^2} \int_{\R^d} |g_d(t,x) - \phi_{\bar{\varepsilon},d,t}(x)|^2 \rho^d_t(dx)
\\ & \leq \frac{4}{\delta^2} \int_{\R^d} |\bar{\varepsilon}c d^{q} (1+\|x\|^p) |^2 \rho^d_t(dx) 
\\ & \leq \frac{8(\bar{\varepsilon}c d^{q})^2}{\delta^2} \left(1+\int_{\R^d} \|x\|^{2p} \rho^d_t(dx) \right)
\\ & \leq \frac{8(1+c_t)c^2 \bar{\varepsilon}^2 d^{2q+q_t}}{\delta^2}.
\end{aligned} 
\end{equation}

Similarly, Markov's inequality and \eqref{eq:auxEq32} yield
\begin{equation}
\label{eq:auxEq16}
\begin{aligned}
\rho^d_t(B) = \rho^d_t & \left(\big\{x \in \R^d \, \colon   \, |\E[V_d(t+1,X_{t+1}^d)|X_t^d=x]- \gamma_{\bar{\varepsilon},d,t}(x)|> \frac{\delta}{2} \big\}\right) \\ & \leq \frac{4}{\delta^2} \int_{\R^d} |\E[V_d(t+1,X_{t+1}^d)|X_t^d=x]- \gamma_{\bar{\varepsilon},d,t}(x)|^2 \rho^d_t(dx)
\\ & \leq  \frac{12}{\delta^2} \tilde{c}_2 d^{\tilde{q}_2} [\bar{\varepsilon}^2 + N^{-1}\bar{\varepsilon}^{-2\mathfrak{r}_{t+1}}].
\end{aligned} 
\end{equation}
Putting together \eqref{eq:auxEq10}, \eqref{eq:auxEq11}, \eqref{eq:auxEq13} and  \eqref{eq:auxEq16} and inserting the choices \eqref{eq:NandDelaChoice} we obtain
\begin{equation}
\label{eq:auxEq40}
\begin{aligned}
& \left(\int_{\R^d} |\E[V_d(t+1,X_{t+1}^d)|X_t^d=x]  - \phi_{\bar{\varepsilon},d,t}(x) + \delta|^2 \mathbbm{1}_{C_t \cap \hat{C}_t^c}(x)  \rho^d_t(dx)\right)^{1/2}
\\ & \quad \leq 4 \max(\hat{c}_{t+1},c) d^{\hat{q}_{t+1}}\left[1+ \tilde{c}_1 d^{\tilde{q}_1} 2 (1+ c_t^{\frac{1}{4}}d^{\frac{q_t}{4}})\right]  \left(\rho^d_t(A) + \rho^d_t(B)\right)^{\frac{1}{4}}
\\ & \quad \leq 4 \max(\hat{c}_{t+1},c) d^{\hat{q}_{t+1}+\tilde{q}_1+\frac{q_t}{4}}\left[1+ 2\tilde{c}_1   (1+ c_t^{\frac{1}{4}})\right] 
\\ & \quad \quad \cdot  \left(\frac{8(1+c_t)c^2 \bar{\varepsilon}^2 d^{2q+q_t}}{\delta^2} + \frac{12}{\delta^2} \tilde{c}_2 d^{\tilde{q}_2} [\bar{\varepsilon}^2 + N^{-1}\bar{\varepsilon}^{-2\mathfrak{r}_{t+1}}]\right)^{\frac{1}{4}}
\\ & \quad \leq  \tilde{c}_5 d^{\tilde{q}_5} \bar{\varepsilon}^{\frac{1}{4}}  
\end{aligned}
\end{equation}
with $\tilde{c}_5 = 4 \max(\hat{c}_{t+1},c) (1+ 2\tilde{c}_1  (1+ c_t^{\frac{1}{4}})) \left(8(1+c_t)c^2   + 24 \tilde{c}_2  \right)^{\frac{1}{4}}$ and $\tilde{q}_5 =\hat{q}_{t+1}+\tilde{q}_1+\frac{q_t}{4}+\frac{1}{4}\max(2q+q_t,\tilde{q}_2) $.

\textbf{10.c) Bounding the approximation error on $C_t^c \cap \hat{C}_t$:}
We are now concerned with the third term in \eqref{eq:auxEq7}. Observe that
\begin{equation}
\label{eq:auxEq17}
\begin{aligned}
| g_d(t,x) - \gamma_{\bar{\varepsilon},d,t}(x) |\mathbbm{1}_{C_t^c \cap \hat{C}_t}(x) \leq | g_d(t,x) - \gamma_{\bar{\varepsilon},d,t}(x) |(\mathbbm{1}_{A}(x) + \mathbbm{1}_{B}(x) + \mathbbm{1}_{A^c \cap B^c \cap C_t^c \cap \hat{C}_t}(x)).
\end{aligned}
\end{equation}
For $x \in A^c \cap \hat{C}_t$ we have  
\begin{equation}
\label{eq:auxEq19}
\begin{aligned}
\gamma_{\bar{\varepsilon},d,t}(x) + \frac{3}{2}\delta > \phi_{\bar{\varepsilon},d,t}(x)  + \frac{\delta}{2} \geq g_d(t,x)
\end{aligned}
\end{equation}
and therefore for $x \in A^c \cap B^c \cap C_t^c \cap \hat{C}_t$ it follows that
\begin{equation}
\label{eq:auxEq18}
\begin{aligned}
| g_d(t,x) - \gamma_{\bar{\varepsilon},d,t}(x) | & \leq | g_d(t,x) - \gamma_{\bar{\varepsilon},d,t}(x) - \frac{3}{2}\delta | + \frac{3}{2} \delta
\\ & = \gamma_{\bar{\varepsilon},d,t}(x) + 3 \delta - g_d(t,x)
\\ & \leq \gamma_{\bar{\varepsilon},d,t}(x) + 3 \delta - \E[V_d(t+1,X_{t+1}^d)|X_t^d=x]
\\ & \leq  \frac{7}{2} \delta.
\end{aligned}
\end{equation}
Combining this with \eqref{eq:auxEq17} we obtain
\begin{equation}
\label{eq:auxEq20}
\begin{aligned}
| g_d(t,x) - \gamma_{\bar{\varepsilon},d,t}(x) |\mathbbm{1}_{C_t^c \cap \hat{C}_t}(x) \leq | g_d(t,x) - \gamma_{\bar{\varepsilon},d,t}(x) |(\mathbbm{1}_{A}(x) + \mathbbm{1}_{B}(x)) + \frac{7}{2} \delta
\end{aligned}
\end{equation}
and consequently the growth bounds \eqref{eq:gGrowth}, \eqref{eq:condMomentGrowth2} and \eqref{eq:VGrowthLemma} on $g_d$, on the conditional moments and on $V_d$, Hölder's inequality and the approximation error bound \eqref{eq:auxEq32} for the continuation value  yield
\begin{equation}
\label{eq:auxEq38}
\begin{aligned}
& \left(\int_{\R^d} | g_d(t,x) - \gamma_{\bar{\varepsilon},d,t}(x) |^2\mathbbm{1}_{C_t^c \cap \hat{C}_t}(x)  \rho^d_t(dx)\right)^{1/2}
\\ & \quad \leq 
\left(\int_{\R^d} [| g_d(t,x) - \gamma_{\bar{\varepsilon},d,t}(x) |(\mathbbm{1}_{A}(x) + \mathbbm{1}_{B}(x))]^2  \rho^d_t(dx)\right)^{1/2} + \frac{7}{2}\delta 
\\ & \quad \leq 
\left(\int_{\R^d} [| g_d(t,x) - \E[V_d(t+1,X_{t+1}^d)|X_t^d=x] |(\mathbbm{1}_{A}(x) + \mathbbm{1}_{B}(x))]^2  \rho^d_t(dx)\right)^{1/2} 
\\ & \quad \quad + \left(\int_{\R^d} [| \E[V_d(t+1,X_{t+1}^d)|X_t^d=x] - \gamma_{\bar{\varepsilon},d,t}(x) |(\mathbbm{1}_{A}(x) + \mathbbm{1}_{B}(x))]^2  \rho^d_t(dx)\right)^{1/2} + \frac{7}{2}\delta 
\\ & \quad \leq 
2 \hat{c}_{t+1} d^{\hat{q}_{t+1}} \left(  \int_{\R^d} \left(1+\sum_{s=t}^{t+1} \E[\|X_s^d\||X_t^d=x]\right)^2 (\mathbbm{1}_{A}(x) + \mathbbm{1}_{B}(x))^2  \rho^d_t(dx)\right)^{1/2} 
\\ & \quad \quad + 2\left(\int_{\R^d}|\E[V_d(t+1,X_{t+1}^d)|X_t^d=x] - \gamma_{\bar{\varepsilon},d,t}(x) |^2  \rho^d_t(dx)\right)^{1/2} + \frac{7}{2}\delta 
\\ & \quad \leq 2 \hat{c}_{t+1} d^{\hat{q}_{t+1}} \left(\int_{\R^d} \left[1+2\tilde{c}_1 d^{\tilde{q}_1} (1+\|x\|) \right]^4 \rho^d_t(dx) \right)^{\frac{1}{4}}  \left[\left(\rho^d_t(A)\right)^{\frac{1}{4}}+\left(\rho^d_t(B)\right)^{\frac{1}{4}}\right]
 \\ & \quad \quad + 2 \left(3 \tilde{c}_2 d^{\tilde{q}_2} [\bar{\varepsilon}^2 + N^{-1}\bar{\varepsilon}^{-2\mathfrak{r}_{t+1}}]\right)^{1/2} + \frac{7}{2}\delta.
 \end{aligned}
 \end{equation}
Inserting the bound \eqref{eq:CondMomentIntegrable} on the moments of $\rho^d_t$, the bounds \eqref{eq:auxEq13}, \eqref{eq:auxEq16} on $\rho^d_t(A), \rho^d_t(B)$ and the choices \eqref{eq:NandDelaChoice} for $N$ and $\delta$ thus shows that
\begin{equation}
\label{eq:auxEq39}
\begin{aligned} 
& \left(\int_{\R^d} | g_d(t,x) - \gamma_{\bar{\varepsilon},d,t}(x) |^2\mathbbm{1}_{C_t^c \cap \hat{C}_t}(x)  \rho^d_t(dx)\right)^{1/2}
 \\ & \quad \leq 2 \hat{c}_{t+1} d^{\hat{q}_{t+1}} (1+ 2\tilde{c}_1 d^{\tilde{q}_1}(1+c_t^{\frac{1}{4}} d^{\frac{q_t}{4}})) \left[\left(\rho^d_t(A)\right)^{\frac{1}{4}}+\left(\rho^d_t(B)\right)^{\frac{1}{4}}\right]
+ 2 \left(6 \tilde{c}_2 d^{\tilde{q}_2} \bar{\varepsilon}^2 \right)^{1/2} + \frac{7}{2}\delta 
 \\ & \quad \leq 2 \hat{c}_{t+1} d^{\hat{q}_{t+1}+\tilde{q}_1+\frac{q_t}{4}} (1+ 2\tilde{c}_1 (1+c_t^{\frac{1}{4}})) \left[\left(\frac{8(1+c_t)c^2 \bar{\varepsilon}^2 d^{2q+q_t}}{\delta^2} \right)^{\frac{1}{4}}+\left(\frac{24}{\delta^2} \tilde{c}_2 d^{\tilde{q}_2} \bar{\varepsilon}^2 \right)^{\frac{1}{4}}\right]
\\ & \quad \quad +  2 \left(6 \tilde{c}_2 d^{\tilde{q}_2} \bar{\varepsilon}^2 \right)^{1/2} + \frac{7}{2} \delta 
 \\ & \quad \leq \tilde{c}_6 d^{\tilde{q}_6} \bar{\varepsilon}^{\frac{1}{4}} 
\end{aligned}
\end{equation}
with $\tilde{c}_6 = 2 \hat{c}_{t+1} (1+ 2\tilde{c}_1 (1+c_t^{\frac{1}{4}})) [(8(1+c_t)c^2 )^{\frac{1}{4}}+(24 \tilde{c}_2  )^{\frac{1}{4}}] +  2  (6 \tilde{c}_2   )^{1/2} + \frac{7}{2}$ and $\tilde{q}_6 = \hat{q}_{t+1}+\frac{1}{2}q+\tilde{q}_1+\frac{q_t}{2}+\frac{\tilde{q}_2}{2}$. 

\textbf{10.d) Combining the individual error estimates:}
Finally, note that the second and last line of \eqref{eq:auxEq13} yield
\begin{equation}
\begin{aligned}
\label{eq:auxEq42}
\int_{\R^d} |g_d(t,x) - \phi_{\bar{\varepsilon},d,t}(x)|^2 \rho^d_t(dx) \leq 2(1+c_t)c^2 \bar{\varepsilon}^2 d^{2q+q_t} .
\end{aligned} 
\end{equation} 
Consequently, combining the decomposition \eqref{eq:auxEq7} with the individual estimates \eqref{eq:auxEq32}, \eqref{eq:auxEq40}, \eqref{eq:auxEq39}
 and \eqref{eq:auxEq42} we obtain
 \begin{equation}
 \begin{aligned}
\label{eq:approxErrorEstimate}
& \left(\int_{\R^d} |V_d(t,x) - \psi_{\varepsilon,d,t}(x)|^2 \rho^d_t(d x) \right)^{1/2} 
\\ & \leq 
 \left(\int_{\R^d}| g_d(t,x) -\phi_{\bar{\varepsilon},d,t}(x) + \delta |^2\mathbbm{1}_{C_t^c \cap \hat{C}_t^c}(x) \rho^d_t(d x) \right)^{1/2} 
 \\ & \quad +
 \left(\int_{\R^d}  |\E[V_d(t+1,X_{t+1}^d)|X_t^d=x] - \phi_{\bar{\varepsilon},d,t}(x) + \delta|^2 \mathbbm{1}_{C_t \cap \hat{C}_t^c}(x) \rho^d_t(d x) \right)^{1/2}
\\ & \quad + \left(\int_{\R^d}  | g_d(t,x) -\gamma_{\bar{\varepsilon},d,t}(x) |^2\mathbbm{1}_{C_t^c \cap \hat{C}_t}(x) \rho^d_t(d x) \right)^{1/2}
\\ & \quad  + \left(\int_{\R^d}  | \E[V_d(t+1,X_{t+1}^d)|X_t^d=x] -\gamma_{\bar{\varepsilon},d,t}(x)  |^2\mathbbm{1}_{C_t \cap \hat{C}_t}(x) \rho^d_t(d x) \right)^{1/2}
\\ & \leq (2(1+c_t))^{\frac{1}{2}} c \bar{\varepsilon} d^{q+\frac{q_t}{2}} + \delta + \tilde{c}_5 d^{\tilde{q}_5} \bar{\varepsilon}^{\frac{1}{4}}   + \tilde{c}_6 d^{\tilde{q}_6} \bar{\varepsilon}^{\frac{1}{4}} + \bar{\varepsilon} d^{\frac{\tilde{q}_2}{2}}\left(6 \tilde{c}_2 \right)^{\frac{1}{2}}
\\ & \leq \tilde{c}_7 \bar{\varepsilon}^{\frac{1}{4}} d^{\tilde{q}_7}
\end{aligned} 
\end{equation}
with $\tilde{c}_7 =  (2(1+c_t))^{\frac{1}{2}} c  + 1 + \tilde{c}_5    + \tilde{c}_6 + \left(6 \tilde{c}_2 \right)^{\frac{1}{2}}$ and $\tilde{q}_7 = \max(q+\frac{q_t}{2},\tilde{q}_5,\tilde{q}_6,\frac{\tilde{q}_2}{2})$. Now choose  
\begin{align}
\label{eq:epsbarchoice}
\bar{\varepsilon} & = \varepsilon^{4} \left[\tilde{c}_7 d^{\tilde{q}_7} \right]^{-4}.
\end{align}
Inserting \eqref{eq:epsbarchoice} in \eqref{eq:approxErrorEstimate} proves that \eqref{eq:approxErrort} is satisfied. Furthermore, choosing 
\begin{align}
\label{eq:kappachoice} 
\kappa_{t} & = \max(\tilde{c}_3 (\tilde{c}_7)^{4\tilde{r}_3},\tilde{c}_4(\tilde{c}_7)^{4\tilde{r}_4},c \max(\kappa_{t+1},1)),
\\ \label{eq:pchoice} 
\mathfrak{q}_{t} & = \max(\tilde{q}_3+4\tilde{r}_3\tilde{q}_7 ,\tilde{q}_4+4\tilde{r}_4\tilde{q}_7,\mathfrak{q}_{t+1}+q),
\\ \label{eq:qchoice} 
\mathfrak{r}_{t} & = \max(4\tilde{r}_3,4\tilde{r}_4),
\end{align}
we obtain from \eqref{eq:auxEq44}, \eqref{eq:sizeVhat} and \eqref{eq:psiLipschitzConstant}  that \eqref{eq:NNgrowth}, \eqref{eq:NNsparse1} and \eqref{eq:NNlipschitz1} are satisfied. This completes the induction step. Hence, the statement follows.
\end{proof}

\begin{proof}[Proof of Corollary~\ref{cor:Lipschitz}] Fix $d \in \N$, $h \in [-R,R]^d$ and set $\rho^d = \frac{1}{2}\nu^d_0 + \frac{1}{2} \nu^d_h $, where $\nu^d_x$ denotes a multivariate normal distribution on $\R^d$ with mean $x$ and identity covariance. Then we estimate
	\[
	\int_{\R^d} \|x\|^{2\max(p,2)} \rho^d(dx) \leq \frac{1}{2}(1+2^{2\max(p,2)-1}) \int_{\R^d} \|x\|^{2\max(p,2)} \nu^d_0(dx) + \frac{1}{4} (2\|h\|)^{2\max(p,2)}
	\]
	and thus $\|h\|\leq R d^{1/2} $ and	\eqref{eq:auxEq59} show that 
	there exist $c >0$, $q \geq 0$ only depending on $p$ and $R$ such that the bound $\int_{\R^d} \|x\|^{2\max(p,2)} \rho^d(dx) \leq c d^q$ holds.  Hence, we can apply Theorem~\ref{thm:DNNapprox} and obtain for all $\varepsilon \in (0,1]$, $t \in \{0,\ldots,T\}$ the existence of a neural network $\psi_{\varepsilon,d,t}$ such that \eqref{eq:approxError} holds. From the proof of Theorem~\ref{thm:DNNapprox} we obtain that these networks satisfy the Lipschitz condition \eqref{eq:NNlipschitz1}. Therefore, for any $\varepsilon >0$ we may use Minkowski's inequality, the bound $\|\cdot\|_{L^2(\nu_y^d)} \leq \sqrt{2}\|\cdot\|_{L^2(\rho^d)}$ for $y \in \{0,h\}$, the approximation error bound \eqref{eq:approxError} and the Lipschitz property \eqref{eq:NNlipschitz1} to estimate
	\begin{equation}
	\label{eq:L2Lipschitz}
	\begin{aligned}
	& \|V_d(t,\cdot) - V_d(t,\cdot+h)\|_{L^2(\nu_0^d)} 
	\\ & \leq \|V_d(t,\cdot) - \psi_{\varepsilon,d,t}\|_{L^2(\nu_0^d)} + \|\psi_{\varepsilon,d,t} - \psi_{\varepsilon,d,t}(\cdot+h)\|_{L^2(\nu_0^d)} + \|\psi_{\varepsilon,d,t}(\cdot+h)-V_d(t,\cdot+h)\|_{L^2(\nu_0^d)} 
	\\ & \leq \sqrt{2} \|V_d(t,\cdot) - \psi_{\varepsilon,d,t}\|_{L^2(\rho^d)} + \|\psi_{\varepsilon,d,t} - \psi_{\varepsilon,d,t}(\cdot+h)\|_{L^2(\nu_0^d)} + \|\psi_{\varepsilon,d,t}-V_d(t,\cdot)\|_{L^2(\nu_h^d)}
	\\ & \leq 2 \sqrt{2} \varepsilon + \|h\| \kappa_t d^{\mathfrak{q}_t}.
	\end{aligned}
	\end{equation}
	This holds for any $\varepsilon >0 $ and from the statement in Step 2 of the proof of Theorem~\ref{thm:DNNapprox} the constants $\kappa_t$, $\mathfrak{q}_t$ do not depend on $d$, $\varepsilon$ and $h$ (but they depend on $R$); letting $\varepsilon$ tend to $0$ therefore yields the claimed statement.
\end{proof}

\subsection{Proof of Theorem~\ref{thm:refined}}
\label{subsec:DNNapproxProofrefined}
This subsection is devoted to the proof of Theorem~\ref{thm:refined}. It is based on the proof of Theorem~\ref{thm:DNNapprox} given in the previous subsection. 

\begin{proof}[Proof of Theorem~\ref{thm:refined}]
$ $\\
Proving this result just requires slight modifications in the proof of Theorem~\ref{thm:DNNapprox}. W.l.o.g.\ we may assume $c\geq h$, $q \geq \bar{q}$. 
In Step~1 we only need to choose $c_0,\ldots,c_T$ differently. Indeed, let $c_0 = c$, $c_{t+1} =  h(1+c_t)$ and set $q_t = \bar{q}(t+1)$. 
In Step 2, the stronger statement is modified accordingly: we will prove that for any $t \in \{0,\ldots,T\}$ there exist constants $\kappa_t,\mathfrak{q}_t,\mathfrak{r}_t \in [0,\infty)$ such 
that
for any family of probability measures $\rho_t^d$ on $\R^d$, $d \in \N$, with 
\begin{equation}
\label{eq:CondMomentIntegrable3}
\int_{\R^d} \|x\|^{2m\max(p,2)} \rho^d_t(dx) \leq c_t d^{q_t}
\end{equation}
and for all $d \in \N$, $\varepsilon \in (0,1]$ there exists a neural network $\psi_{\varepsilon,d,t}$ such that the approximation error estimate \eqref{eq:approxErrort} holds and $\psi_{\varepsilon,d,t}$ satisfies the growth \eqref{eq:NNgrowth}
and size conditions \eqref{eq:NNsparse1} and the modified Lipschitz condition
\begin{align} 
 \label{eq:NNlipschitz2} 
\mathrm{Lip}(\psi_{\varepsilon,d,t}) & \leq \kappa_t d^{\mathfrak{q}_t} \varepsilon^{-\zeta(T-t)}.
\end{align}
For $t=T$ condition \eqref{eq:NNlipschitz2} coincides with \eqref{eq:NNlipschitz1} and so the base case (Step 3) remains the same as in the proof of Theorem~\ref{thm:DNNapprox}. Due to the assumption \eqref{eq:fGrowthbdd} also Steps~4 and 5 only require slight modifications: \eqref{eq:auxEq30} becomes 
\begin{equation}
\label{eq:auxEq55}
\begin{aligned}
\int_{\R^d} \|z\|^{2m\max(p,2)} \rho_{t+1}^d(dz) & = \int_{\R^d} \int_{\R^d} \|f_t^d(x,y)\|^{2m\max(p,2)} \nu_{t}^d(dy) \rho^d_t(dx)
\\ & =  \int_{\R^d} \E[ \|f_t^d(x,Y^d_t)\|^{2m\max(p,2)}] \rho^d_t(dx)
\\ & \leq  hd^{\bar{q}} \left( 1 + \int_{\R^d}\|x\|^{2m\max(p,2)} \rho^d_t(dx) \right)
\\ & \leq h d^{q_t+\bar{q}} \left( 1 + c_t \right)
\\ & = c_{t+1} d^{q_{t+1}},
\end{aligned}
\end{equation}
 the Lipschitz condition \eqref{eq:NNlipschitz1t+1} is replaced by 
\begin{align} 
\label{eq:NNlipschitzt+1bdd} 
\mathrm{Lip}(\psi_{\varepsilon,d,t+1}) & \leq \kappa_{t+1} d^{\mathfrak{q}_{t+1}} \varepsilon^{-\zeta(T-t-1)}
\end{align}
and we modify the choice of $N$ in \eqref{eq:NandDelaChoice} to $N 
=\lceil \bar{\varepsilon}^{-2 \mathfrak{r}_{t+1}-2-2\theta} \rceil$. 

For the beginning of Step~6 and for Step~6.a) we proceed precisely as above and obtain the error estimates \eqref{eq:auxEq14}, \eqref{eq:auxEq25} and \eqref{eq:auxEq26}. In 6.b) the Lipschitz property \eqref{eq:NNlipschitzt+1bdd}  of the network now yields the additional factor $\bar{\varepsilon}^{-2\zeta(T-t-1)}$ and the approximation property \eqref{eq:fapprox} only holds on $[-(\bar{\varepsilon}^{-\beta}),\bar{\varepsilon}^{-\beta}]^d$, see \eqref{eq:fapproxBdd}. Hence, we estimate
\begin{equation}
\label{eq:auxEq46}
\begin{aligned}
\int_{\R^d} & | \E[ \hat{v}_{\bar{\varepsilon},d,t+1}(f^d_t(x,Y_t^d))]-\E[ \hat{v}_{\bar{\varepsilon},d,t+1}(\eta_{\bar{\varepsilon},d,t}(x,Y^{d}_t))]|^2 \rho^d_t(dx)
\\ & \leq (\kappa_{t+1} d^{\mathfrak{q}_{t+1}} \bar{\varepsilon}^{-\zeta(T-t-1)})^2 \int_{\R^d} | \E[ \| f^d_t(x,Y_t^d)- \eta_{\bar{\varepsilon},d,t}(x,Y^{d}_t)\|]|^2 \rho^d_t(dx)
\\ & \leq 2(\kappa_{t+1} d^{\mathfrak{q}_{t+1}} \bar{\varepsilon}^{-\zeta(T-t-1)})^2 \Bigg[ (\bar{\varepsilon} c d^{q})^2 \int_{[-(\bar{\varepsilon}^{-\beta}),\bar{\varepsilon}^{-\beta}]^d} |1+\|x\|^p + \E[\|Y^{d}_t\|^p]|^2 \rho^d_t(dx) 
\\ & \quad +  \int_{\R^d \setminus [-(\bar{\varepsilon}^{-\beta}),\bar{\varepsilon}^{-\beta}]^d} | \E[ \| f^d_t(x,Y_t^d)- \eta_{\bar{\varepsilon},d,t}(x,Y^{d}_t)\|]|^2 \rho^d_t(dx)
\\ & \quad + \int_{\R^d} | \E[ \| f^d_t(x,Y_t^d)- \eta_{\bar{\varepsilon},d,t}(x,Y^{d}_t)\|\mathbbm{1}_{\{Y^{d}_t \in \R^d \setminus [-(\bar{\varepsilon}^{-\beta}),\bar{\varepsilon}^{-\beta}]^d\}}]|^2 \rho^d_t(dx)\Bigg].
\end{aligned}
\end{equation}
The first term can be bounded as before. For the second term, note that Jensen's inequality and \eqref{eq:fGrowthbdd} imply that $	\E[\|f_t^d(x,Y_t^d)\|] \leq  c d^{q}  (1+\|x\|)$. Hence, we may apply Hölder's inequality, the growth bound \eqref{eq:etaGrowthbdd} on $\eta_{\bar{\varepsilon},d,t}$, Markov's inequality and \eqref{eq:CondMomentIntegrable3} to obtain
\begin{equation}
\label{eq:auxEq48}
\begin{aligned}
& \int_{\R^d \setminus [-(\bar{\varepsilon}^{-\beta}),\bar{\varepsilon}^{-\beta}]^d} | \E[ \| f^d_t(x,Y_t^d)- \eta_{\bar{\varepsilon},d,t}(x,Y^{d}_t)\|]|^2 \rho^d_t(dx)
\\ & \quad \leq \left(\int_{\R^d} | \E[ \| f^d_t(x,Y_t^d)- \eta_{\bar{\varepsilon},d,t}(x,Y^{d}_t)\|]|^4 \rho^d_t(dx)\right)^{\frac{1}{2}} \left(\rho^d_t(\R^d \setminus [-(\bar{\varepsilon}^{-\beta}),\bar{\varepsilon}^{-\beta}]^d)\right)^{\frac{1}{2}}
\\ & \quad \leq \left(\int_{\R^d} | \E[2c d^{q} \bar{\varepsilon}^{-\theta} (2+\|x\|+\|Y_t^d\|)]|^4 \rho^d_t(dx)\right)^{\frac{1}{2}}  \left(\frac{\int_{\R^d} \|x\|_\infty^{2m\max(2,p)} \rho^d_t(dx)}{(\bar{\varepsilon}^{-\beta})^{2m\max(2,p)}}\right)^{\frac{1}{2}}
\\ & \quad \leq  3(2 c d^{q})^2 \bar{\varepsilon}^{-2\theta}  \left(4+\E[\|Y_t^d\|]^2 + \left(\int_{\R^d} \|x\|^4 \rho^d_t(dx)\right)^{\frac{1}{2}}\right) \bar{\varepsilon}^{\beta m \max(2,p)}
\\ & \quad \quad \cdot  \left(\int_{\R^d} \|x\|^{2m\max(2,p)} \rho^d_t(dx)\right)^{\frac{1}{2}}
\\ & \quad \leq 3 \bar{\varepsilon}^{\beta m \max(2,p)-2\theta} (2 c d^{q})^2 \left(4+(cd^q)^2 + \left(c_t d^{q_t}\right)^{\frac{1}{2}}\right)  (c_t d^{q_t})^{\frac{1}{2}}.
\end{aligned}
\end{equation}
For the last term in \eqref{eq:auxEq46} we note that \eqref{eq:fGrowthbdd} and Jensen's inequality imply 
\[ 	\E[\|f_t^d(x,Y_t^d)\|^2] 
\leq (cd^q (1+\|x\|^{2m\max(p,2)}))^{\frac{1}{m\max(p,2)}} \leq  c d^{q}  (1+\|x\|)^2.
\]
Using this, Hölder's inequality, \eqref{eq:etaGrowthbdd} and Markov's inequality we obtain 
\begin{equation}
\label{eq:auxEq56}
\begin{aligned}
& \int_{\R^d} | \E[ \| f^d_t(x,Y_t^d)- \eta_{\bar{\varepsilon},d,t}(x,Y^{d}_t)\|\mathbbm{1}_{\{Y^{d}_t \in \R^d \setminus [-(\bar{\varepsilon}^{-\beta}),\bar{\varepsilon}^{-\beta}]^d\}}]|^2 \rho^d_t(dx)
\\ & \quad \leq \int_{\R^d}  \E[ \| f^d_t(x,Y_t^d)- \eta_{\bar{\varepsilon},d,t}(x,Y^{d}_t)\|^2] \rho^d_t(dx)  \P({Y^{d}_t \in \R^d \setminus [-(\bar{\varepsilon}^{-\beta}),\bar{\varepsilon}^{-\beta}]^d})
\\ & \quad \leq 2 \int_{\R^d}  \E[2 (c d^{q} \bar{\varepsilon}^{-\theta})^2 (2+\|x\|+\|Y_t^d\|)^2] \rho^d_t(dx)  \left(\frac{1}{(\bar{\varepsilon}^{-\beta})^{2m\max(2,p)}}\E[ \|Y_t^d\|_\infty^{2m\max(2,p)}]\right)
\\ & \quad \leq  3(2 c d^{q})^2 \left(4+\E[\|Y_t^d\|^2] + \left(\int_{\R^d} \|x\|^4 \rho^d_t(dx)\right)^{\frac{1}{2}}\right) \bar{\varepsilon}^{2m\beta \max(2,p)-2\theta} \E[ \|Y_t^d\|^{2m\max(2,p)}]
\\ & \quad \leq 3 \bar{\varepsilon}^{\beta m \max(2,p)-2\theta} (2 c d^{q})^2 \left(4+cd^q + \left(c_td^{q_t}\right)^{\frac{1}{2}}\right)  cd^q.
\end{aligned}
\end{equation}

Together this yields
\begin{equation}
\label{eq:auxEq47}
\begin{aligned}
\int_{\R^d} & | \E[ \hat{v}_{\bar{\varepsilon},d,t+1}(f^d_t(x,Y_t^d))]-\E[ \hat{v}_{\bar{\varepsilon},d,t+1}(\eta_{\bar{\varepsilon},d,t}(x,Y^{d}_t))]|^2 \rho^d_t(dx)
\\ & \leq 6 (\bar{\varepsilon}^{1-\zeta(T-t-1)} c \kappa_{t+1})^2 d^{2(q+\mathfrak{q}_{t+1})+\max(2q,\frac{pq_t}{\max(p,2)})}  \left(1+ c^2 + c_t^{\frac{p}{\max(p,2)}} \right)
\\ & \quad + 48 \bar{\varepsilon}^{\beta m \max(2,p)-2\theta-2\zeta(T-t-1)} (\kappa_{t+1})^2 d^{4q+2\mathfrak{q}_{t+1}+q_t} c^2 \left(4 +c^2 +  c_t^\frac{1}{2}\right) c_t.  
\end{aligned}
\end{equation}
Similarly, in Step 6.c) the factor  $\bar{\varepsilon}^{-2\mathfrak{r}_{t+1}}$ is replaced by $\bar{\varepsilon}^{-2\mathfrak{r}_{t+1}-2\theta}$ due to the growth bound \eqref{eq:etaGrowthbdd}. 
This and the modified bound \eqref{eq:auxEq47} (as opposed to \eqref{eq:auxEq27}) then also leads to a different estimate in 
Step 6.d), where we obtain
\begin{equation}
\label{eq:auxEq49}
\begin{aligned}
& \int_{\R^d}  \E[| \E[V_d(t+1,X_{t+1}^d)|X_t^d=x] -\Gamma_{\bar{\varepsilon},d,t}(x) |^2] \rho^d_t(dx)  
\\ & < \tilde{c}_2 d^{\tilde{q}_2} [\bar{\varepsilon}^{2-2\zeta(T-t-1)} + \bar{\varepsilon}^{\beta m \max(2,p)-2\theta - 2\zeta(T-t-1)}+ N^{-1}\bar{\varepsilon}^{-2\mathfrak{r}_{t+1}-2\theta}]
\end{aligned}
\end{equation}
with slightly modified constants $\tilde{c}_2 = 96 c^2\kappa_{t+1}^2(4+c^2 + c_t)c_t$ and $\tilde{q}_2 = 4q+2\mathfrak{q}_{t+1}+\max(2q,q_t)$. Thus, the same argument as before yields
the existence of
$\omega \in \Omega$ such that  $\gamma_{\bar{\varepsilon},d,t}$, the  realization of $\Gamma_{\bar{\varepsilon},d,t}$ at $\omega$,  satisfies
\begin{equation}
\label{eq:auxEq50}
\begin{aligned}
\int_{\R^d} & | \E[V_d(t+1,X_{t+1}^d)|X_t^d=x] - \gamma_{\bar{\varepsilon},d,t}(x) |^2 \rho^d_t(dx)
\\ & \leq 3 \tilde{c}_2 d^{\tilde{q}_2} [\bar{\varepsilon}^{2-2\zeta(T-t-1)} + \bar{\varepsilon}^{\beta m \max(2,p)-2\theta-2\zeta(T-t-1)}+ N^{-1}\bar{\varepsilon}^{-2\mathfrak{r}_{t+1}-2\theta}]
\end{aligned}
\end{equation}
 and \eqref{eq:Ybound} holds. In Step 7 the modified growth bound \eqref{eq:etaGrowthbdd} and the modified choice  $N 
 =\lceil \bar{\varepsilon}^{-2 \mathfrak{r}_{t+1}-2-2\theta} \rceil$ lead to an additional factor $\bar{\varepsilon}^{-3\theta}$ in \eqref{eq:auxEq44} and to a modified choice  $\tilde{r}_3 = 3\mathfrak{r}_{t+1}+2+3\theta$.
 Similarly, in Step 8 the modified choice of $N$ leads to an additional factor $\bar{\varepsilon}^{-2\theta}$ and to a modified choice  $\tilde{r}_4 = 2 \mathfrak{r}_{t+1}+2+\max(\alpha,\mathfrak{r}_{t+1})+2\theta$. Next, for the Lipschitz constant of the constructed network, i.e.\ Step 9, we need to verify \eqref{eq:NNlipschitz2}. To do this, we note that the induction hypothesis \eqref{eq:NNlipschitzt+1bdd} and the Lipschitz property \eqref{eq:etaLipschitzBdd} of $\eta_{\bar{\varepsilon},d,t}$ imply for all $x,y \in \R^d$ that
 \begin{equation} 
 \begin{aligned}
 \label{eq:auxEq51}
 |\gamma_{\bar{\varepsilon},d,t}(x)-\gamma_{\bar{\varepsilon},d,t}(y)| &  = \left| \frac{1}{N}\sum_{i=1}^N (\hat{v}_{\bar{\varepsilon},d,t+1}(\eta_{\bar{\varepsilon},d,t}(x,Y^{d,i}_t(\omega))) -  \hat{v}_{\bar{\varepsilon},d,t+1}(\eta_{\bar{\varepsilon},d,t}(y,Y^{d,i}_t(\omega)))) \right| 
 \\ & \leq \frac{1}{N}\sum_{i=1}^N  \kappa_{t+1} d^{\mathfrak{q}_{t+1}} \bar{\varepsilon}^{-\zeta(T-t-1)} \|\eta_{\bar{\varepsilon},d,t}(x,Y^{d,i}_t(\omega)) -  \eta_{\bar{\varepsilon},d,t}(y,Y^{d,i}_t(\omega))\| 
 \\ & \leq  \kappa_{t+1} d^{\mathfrak{q}_{t+1}} \bar{\varepsilon}^{-\zeta(T-t)} c d^q \|x -  y\|. 
 \end{aligned}
 \end{equation}
 Thus, we obtain 
 \begin{equation}
 \label{eq:psiLipschitzConstantBdd}
 \mathrm{Lip}(\psi_{\varepsilon,d,t})  \leq \max(c d^q,\kappa_{t+1} d^{\mathfrak{q}_{t+1}+q} \bar{\varepsilon}^{-\zeta(T-t)} c ) \leq c \max(\kappa_{t+1},1) d^{\mathfrak{q}_{t+1}+q} \bar{\varepsilon}^{-\zeta(T-t)}.
 \end{equation}
 Step 10 requires a different choice of $\delta$ and otherwise only minor modifications. We choose $\delta = \bar{\varepsilon}^{\frac{1}{2}(\min(1,\beta m-\theta) -\zeta(T-1))} $. Then in Step 10.b) the new bound \eqref{eq:auxEq50} leads to a slightly different bound  than in \eqref{eq:auxEq16} and \eqref{eq:auxEq40}. We obtain 
\begin{equation}
 \label{eq:auxEq52}
 \begin{aligned}
 	& \left(\int_{\R^d} |\E[V_d(t+1,X_{t+1}^d)|X_t^d=x]  - \phi_{\bar{\varepsilon},d,t}(x) + \delta|^2 \mathbbm{1}_{C_t \cap \hat{C}_t^c}(x)  \rho^d_t(dx)\right)^{1/2}
 	\\ & \quad \leq 4 \max(\hat{c}_{t+1},c) d^{\hat{q}_{t+1}+\tilde{q}_1+\frac{q_t}{4}}\left[1+ 2\tilde{c}_1  (1+ c_t^{\frac{1}{4}})\right] 
 	  \Bigg(\frac{8(1+c_t)c^2 \bar{\varepsilon}^2 d^{2q+q_t}}{\delta^2} \\ & \quad \quad + \frac{12}{\delta^2} \tilde{c}_2 d^{\tilde{q}_2} [\bar{\varepsilon}^{2-2\zeta(T-t-1)} + \bar{\varepsilon}^{\beta m \max(2,p)-2\theta-2\zeta(T-t-1)}+ N^{-1}\bar{\varepsilon}^{-2\mathfrak{r}_{t+1}-2\theta}]\Bigg)^{\frac{1}{4}}
 	\\ & \quad \leq  \tilde{c}_5 d^{\tilde{q}_5} \bar{\varepsilon}^{\frac{1}{4}(\min(1,\beta m-\theta) -\zeta(T-1))}  
 \end{aligned}
\end{equation}
with slightly modified constant $\tilde{c}_5 =  4 \hat{c}_{t+1} (1+ 2\tilde{c}_1  (1+ c_t^{\frac{1}{4}})) \left(8(1+c_t)c^2   + 36 \tilde{c}_2  \right)^{\frac{1}{4}}$ and $\tilde{q}_5$ as before.  
In Step 10.c), \eqref{eq:auxEq50} leads to analogous modifications 
in \eqref{eq:auxEq38} and \eqref{eq:auxEq39}, yielding
\begin{equation}
\label{eq:auxEq53}
\begin{aligned} 
& \left(\int_{\R^d} | g_d(t,x) - \gamma_{\bar{\varepsilon},d,t}(x) |^2\mathbbm{1}_{C_t^c \cap \hat{C}_t}(x)  \rho^d_t(dx)\right)^{1/2}
 \leq \tilde{c}_6 d^{\tilde{q}_6} \bar{\varepsilon}^{\frac{1}{4}(\min(1,\beta m-\theta) -\zeta(T-1))} 
\end{aligned}
\end{equation}
with $\tilde{c}_6 = 2 \hat{c}_{t+1} (1+ 2\tilde{c}_1 (1+c_t^{\frac{1}{4}})) [(8(1+c_t)c^2 )^{\frac{1}{4}}+(36 \tilde{c}_2  )^{\frac{1}{4}}] +  2  (9 \tilde{c}_2   )^{1/2} + \frac{7}{2}$ and $\tilde{q}_6$ as before. 
Combining \eqref{eq:auxEq50}, \eqref{eq:auxEq52}, \eqref{eq:auxEq53} and \eqref{eq:auxEq42} we thus obtain for Step 10.d) the bound
 \begin{equation}
\begin{aligned}
\label{eq:approxErrorEstimateBdd}
& \left(\int_{\R^d} |V_d(t,x) - \psi_{\varepsilon,d,t}(x)|^2 \rho^d_t(d x) \right)^{1/2}  \leq \tilde{c}_7\bar{\varepsilon}^{\frac{1}{4}(\min(1,\beta m-\theta) -\zeta(T-1))}  d^{\tilde{q}_7}
\end{aligned} 
\end{equation}
with $\tilde{c}_7 =  (2(1+c_t))^{\frac{1}{2}} c  + 1 + \tilde{c}_5    + \tilde{c}_6 + \left(9 \tilde{c}_2 \right)^{\frac{1}{2}} $ and $\tilde{q}_7$ as before. 
Choose 
\[\bar{\varepsilon}  = \left(\varepsilon^{-1}\tilde{c}_7 d^{\tilde{q}_7} \right)^{-\frac{4}{\min(1,\beta m-\theta) -\zeta(T-1))}}\]
and note that $\bar{\varepsilon} \in (0,1)$ because $\tilde{c}_7>1$ and $\frac{\min(1,\beta m-\theta)}{T-1} >\zeta$. By inserting this choice of $\bar{\varepsilon}$ in the bounds for the growth, size and Lipschitz constant of $\psi_{\varepsilon,d,t}$ we may then appropriately choose $\kappa_{t}$, $\mathfrak{q}_{t}$, $\mathfrak{r}_{t}$ (analogously to \eqref{eq:kappachoice}, \eqref{eq:pchoice}, \eqref{eq:qchoice}) and complete the proof. 
\end{proof}

{\small 
\bibliographystyle{abbrvnat}
\bibliography{references}

\begin{thebibliography}{53}
\providecommand{\natexlab}[1]{#1}
\providecommand{\url}[1]{\texttt{#1}}
\expandafter\ifx\csname urlstyle\endcsname\relax
  \providecommand{\doi}[1]{doi: #1}\else
  \providecommand{\doi}{doi: \begingroup \urlstyle{rm}\Url}\fi

\bibitem[Andersen and Broadie(2004)]{Andersen2004PrimalDual}
L.~Andersen and M.~Broadie.
\newblock Primal-dual simulation algorithm for pricing multidimensional
  american options.
\newblock \emph{Manag. Sci.}, 50\penalty0 (9):\penalty0 1222--1234, 2004.

\bibitem[Applebaum(2009)]{Applebaum2009}
D.~Applebaum.
\newblock \emph{L\'{e}vy processes and stochastic calculus}, volume 116 of
  \emph{Cambridge Studies in Advanced Mathematics}.
\newblock Cambridge University Press, Cambridge, second edition, 2009.

\bibitem[Bayer et~al.(2021{\natexlab{a}})Bayer, Belomestny, Hager, Pigato, and
  Schoenmakers]{Bayer2021a}
C.~Bayer, D.~Belomestny, P.~Hager, P.~Pigato, and J.~Schoenmakers.
\newblock Randomized optimal stopping algorithms and their convergence
  analysis.
\newblock \emph{SIAM J. Financial Math.}, 12\penalty0 (3):\penalty0 1201--1225,
  2021{\natexlab{a}}.

\bibitem[Bayer et~al.(2021{\natexlab{b}})Bayer, Hager, Riedel, and
  Schoenmakers]{Bayer21}
C.~Bayer, P.~Hager, S.~Riedel, and J.~Schoenmakers.
\newblock Optimal stopping with signatures.
\newblock \emph{Preprint, arXiv 2105.00778}, 2021{\natexlab{b}}.

\bibitem[Beck et~al.(2020)Beck, Hutzenthaler, Jentzen, and Kuckuck]{Beck2020}
C.~Beck, M.~Hutzenthaler, A.~Jentzen, and B.~Kuckuck.
\newblock An overview on deep learning-based approximation methods for partial
  differential equations.
\newblock \emph{Preprint, arXiv 2012.12348}, 2020.

\bibitem[Becker et~al.(2019)Becker, Cheridito, and Jentzen]{Becker2019}
S.~Becker, P.~Cheridito, and A.~Jentzen.
\newblock Deep optimal stopping.
\newblock \emph{J. Mach. Learn. Res.}, 20:\penalty0 Paper No. 74, 25, 2019.

\bibitem[Becker et~al.(2020)Becker, Cheridito, and Jentzen]{Becker2020}
S.~Becker, P.~Cheridito, and A.~Jentzen.
\newblock Pricing and hedging american-style options with deep learning.
\newblock \emph{J. Risk Financial Manag.}, 13\penalty0 (7), 2020.

\bibitem[Becker et~al.(2021)Becker, Cheridito, Jentzen, and Welti]{Becker2021a}
S.~Becker, P.~Cheridito, A.~Jentzen, and T.~Welti.
\newblock Solving high-dimensional optimal stopping problems using deep
  learning.
\newblock \emph{European J. Appl. Math.}, 32\penalty0 (3):\penalty0 470--514,
  2021.

\bibitem[Belomestny(2011)]{Belomestny2011}
D.~Belomestny.
\newblock On the rates of convergence of simulation-based optimization
  algorithms for optimal stopping problems.
\newblock \emph{Ann. Appl. Probab.}, 21\penalty0 (1):\penalty0 215--239, 2011.

\bibitem[Belomestny et~al.(2009)Belomestny, Bender, and
  Schoenmakers]{Belomestny2009}
D.~Belomestny, C.~Bender, and J.~Schoenmakers.
\newblock True upper bounds for {B}ermudan products via non-nested {M}onte
  {C}arlo.
\newblock \emph{Math. Finance}, 19\penalty0 (1):\penalty0 53--71, 2009.

\bibitem[Berner et~al.(2020)Berner, Grohs, and Jentzen]{BernerGrohsJentzen2018}
J.~Berner, P.~Grohs, and A.~Jentzen.
\newblock Analysis of the generalization error: empirical risk minimization
  over deep artificial neural networks overcomes the curse of dimensionality in
  the numerical approximation of {B}lack-{S}choles partial differential
  equations.
\newblock \emph{SIAM J. Math. Data Sci.}, 2\penalty0 (3):\penalty0 631--657,
  2020.

\bibitem[Bouchard and Warin(2012)]{Bouchard2012}
B.~Bouchard and X.~Warin.
\newblock Monte-carlo valuation of american options: Facts and new algorithms
  to improve existing methods.
\newblock In R.~A. Carmona, P.~Del~Moral, P.~Hu, and N.~Oudjane, editors,
  \emph{Numerical Methods in Finance}, pages 215--255, Berlin, Heidelberg,
  2012. Springer Berlin Heidelberg.

\bibitem[Broadie and Glasserman(2004)]{Broadie2004}
M.~Broadie and P.~Glasserman.
\newblock A stochastic mesh method for pricing high-dimensional american
  options.
\newblock \emph{J. Comput. Finance}, 7\penalty0 (4):\penalty0 35--72, 2004.

\bibitem[Buehler et~al.(2019)Buehler, Gonon, Teichmann, and Wood]{Buehler2018}
H.~Buehler, L.~Gonon, J.~Teichmann, and B.~Wood.
\newblock Deep hedging.
\newblock \emph{Quant. Finance}, 19\penalty0 (8):\penalty0 1271--1291, 2019.

\bibitem[Cioica-Licht et~al.(2022)Cioica-Licht, Hutzenthaler, and
  Werner]{Cioica2022}
P.~A. Cioica-Licht, M.~Hutzenthaler, and P.~T. Werner.
\newblock Deep neural networks overcome the curse of dimensionality in the
  numerical approximation of semilinear partial differential equations.
\newblock \emph{Preprint, arXiv 2205.14398}, 2022.

\bibitem[Cl\'{e}ment et~al.(2002)Cl\'{e}ment, Lamberton, and
  Protter]{Clement2002}
E.~Cl\'{e}ment, D.~Lamberton, and P.~Protter.
\newblock An analysis of a least squares regression method for {A}merican
  option pricing.
\newblock \emph{Finance Stoch.}, 6\penalty0 (4):\penalty0 449--471, 2002.

\bibitem[Cont and Tankov(2004)]{Cont2004}
R.~Cont and P.~Tankov.
\newblock \emph{{Financial Modelling with Jump Processes}}.
\newblock Chapman \& Hall/CRC, 2004.

\bibitem[Cuchiero et~al.(2020)Cuchiero, Khosrawi, and Teichmann]{Cuchiero2019}
C.~Cuchiero, W.~Khosrawi, and J.~Teichmann.
\newblock A generative adversarial network approach to calibration of local
  stochastic volatility models.
\newblock \emph{Risks}, 8\penalty0 (4):\penalty0 101, 2020.

\bibitem[Eberlein and Kallsen(2019)]{EberKall19}
E.~Eberlein and J.~Kallsen.
\newblock \emph{Mathematical finance}.
\newblock Springer Finance. Springer, Cham, 2019.

\bibitem[Ech-Chafiq et~al.(2022)Ech-Chafiq, Henry-Labordere, and
  Lelong]{EchChafiq21}
Z.~E.~F. Ech-Chafiq, P.~Henry-Labordere, and J.~Lelong.
\newblock Pricing bermudan options using regression trees/random forests.
\newblock \emph{Preprint, arXiv 2201.02587}, 2022.

\bibitem[Elbr\"{a}chter et~al.(2022)Elbr\"{a}chter, Grohs, Jentzen, and
  Schwab]{EGJS18_787}
D.~Elbr\"{a}chter, P.~Grohs, A.~Jentzen, and C.~Schwab.
\newblock D{NN} expression rate analysis of high-dimensional {PDE}s:
  application to option pricing.
\newblock \emph{Constr. Approx.}, 55\penalty0 (1):\penalty0 3--71, 2022.

\bibitem[F{\"o}llmer and Schied(2016)]{follmerschied}
H.~F{\"o}llmer and A.~Schied.
\newblock \emph{Stochastic finance: an introduction in discrete time}.
\newblock Walter de Gruyter, 4th rev. ed. edition, 2016.

\bibitem[Garcia(2003)]{Garcia2003}
D.~Garcia.
\newblock Convergence and biases of monte carlo estimates of american option
  prices using a parametric exercise rule.
\newblock \emph{J. Econ. Dyn. Control}, 27\penalty0 (10):\penalty0 1855--1879,
  2003.

\bibitem[Germain et~al.(2021)Germain, Pham, and Warin]{Germain2021}
M.~Germain, H.~Pham, and X.~Warin.
\newblock Neural networks-based algorithms for stochastic control and pdes in
  finance.
\newblock \emph{Preprint, arXiv 2101.08068}, 2021.

\bibitem[Gonon(2021)]{Gonon2021}
L.~Gonon.
\newblock {Random feature neural networks learn Black-Scholes type PDEs without
  curse of dimensionality}.
\newblock \emph{Preprint, arXiv 2106.08900}, 2021.

\bibitem[Gonon and Schwab(2021{\natexlab{a}})]{GS20_925}
L.~Gonon and C.~Schwab.
\newblock Deep {R}e{LU} network expression rates for option prices in
  high-dimensional, exponential {L}\'{e}vy models.
\newblock \emph{Finance Stoch.}, 25\penalty0 (4):\penalty0 615--657,
  2021{\natexlab{a}}.

\bibitem[Gonon and Schwab(2021{\natexlab{b}})]{GS21}
L.~Gonon and C.~Schwab.
\newblock Deep relu neural networks overcome the curse of dimensionality for
  partial integrodifferential equations.
\newblock \emph{To appear in Anal. Appl.; arXiv:2102.11707},
  2021{\natexlab{b}}.

\bibitem[Gonon et~al.(2021)Gonon, Grohs, Jentzen, Kofler, and
  {\v{S}}i{\v{s}}ka]{Gonon2019}
L.~Gonon, P.~Grohs, A.~Jentzen, D.~Kofler, and D.~{\v{S}}i{\v{s}}ka.
\newblock Uniform error estimates for artificial neural network approximations
  for heat equations.
\newblock \emph{IMA J. Numer. Anal.}, 42\penalty0 (3):\penalty0 1991--2054,
  2021.

\bibitem[Grohs and Herrmann(2021)]{Grohs2021}
P.~Grohs and L.~Herrmann.
\newblock Deep neural network approximation for high-dimensional parabolic
  hamilton-jacobi-bellman equations.
\newblock \emph{Preprint, arXiv 2103.05744}, 2021.

\bibitem[Grohs et~al.(2018)Grohs, Hornung, Jentzen, and von
  Wurstemberger]{HornungJentzen2018}
P.~Grohs, F.~Hornung, A.~Jentzen, and P.~von Wurstemberger.
\newblock A proof that artificial neural networks overcome the curse of
  dimensionality in the numerical approximation of {B}lack-{S}choles partial
  differential equations.
\newblock \emph{To appear in Mem. Am. Math. Soc.; arXiv:1809.02362}, 2018.

\bibitem[Haugh and Kogan(2004)]{Haugh2004}
M.~B. Haugh and L.~Kogan.
\newblock Pricing american options: A duality approach.
\newblock \emph{Oper. Res.}, 52\penalty0 (2):\penalty0 258--270, 2004.

\bibitem[Herrera et~al.(2021)Herrera, Krach, Ruyssen, and Teichmann]{Herrera21}
C.~Herrera, F.~Krach, P.~Ruyssen, and J.~Teichmann.
\newblock Optimal stopping via randomized neural networks.
\newblock \emph{Preprint, arXiv 2104.13669}, 2021.

\bibitem[Hutzenthaler et~al.(2020)Hutzenthaler, Jentzen, Kruse, and
  Nguyen]{HutzenthalerJentzenKruse2019}
M.~Hutzenthaler, A.~Jentzen, T.~Kruse, and T.~A. Nguyen.
\newblock A proof that rectified deep neural networks overcome the curse of
  dimensionality in the numerical approximation of semilinear heat equations.
\newblock \emph{SN PDE}, 1\penalty0 (2):\penalty0 10, 2020.

\bibitem[Jain and Oosterlee(2015)]{Jain2015}
S.~Jain and C.~W. Oosterlee.
\newblock The stochastic grid bundling method: Efficient pricing of bermudan
  options and their greeks.
\newblock \emph{Appl. Math. Comput.}, 269:\penalty0 412--431, 2015.

\bibitem[Jentzen et~al.(2021)Jentzen, Salimova, and Welti]{JentzenSalimova2021}
A.~Jentzen, D.~Salimova, and T.~Welti.
\newblock A proof that deep artificial neural networks overcome the curse of
  dimensionality in the numerical approximation of {K}olmogorov partial
  differential equations with constant diffusion and nonlinear drift
  coefficients.
\newblock \emph{Commun. Math. Sci.}, 19\penalty0 (5):\penalty0 1167--1205,
  2021.

\bibitem[Kallenberg()]{Kallenberg2002}
O.~Kallenberg.
\newblock \emph{{Foundations of Modern Probability}}.
\newblock Probability and Its Applications. Springer New York, second edition.

\bibitem[Kohler et~al.(2010)Kohler, Krzy\.{z}ak, and Todorovic]{Kohler2010}
M.~Kohler, A.~Krzy\.{z}ak, and N.~Todorovic.
\newblock Pricing of high-dimensional {A}merican options by neural networks.
\newblock \emph{Math. Finance}, 20\penalty0 (3):\penalty0 383--410, 2010.

\bibitem[Lapeyre and Lelong(2021)]{LapeyreLelong2021}
B.~Lapeyre and J.~Lelong.
\newblock Neural network regression for bermudan option pricing.
\newblock \emph{Monte Carlo Methods Appl.}, 27\penalty0 (3):\penalty0 227--247,
  2021.

\bibitem[Longstaff and Schwartz(2001)]{LongSchw01}
F.~A. Longstaff and E.~S. Schwartz.
\newblock Valuing {A}merican options by simulation: a simple least-squares
  approach.
\newblock \emph{Rev. Financ. Stud.}, 14\penalty0 (1):\penalty0 113--147, 2001.

\bibitem[Opschoor et~al.(2020)Opschoor, Petersen, and Schwab]{Opschoor2020}
J.~A.~A. Opschoor, P.~C. Petersen, and C.~Schwab.
\newblock Deep {R}e{LU} networks and high-order finite element methods.
\newblock \emph{Anal. Appl. (Singap.)}, 18\penalty0 (5):\penalty0 715--770,
  2020.

\bibitem[Peskir and Shiryaev(2006)]{peskir2006optimal}
G.~Peskir and A.~Shiryaev.
\newblock \emph{Optimal Stopping and Free-Boundary Problems}.
\newblock Number Bd. 10 in Lectures in Mathematics. ETH Z{\"u}rich. Springer
  Verlag NY, 2006.

\bibitem[Petersen and Voigtlaender(2018)]{PETERSEN2018296}
P.~Petersen and F.~Voigtlaender.
\newblock Optimal approximation of piecewise smooth functions using deep {ReLU}
  neural networks.
\newblock \emph{Neural Netw.}, 108:\penalty0 296 -- 330, 2018.

\bibitem[Reisinger and Zhang(2020)]{ReisingerZhang2019}
C.~Reisinger and Y.~Zhang.
\newblock Rectified deep neural networks overcome the curse of dimensionality
  for nonsmooth value functions in zero-sum games of nonlinear stiff systems.
\newblock \emph{Anal. Appl. (Singap.)}, 18\penalty0 (6):\penalty0 951--999,
  2020.

\bibitem[Reppen et~al.(2022)Reppen, Soner, and Tissot-Daguette]{Reppen22}
A.~M. Reppen, H.~M. Soner, and V.~Tissot-Daguette.
\newblock Neural optimal stopping boundary.
\newblock \emph{Preprint, arXiv 2205.04595}, 2022.

\bibitem[Revuz(1984)]{revuz1984markov}
D.~Revuz.
\newblock \emph{Markov Chains}.
\newblock Mathematical Studies. North-Holland, 1984.

\bibitem[Rogers(2002)]{Rogers2002}
L.~C.~G. Rogers.
\newblock Monte {C}arlo valuation of {A}merican options.
\newblock \emph{Math. Finance}, 12\penalty0 (3):\penalty0 271--286, 2002.

\bibitem[Ruf and Wang(2020)]{Ruf2020}
J.~Ruf and W.~Wang.
\newblock Neural networks for option pricing and hedging: a literature review.
\newblock \emph{J. Comput. Finance}, 24\penalty0 (1):\penalty0 1--46, 2020.

\bibitem[Sato(1999)]{Sato1999}
K.-I. Sato.
\newblock \emph{{L\'{e}vy processes and infinitely divisible distributions}}.
\newblock Cambridge University Press, 1999.

\bibitem[Sirignano and Spiliopoulos(2018)]{Sirignano2018}
J.~Sirignano and K.~Spiliopoulos.
\newblock D{GM: A} deep learning algorithm for solving partial differential
  equations.
\newblock \emph{J. Comput. Phys.}, 375:\penalty0 1339--1364, 2018.

\bibitem[Takahashi and Yamada(2021)]{Takahashi2021}
A.~Takahashi and T.~Yamada.
\newblock Asymptotic expansion and deep neural networks overcome the curse of
  dimensionality in the numerical approximation of kolmogorov partial
  differential equations with nonlinear coefficients.
\newblock CIRJE F-Series CIRJE-F-1167, CIRJE, Faculty of Economics, University
  of Tokyo, 2021.

\bibitem[Tsitsiklis and Van~Roy(2001)]{Tsitsiklis2001Regression}
J.~Tsitsiklis and B.~Van~Roy.
\newblock Regression methods for pricing complex american-style options.
\newblock \emph{IEEE Trans. Neural Netw. Learn. Syst.}, 12\penalty0
  (4):\penalty0 694--703, 2001.

\bibitem[Wang and Perdikaris(2021)]{Wang2021}
S.~Wang and P.~Perdikaris.
\newblock Deep learning of free boundary and {S}tefan problems.
\newblock \emph{J. Comput. Phys.}, 428:\penalty0 Paper No. 109914, 24, 2021.

\bibitem[Yarotsky(2017)]{Yarotsky2017}
D.~Yarotsky.
\newblock {Error bounds for approximations with deep ReLU networks}.
\newblock \emph{Neural Netw.}, 94:\penalty0 103--114, 2017.

\end{thebibliography}
}
\end{document}